\documentclass[a4paper,11pt]{amsart}
\usepackage{amsmath,  amsfonts, amssymb, amsthm, amscd}
\usepackage{graphicx}
\usepackage{enumerate}
\usepackage{url}
\usepackage{color}
\usepackage[utf8]{inputenc}
\usepackage{tikz}
\usetikzlibrary{arrows}
\usepackage[a4paper]{geometry}
\usepackage[colorlinks=true,linkcolor=blue,citecolor=magenta]{hyperref}
\usepackage{epsfig, enumerate}
\usepackage{graphicx}
\usepackage{enumerate}
\usepackage{booktabs}
\usepackage{dsfont}
\usepackage[utf8]{inputenc}
\usepackage[T1]{fontenc}
\usepackage{mathrsfs}
\usepackage{a4wide}
\usetikzlibrary{calc}
\usetikzlibrary{external}
\numberwithin{equation}{section}

\newtheorem{theorem}{Theorem}[section]
\newtheorem{lemma}[theorem]{Lemma}
\newtheorem{proposition}[theorem]{Proposition}
\newtheorem{corollary}[theorem]{Corollary}
\newtheorem{remark}[theorem]{Remark}

\newcommand{\one}{\mathds{1}}
\newcommand{\mc}[1]{{\mathcal #1}}

\newcommand{\bb}[1]{{\mathbb #1}}

\renewcommand{\epsilon}{\varepsilon}

\newcommand{\lex}{{\text{\rm lex}}}
\newcommand{\ex}{{\text{\rm ex}}}
\newcommand{\rw}{{\text{\rm rw}}}
\newcommand{\coup}{\hspace{0.5pt}{\text{\rm coup}}}
\newcommand{\eps}{\varepsilon}
\def\centerarc[#1](#2)(#3:#4:#5){\draw[#1] ($(#2)+({#5*cos(#3)},{#5*sin(#3)})$) arc (#3:#4:#5);}
\newcommand{\pfrac}[2]{\genfrac{}{}{}{1}{#1}{#2}}

\newcommand{\bbE}{{\ensuremath{\mathbb E}} }

\newcommand{\bbP}{{\ensuremath{\mathbb P}} }

\newcommand{\bbR}{{\ensuremath{\mathbb R}} }

\newcommand{\bbZ}{{\ensuremath{\mathbb Z}} }

\newcommand{\cA}{{\ensuremath{\mathcal A}} }

\newcommand{\cE}{{\ensuremath{\mathcal E}} }

\newcommand{\cL}{{\ensuremath{\mathcal L}} }

\newcommand{\fD}{{\ensuremath{\mathfrak{D}}}}

\newcommand{\gep}{\varepsilon}

\renewcommand{\tilde}{\widetilde}          % wider `tilde'
\DeclareMathSymbol{\leqslant}{\mathalpha}{AMSa}{"36} % nicer `smaller or equal'
\DeclareMathSymbol{\geqslant}{\mathalpha}{AMSa}{"3E} % nicer `larger or equal'
\DeclareMathSymbol{\eset}{\mathalpha}{AMSb}{"3F}     % nicer `emptyset'
\newcommand{\dd}{\text{\rm d}}             % a straight d for differentials
\newcommand{\R}{\mathbb{R}}

\newcommand{\Z}{\mathbb{Z}}
\newcommand{\N}{\mathbb{N}}

\renewcommand{\epsilon}{\varepsilon}
\newenvironment{myenumerate}{
	\renewcommand{\theenumi}{\arabic{enumi}}
	\renewcommand{\labelenumi}{{\rm(\theenumi)}}
	\begin{list}{\labelenumi}
		{
			\setlength{\itemsep}{0.4em}
			\setlength{\topsep}{0.5em}
			\setlength\leftmargin{2.45em}
			\setlength\labelwidth{2.05em}
			\setlength{\labelsep}{0.4em}
			\usecounter{enumi}
		}
	}
	{\end{list}
}

\renewenvironment{enumerate}{
	\begin{myenumerate}}
	{\end{myenumerate}}

\newcommand{\beq}{\begin{equation}}
\newcommand{\eeq}{\end{equation}}
\newcommand{\ba}{\begin{aligned}}
	\newcommand{\ea}{\end{aligned}}

\newcommand{\I}{\mathrm{I}}

\usetikzlibrary{shapes.misc}
\usetikzlibrary{shapes.symbols}
\usetikzlibrary{shapes.geometric}
\usetikzlibrary{snakes}
\usetikzlibrary{decorations}
\usetikzlibrary{decorations.markings}

\tikzstyle{tinydots}=[dash pattern=on \pgflinewidth off 2*\pgflinewidth]

\let\oldtocsection=\tocsection
\let\oldtocsubsection=\tocsubsection
\let\oldtocsubsubsection=\tocsubsubsection
\renewcommand{\tocsection}[2]{\hspace{0em}\oldtocsection{#1}{#2}}
\renewcommand{\tocsubsection}[2]{\hspace{1em}\oldtocsubsection{#1}{#2}}
\renewcommand{\tocsubsubsection}[2]{\hspace{2em}\oldtocsubsubsection{#1}{#2}}
\DeclareRobustCommand{\SkipTocEntry}[5]{}
\begin{document}

	\title[Joint fluctuations for current and occupation time]{Nonequilibrium joint fluctuations for current and occupation time in the symmetric exclusion process}
	\author{Dirk Erhard, Tertuliano Franco and Tiecheng Xu}

	\begin{abstract}
		We  provide a full description for the joint fluctuations of current and occupation time
		in the one-dimensional nonequilibrium simple symmetric exclusion process, furnishing  explicit formulas for the covariances of the limiting Gaussian process. The main novelties consist of
		a proof of the tightness of the nonequilibrium  current  based on new correlation estimates, refined  estimates on the discrete gradient of the transition probabilities of the SSEP, and a nonequilibrium Kipnis-Varadhan Lemma based on a Fourier approach.

	\end{abstract}

	\maketitle

	\tableofcontents

	\allowdisplaybreaks

\section{Introduction}\label{sec:Intro}

The exclusion process is a prototype model in Probability Theory and Statistical Mechanics. Since the seventies, a vast literature about it has been developed, leading to many discoveries and rigorous mathematical descriptions
of sundry physical phenomena.

In plain words, the exclusion process consists of independent continuous time random walks on a graph under the additional rule
that two walkers cannot simultaneously occupy the same site, which is called the \textit{exclusion rule}. We say that the exclusion process is  \textit{symmetric}, if the walkers are symmetric random walks, and we say that the exclusion process is \textit{simple}, if jumps are allowed only to nearest neighbour sites on the graph. Many relevant variations of the exclusion process have been studied; in this paper, we present new results on one of the most standard versions of the exclusion process: the  symmetric simple exclusion process (SSEP) on $\bb Z$.

The density fluctuations, that is, the central limit theorem for the spatial density of particles of the SSEP  (first obtained by \cite{Galves} and \cite{Ravi1992}) is nowadays a standard result. On the other hand, observables of the exclusion process, such as the current and the occupation time still attract attention. These two observables for the exclusion process have clear physical interpretations: the current corresponds to the total net of particles crossing an edge, while the occupation times measures the portion of time a certain site has been occupied.

Below we will give a short summary of the literature. Precise results are deferred to Section~\ref{sec:model}. In what follows, by \textit{equilibrium setting} we mean that the initial measure of the one-dimensional SSEP on $\bb Z$ is the Bernoulli product measure of constant parameter, which is invariant and even reversible for this process. By \textit{nonequilibrium setting} we mean that the initial measure is a slowly varying Bernoulli product measure whose parameter is associated to a profile $\rho_0:\bb R\to [0,1]$ whose smoothness assumptions may change from one work to another. Unless mentioned otherwise, we always refer to the one-dimensional SSEP. Due to the total ordering in dimension one, current and tagged particle position are very close concepts.

In 1983, in the equilibrium setting,  Arratia \cite{Arratia} proved that the limiting variance of a tagged particle for the one-dimensional SSEP for a fixed time $t>0$ matches that one of a $1/4$~fractional Brownian motion (fBM), i.e., it is of order $\sqrt{t}$. A similar result for a related model of one-dimensional interacting diffusions was shortly afterwards also shown by Rost and Vares \cite{Rost_Vares}. However, to the best of our knowledge, the first result showing a variance of order $\sqrt{t}$ on the variance on a tagged particle (in the symmetric scenario) was first proven by Harris  \cite{Harris} about a model of one-dimensional colliding Brownian motions.

In 1991, Spohn  \cite[Conjecture 6.5, page 294]{Spohn} conjectured   that the equilibrium fluctuations of the current are given in the limit by a $1/4$ fractional Brownian Motion (fBM).

In 2002,  in the equilibrium setting,  De Masi and Ferrari \cite{Masi_Ferrari} showed  convergence in distribution of the current fluctuations of the SSEP for a fixed time $t>0$ towards a centred  Gaussian whose variance is $\sqrt{t}$ times a constant. This essentially implies convergence in the sense of finite-dimensional distributions towards a $1/4$-fBM.
Alternative proofs of this statement on convergence in finite-dimensional distributions of the current can be found in Jara/Landim \cite{jaralandim2006} by taking the initial profile as constant therein (i.e.\ letting the system start from equilibrium). It is also a particular case of  Franco/Gonçalves/Neumann \cite[Theorem~2.8]{fgn3} by taking $\beta=0$ therein.

In 2008, in the equilibrium setting, Peligrad and Sethumaran \cite{SP} showed tightness of the current. This finally extended the convergence of the current  towards a $1/4$-fBM to be pathwise, and confirmed the conjecture of Spohn from 1991. %that is, in the $J_1$-Skorohod topology of $\mathcal{D}$.

In 2006, Jara and Landim \cite{jaralandim2006} studied the nonequilibrium fluctuations of the current. The authors proved that the current of particles converges, in the sense of finite-dimensional distributions, and they provided explicit formulas for the covariances.

In 2000, in the equilibrium setting, Sethuraman \cite{sunder} proved that the limiting fluctuations of the occupation time in the uniform topology is given by a $3/4$-fBM. See also \cite{CPA} by Jara and Gonçalves on a similar subject.

In this work we deal with the nonequilibrium joint fluctuations of current and occupation times, i.e., the central limit theorem for the \textit{pair current and occupation time}. There are three main tools %and one main technical estimate
we have developed to do so, which have importance \textit{per se} and they are, by far, the most challenging results here.

First, based on Fourier techniques we prove a nonequilibrium Kipnis-Varadhan inequality, which extends \cite[Proposition~6.1, page 333]{kl}. Actually, a proper nomenclature for it maybe should be a \textit{non-stationary} Kipnis-Varadhan inequality, since a nonequilibrium (that is, the scenario where the system starts from an invariant but non-reversible measure) Kipnis-Varadhan inequality has been already obtained in \cite{Chang_Landim_Olla}. Anyway, we keep the \textit{nonequilibrium} terminology of \cite{jaralandim2006} (and others) along the paper.

Second,  through a careful understanding and connection of \cite{EH22}, \cite{FPSV_I} and \cite{Landim05}, we  prove a precise multiple point space-time correlations estimate for the nonequilibrium SSEP which leads to moment estimates and ultimately implies tightness of the nonequilibrium current. We highlight that the correlation estimates we obtain generalize and improve the corresponding estimates of \cite{FPSV_I}, \cite{jaralandim2006}, and~\cite{Landim05}.

Third, to obtain the correlation estimates above, we derive a bound on the discrete partial derivatives of the transition probability of $n$ exclusion particles. By \cite{Landim05} the transition probability of $n$ exclusion particles is essentially given by an $n$-dimensional Gaussian whose covariance matrix has only zeros on the off-diagonal. Our bound on the discrete derivative is almost sharp in the sense that the discrete partial derivatives of the transition probability of $n$ exclusion particles almost coincide with the partial derivatives of the aforementioned Gaussian density.
Our result is thus a type of refined local central limit theorem. Besides the works on the random conductance model such as \cite{Andres}, we are not aware of other works that obtain a local central limit theorem on the level of discrete derivatives (except for random walks).

In possession of the aforementioned tools, we prove the nonequilibrium joint fluctuations for the pair current and occupation time, exhibiting an explicit formula for covariances of the limiting Gaussian process. The Kipnis-Varadhan type inequality is used in the characterization of the limit, and the tightness of the occupation time and the correlations estimates are used in the proof of tightness of the current.
Note that the latter is a highly non-trivial technical issue as evidenced by the fact that even in the equilibrium setting, where many magical identities are valid, it took a long road to establish tightness.

As a curious corollary for the equilibrium setting, we infer that current and occupation time, whose distributions are given by $1/4$-fBM and $3/4$-fBM, respectively,  at any two macroscopic fixed points $u_1,u_2\in \bb R$ and at a same time $t>0$, are independent, while the corresponding processes are not. An intuition of why this happens is presented in Section~\ref{sec:model}.

The paper is organized as follows. In Section~\ref{sec:model} we state the results. In Section~\ref{sec:3} we prove the nonequilibrium type Kipnis-Varadhan inequality. In Section~\ref{sec_4} we deal with the space-time correlations estimate and the bound on the discrete partial derivatives of the transition probability of $n$ exclusion particles. Finally, in Section~\ref{sec_5} we prove the nonequilibrium joint central limit theorem for the pair current and occupation time.

\section{Statements}\label{sec:model}
\subsection{Definitions and previous results}
The one-dimensional symmetric simple exclusion process (SSEP) is a celebrated  Markov process
$\{\eta_t\,: \, t \geq 0\}$ whose dynamics can be entirely defined by its infinitesimal
generator $\mathcal{L}$. The latter is defined on local functions $f: \Omega \to \bbR$, where the state space is given by $\Omega = \{0,1\}^{\bbZ}$ via
\begin{equation}\label{eq2.1}
\mathcal{L} f(\eta) \;=\;\sum_{x\in \bb Z} \{f(\eta^{x,x+1})-f(\eta)\}\,,
\end{equation}
where
\[
\eta^{x,y}(z)\;=\;
\left\{
\begin{array}{rl}
\eta(y)\,,& z=x,\\
\eta(x)\,,& z=y,\\
\eta(z)\,,& z\neq x,y.\\
\end{array}
\right.
\]
We fix once and for all a finite time horizon $T>0$. We are interested in the evolution of this process on the diffusive time scale, therefore we denote by $\{\eta_{t}: t\in [0,T]\}$ the Markov process on $\Omega$ associated to the generator $\mathcal{L}_n = n^2 \mathcal{L}$, where $n\in\bb N$ is a parameter which tends to infinity. The space of trajectories which are right-continuous, with  left-limits and taking values in $\Omega$ is denoted by $\mc D([0,T],\Omega)$. For any initial probability measure $\mu$ on $\Omega$, we denote by $\bb P_{\mu}$ the probability measure on  $\mc D([0,T],\Omega)$ induced by $\mu$ and the Markov process $\{\eta_{t}: t \in [0,T]\}$.

Denote by $\nu_\rho$  the Bernoulli product measure of parameter $\rho\in(0,1)$, which constitutes a family of reversible measures for the SSEP. In the sequel, we recall the \textit{Harris graphical construction} for the SSEP. Here, space is drawn sideways, time is
drawn upwards, and to each edge $e = \{v,w\}$ connecting two neighbouring sites $v$ and $w$ in $\bb Z$ one attaches a Poisson process $N(e)$ with intensity~$1$. Each event of
such a Poisson process is drawn as a link above $e$. The configuration at
time $t$ is then obtained
from the one at time~$0$ by transporting the local states along paths that move
upwards with time and sidewards along the links (see Figure~\ref{fig1}).  The collection of random walkers that move according to these rules is called the stirring process. We refer the reader to \cite{Harris} for more details on that subject, and to \cite{Liggett} for the construction via semigroups and generators.
\begin{figure}[!htb]
	\begin{center}
		\begin{tikzpicture}[scale=1,baseline=0.85cm]
		\draw[->,thick] (0,0) -- (8,0);
		\draw[black!25] (0,4) -- (8,4);
		\foreach \x in {1,...,7}
		{
			\draw (\x,0) -- (\x,4.5);
		}
		\draw[very thick,red!] (4,0) -- (4,1) -- (5,1) -- (5,2.5) -- (4,2.5) -- (4,3.2) -- (3,3.2) -- (3,4);
		\foreach \x/\y in {1/2,1/3,2/1.1,3/2,3/3.2,4/2.5,4/1,5/0.5,6/1.8,6/2.2,6/3.7}
		{
			\draw[tinydots] (\x,\y) -- (\x+1,\y);
		}

		\node[fill=white,below] at (4,0) {$x$};
		\node[fill=white,above] at (3,4) {$y$};
		\node[left] at (0,0) {$0$};
		\node[left] at (0,4) {$t$};
		\node[right] at (8,0) {$\bb Z$};
		\node[inner sep=0pt,minimum size=1mm] at (4,0) [circle, fill=black!] {};
		\node[inner sep=0pt,minimum size=1mm] at (3,4) [circle, fill=black!] {};
		\end{tikzpicture}
	\end{center}
	\caption{Graphical representation. The dashed lines are links
		and the thick red line shows a path from $(x,0)$ to $(y,t)$ following the links.}
	\label{fig1}
\end{figure}

Let   $\rho_0:\bb R\to[0,1]$ be a $C^2$ profile with bounded derivatives and bounded away from zero and from one.
Let $\{\nu_{\rho_0^n(\cdot)}\}_{n\in \bb N}$ be the collection of slowly varying Bernoulli product measures associated to $\rho_0$, i.e., $\nu_{\rho_0^n(\cdot)}$ is the product measure on $\{0,1\}^{\bb Z}$ such that
\begin{align*}
\nu_{\rho_0^n(\cdot)}\big\{\eta\in \{0,1\}^{\bb Z}\,:\, \eta(x)=1\big\}\;=\; \rho_0\big(\pfrac{x}{n})\,.
\end{align*}

Moreover, let $\rho^n_t(x) = \bb E_{\nu_{\rho_0^n}(\cdot)}[\eta_t(x)]$ be the average occupation of site $x$ at time $t$, which is the solution of
the discrete heat equation
\begin{equation*}
\begin{cases}
\partial_t \rho_t^n(x) = \Delta_n \rho_t^n(x)\,,\quad x\in \bb Z \text{ and } t>0\\
\rho_0^n(x) = \rho_0(x/n)\,, \quad x\in \bb Z\,.
\end{cases}
\end{equation*}
Here $\Delta_n$ is the usual discrete Laplace operator accelerated by $n^2$, i.e.,
\begin{equation*}
\Delta_n f(x)\;=\; n^2\sum_{y:\, x\sim y}\left[f(y)-f(x)\right]\,.
\end{equation*}
The centred occupation time $\Gamma^n_x(t)$ of a site $x\in \bb Z$ is defined by
\[
\Gamma^n_x(t)\;=\; \int_0^t \dd s\, \overline{\eta}_{s}(x)\,.
\]

We now detail the previously known results that are most related to the current work and which were already announced in Section~\ref{sec:Intro}. More precisely, we state the results about the fluctuations of the current in equilibrium and in nonequilibrium and about the occupation time in equilibrium.
\begin{theorem}[Equilibrium fluctuations for the occupation time \cite{sunder}]
	Assume that $\eta$ starts from equilibrium,  i.e., $\rho_0\in (0,1)$ is a constant function. In this case, the sequence of processes $\{\sqrt{n}\,\Gamma^n_x(t): t\in [0,T]\}_{n\in \bb N}$ converges in distribution with respect to the $J_1$-Skorohod topology  to a fBM of Hurst exponent~$3/4$.
\end{theorem}
Denote by
$J^n_{x,x+1}(t)$ the current of particles over the bond $\{x,x+1\}$, which is defined as  the total number of jumps from the
site $x$ to the site $x+1$ minus
the total number of jumps from the site $x+1$ to the site $x$ up to time $t$.

Let  $\overline{J}_{x,x+1}^n(t) = J_{x,x+1}^n(t) - \bb E_{\nu_{\rho_0^n(\cdot)}}[J_{x,x+1}^n(t)]$ be the centred current of particles.
As commented in the introduction, the equilibrium current fluctuations in the exclusion process has been an object of intensive studies of many authors, among them \cite{Arratia,Masi_Ferrari,Harris, Rost_Vares}.
\begin{theorem}[Equilibrium fluctuations for the current by many authors]\label{thm_equ_curr}\quad  Assume that $\eta$ starts in equilibrium, i.e., $\rho_0$ is a constant function. Then, the sequence of normalized currents of particles $\big\{J^n_{0,1}/\sqrt{n}:t\in [0,T]\big\}_{n\in \bb N}$ converges in the sense of finite-dimensional distributions  to a fBM of Hurst parameter $1/4$.
\end{theorem}
Note that in  Theorem~\ref{thm_equ_curr}  the convergence is in the sense of  finite-dimensional distributions since tightness was not available. This is due to the fact that it is not sufficient to control second moments of the process. Rather than second moments, moments of sixth order must be controlled in order to apply a  Kolmogorov-Centsov criterion. This issue was only solved in~\cite{SP} by means of a clever decomposition of the current as a sum of independent random variables making use of the stirring process.
\begin{theorem}[Tightness of equilibrium current fluctuations \cite{SP}]\label{prop_tight} Assume that $\eta$ starts in equilibrium, i.e., $\rho_0$ is a constant function. The sequence of normalized currents of particles $\big\{J_{0,1}^n/\sqrt{n}:t\in [0,T]\big\}_{n\in \bb N}$ is tight with respect to the uniform topology.
\end{theorem}
Thus, \cite{SP} allows to improve the convergence of Theorem~\ref{thm_equ_curr} to be pathwise.
Let $\rho_t(u) = \rho(t,u)$ be the solution of the heat equation starting from $\rho_0(\cdot)$, i.e.,
\begin{equation*}
\begin{cases}
\partial_t \rho_t(u) = \Delta \rho_t(u)\,, \quad u\in \bb R \text{ and } t>0\\
\rho_0(u) = \rho_0(u)\,, \quad u\in \bb R
\end{cases}
\end{equation*}
and let $\mc{X}(x)= x(1-x)$.
The out of equilibrium current fluctuations
have been obtained in:
\begin{theorem}[Out of equilibrium current fluctuations \cite{jaralandim2006}]\label{thm_jaralandim}
	Fix $u\in \bb R$. The sequence of\break normalized centered currents   $\big\{\overline{J^n}_{\lfloor un\rfloor,\lfloor un\rfloor +1}/\sqrt{n}:t\in [0,T]\big\}_{n\in \bb N}$
	converges in the sense of finite-dimensional distributions  to a Gaussian process $\{Z_t:t\in[0,T]\}$ whose covariances are given by
	\begin{align*}
	\bb E\big[Z_sZ_t\big]\;&=\;  \int^{\infty}_0\dd v\; \mc{X}(\rho_0(v-u))  P[B_s\geq v]P[B_t\geq v] \\
	&+ \int_{-\infty}^0 \dd v\; \mc{X}(\rho_0(v-u))  P[B_s\leq v]P[B_t\leq v]  \\
	&+ 2 \int_0^s \dd r\int_{-\infty}^{+\infty} \dd v \; \mc{X}(\rho(r,v-u)) \, p_{t-r}(0,v)\, p_{s-r}(0,v)
	\end{align*}
	provided $s\leq t$, where $B_t$ is a standard Brownian motion starting from origin and $p_t(u,v)$ is the Gaussian kernel.
\end{theorem}

\subsection{New results}
We start this section with two results that are of technical nature but are interesting on its own rights.
\begin{theorem}[An out of equilibrium Kipnis-Varadhan type inequality]\label{thm:KipnisVaradhan}
	Let $T>0$ and $\eps\leq 1$ and define
	\begin{equation*}
	g(s,\eta) \;=\; \eta(0)- \rho_s^n(0) -\frac{1}{\ell}\sum_{y=0}^{\ell-1}\big[\eta(y)-\rho_s^n(y)\big]\,.
	\end{equation*}
	Assume that the initial profile $\rho_0$ is globally Lipschitz and is bounded away from zero and one. Then, for any $0 \leq s< t\leq T$, taking $\ell=\eps n$,
	\begin{equation*}
	\bbE_{\nu_{\rho_0^n(\cdot)}}\Big[\Big(\int_s^t  \dd r\,g(r,\eta_{r})\Big)^2\Big] \;\lesssim\; \frac{(t-s)\eps^{3/4}}{n}\big(1+(t-s)^{1/4} + (t-s)\big)\,,
	\end{equation*}
	where the proportionality constant depends on $T$ and on nothing else.
\end{theorem}
\begin{remark}\rm
	The above result is stated in the case of the simple symmetric exclusion process, but we feel that our approach should work for the weakly asymmetric exclusion process with an asymmetry of order $1/n$, and the mean-zero exclusion process with moments of order four. Indeed, as we will see the proof makes use of known correlation estimates for the simple symmetric  exclusion process and of the local central limit theorem. Both tools are available for the latter two processes. We refer to Remark~\ref{rem:otherprocesses} for a more detailed comment.
\end{remark}

Next we state the precise space time correlation estimates, which generalizes and improves the results of \cite{FPSV_I}, \cite{jaralandim2006} and \cite{Landim05}. To do so we however need to introduce some more notation.
Given a list of points $(x_i:i\in I)$ where $I$ is some index set, we say that $x_i$ is a \emph{repetitive point} in $(x_i: i\in I)$ if there exists $j\in I$ with $j\neq i$ such that $x_i=x_j$. Otherwise $x_i$ is called a \emph{non-repetitive point}. Whenever all points in $(x_i:i\in I)$ are distinct, we call it a list of non-repetitive points.

\begin{theorem}\label{gnE}
	Fix $m\in \bb N$ and $k_1, k_2, \ldots, k_m\in \bb N$. Consider $m$ lists of non-repetitive points $(x_{i_1}: 1\leq i_1\leq k_1), \dots, (x_{i_m}: 1\leq i_m\leq k_m)$ and let $0\leq t_1<\cdots <t_m\leq T$. Then, there exists a constant $C$  independent of all the points $(x_{i_j}: 1\leq i_j\leq k_j)$, with $1\leq j\leq m$, such that
	\begin{equation*}
	\Big\lvert\bb E_{\nu_{\rho_0^n(\cdot)}}\Big[\prod_{j=1}^{m} \prod_{i_j=1}^{k_j} \overline{\eta}_{t_j}(x_{i_j})\Big]\Big\rvert \;\leq\; Cn^{-\sum_{j=1}^m k_j/2} \prod_{j=1}^{m-1} \bigg(\frac{n}{\sqrt{n^2(t_{j+1}-t_j)+1}}\bigg)^{k_j\wedge \sum_{l=j+1}^m k_l}
	\end{equation*}
	for every integer $n\geq 1$.
	In particular, if $k_j=1$ for every $1\leq j\leq m$ or, in other words, every collection of points consists of a single point, say $x_j$, then we have
	\begin{equation}\label{diftime}
	\Big\lvert\bb E_{\nu_{\rho_0^n(\cdot)}}\Big[\prod_{j=1}^{m}  \overline{\eta}_{t_j}(x_j)\Big]\Big\rvert \;\leq\; Cn^{-m/2} \prod_{j=1}^{m-1} \frac{n}{\sqrt{n^2(t_{j+1}-t_j)+1}}
	\end{equation}
	for every integer $n\geq 1$.
\end{theorem}

We now present  the discrete gradient estimate of the transition probability of the labelled exclusion process.To that end we introduce more notation.
Fix $k\in \bb N$ and denote by $\Lambda^k$ the set of non-repetitive points with $k$ coordinates:
$$\Lambda^k\;:=\;\big\{\mathbf x=(x_1,\dots,x_k)\in\bb Z^k: x_i\neq x_j, \,\, \forall\,\, 1\leq i,j\leq k\big\}\,.$$
We consider $k$ labelled exclusion particles that evolve through stirring and are accelerated by $n^2$. For $\mathbf x= (x_1, \ldots, x_k), \mathbf y\in (y_1, \ldots, y_k)\in \Lambda^k$ we denote by $p_t^{\lex}(\bf x, \bf y)$ the probability that particles $i\in\{1,\ldots,k\}$ go from $x_i$ to $y_i$ in time $t$. We then have the following result.
\begin{theorem}\label{grad}
	There exists a constant $C>0$ such that for every $t>0$, every $\mathbf x,\mathbf y\in \Lambda^k$, every $1\leq i\leq k$ such that $\mathbf x+e_i\in\Lambda^k$,
	$$\big| p_t^{\lex}(\mathbf x,\mathbf y)-p^{\lex}_t(\mathbf x+e_i,\mathbf y)\big|\;\leq\; \frac{C}{(n^2 t+1)^{(k+1)/2}}\,.$$
\end{theorem}
\begin{remark}\rm
	The above result provides diagonal estimates, i.e., the above estimate should be sharp when $\bf x$ and $\bf y$ are close to each other. We will see in Section~\ref{43} that	combining a result from~\cite{Landim05} with an interpolation estimate allows to obtain almost sharp bounds in the non-diagonal case, i.e., when $\bf x$ and $\bf y$ are not close to each other. This will in particular show that the above discrete derivative is almost bounded by the derivative of a $k$ dimensional Gaussian. We refer to Section~\ref{43} for the details.
\end{remark}

Fix $u_1,\ldots,u_k\in \bb R$, which we call \textit{macroscopic sites}, and we call $\lfloor u_1n\rfloor, \ldots,$ $ \lfloor u_kn\rfloor\in \bb Z$ the associated \textit{microscopic sites}.
We state below the joint fluctuations  of currents and
occupation times for $k$ macroscopic sites. To shorten notation, we sometimes write  $\mathcal X_r := \mc{X}(\rho_r(u))=  \mc{X}(\rho(u, r)) = \rho(u, r) (1-\rho(u, r))$.
\begin{theorem}\label{thm2.6}
	The sequence of processes
	\begin{align*}
	\Big\{\big(n^{-1/2}\overline{J^n}_{\lfloor u_1n\rfloor, \lfloor u_1 n\rfloor +1}(t),&\ldots, n^{-1/2}\overline{J^n}_{\lfloor u_kn\rfloor, \lfloor u_k n\rfloor +1}(t) ,\\
	& n^{1/2}\Gamma^n_{\lfloor u_1n\rfloor}(t),\ldots, n^{1/2}\Gamma^n_{\lfloor u_1n\rfloor}(t)\big): t\in [0,T]\Big\}_{n\in \bb N}
	\end{align*}
	converges in distribution, with respect to the $J_1$-Skorohod topology, to
	a $2k$-dimensional centred Gaussian process
	\begin{align*}
	\Big\{\big(J_{u_1}(t),\ldots, J_{u_k}(t) , \Gamma_{u_1}(t),\ldots, \Gamma_{u_k}(t)\big): t\in [0,T]\Big\}
	\end{align*}
	whose covariances are as follows.  For $s,t>0$ and
	$u_1\leq u_2\in \mathbb R$,
	\begin{align*}\label{JJ}
	\bb E [J_{u_1}(s)J_{u_2}(t)]\;=\;& \int_{-\infty}^{u_1} \dd u\; \bb P\big[B_s\geq u_1-u\big]\bb P\big[B_t\geq u_2-u\big] \mc{X}(\rho_0(u)) \\
	&-\int_{u_1}^{u_2} \dd u\; \bb P\big[B_s\leq u_1-u\big]\bb P\big[B_t\geq u_2-u\big] \mc{X}(\rho_0(u)) \\
	&+\int_{u_2}^\infty \dd u\; \bb P\big[B_s\leq u_1-u\big]\bb P\big[B_t\leq u_2-u\big] \mc{X}(\rho_0(u)) \\
	&+2\int_0^{s\wedge t} \dd r \int_{\bb R} \dd u\; p_{s-r}(u,u_1)p_{t-r}(u,u_2) \mc{X}(\rho_r(u))\, ,\text{ and }\\
	\bb E [\Gamma_{u_1}(s)\Gamma_{u_2}(t)]  \;=\;
	&\int_0^s \dd r_1\int_0^t  \dd r_2 \,\,p_{\lvert r_1-r_2\rvert}(u_1,u_2) \mathcal X(\rho_{r_2}(u_2))\\
	&+ 2\int_0^s \dd r_1\int_0^t  \dd r_2\int_0^{r_1\wedge r_2}\dd\tau \int_{\bb R}\dd u\; p_{r_1-\tau}(u,u_1) p_{r_2-\tau}(u,u_2) \{\partial_\tau \mathcal X_\tau -\Delta \mathcal X_\tau \} \,
	\end{align*}
	and, for any $s,t>0$ and $u_1,u_2\in\bb R$,
	\begin{equation}\label{cov_gamma_current}
	\begin{split}
	\bb E [\Gamma_{u_1}(s)J_{u_2}(t)]\;=\; &\int_0^t \!\dd r\int_{u_1}^\infty\! \dd u\; p_{r-s}(u,u_2)\one\{r>s\} \mathcal X_s\\
	&+ \int_0^t \dd r \, \bb P_0[B_{s-r}\geq u_1-u_2] \one\{r<s\} \mathcal X(\rho_s(u_2))\\
	&+\int_0^t \dd r\int_0^{s\wedge r}\!\! \dd\tau \int_{\bb R}\dd u \;\bb P_0[B_s\geq u_1-u] p_{r-\tau}(u,u_2) \{\partial_\tau \mathcal X_\tau -\Delta \mathcal X_\tau \}\\
	&- \int_0^t \dd r\int_{u_1}^\infty \dd u \,p_r(u,u_2)\, \mc X_0\\
	\end{split}
	\end{equation}
	where $B_t$ is a standard Brownian motion starting from zero and $p_t(u,v)$ is the Gaussian kernel.
\end{theorem}
As a curious immediate corollary of Theorem~\ref{thm2.6}, we get:
\begin{corollary}\label{corollary27}
	Assume that the profile $\rho_0\in (0,1)$ is constant, i.e., the system starts from equilibrium. Then, for each $t>0$ and any $u_1,u_2\in \bb R$, the random variables
	$J_{u_1}(t)$ and $\Gamma_{u_2}(t)$ are independent, where $J_{u_1}(\cdot)$ is a fBM of Hurst exponent  $1/4$ and $\Gamma_{u_2}(\cdot)$ is a fBM of Hurst exponent~$3/4$. However, the processes $J_{u_1}(\cdot)$ and $\Gamma_{u_2}(\cdot)$ are not independent.
\end{corollary}
To deduce the corollary above, note that in equilibrium the first and  third terms in
the right hand side of
\eqref{cov_gamma_current} are identically null, while the second and forth terms cancel each other by doing the change of variables $r\mapsto t-r$ and $u\mapsto u+u_2$ and using the symmetry of the heat kernel.
The fact that the distribution of  $J_{u_1}(\cdot)$ and $\Gamma_{u_2}(\cdot)$ are given by fBMs of parameters $1/4$ and $3/4$ is natural in view of the literature mentioned in Section~\ref{sec:Intro}. On the other hand, the independence in the equilibrium setting (for a same time) is a little surprising.
An intuition of the asymptotic independence of Corollary~\ref{corollary27} relies on the fact that the occupation time is invariant for a time reversion, while the current changes  its sign. So, they must be (asymptotically) uncorrelated, which means independence in the Gaussian scenario.

\section{An out of equilibrium Kipnis-Varadhan Lemma -- Proof of Theorem~\ref{thm:KipnisVaradhan}}\label{sec:3}

Our proof is inspired by~\cite{gj2014}.
For ease of notation we are going to only show  that
\begin{equation}\label{zero_T}
\bbE_{\nu_{\rho_0^n(\cdot)}}\Big[\Big(\int_0^T  \dd s\,g(s,\eta_{s})\Big)^2\Big] \;\lesssim\; \frac{T\eps^{3/4}}{n}\big(1+T^{1/4} + T\big)\,.
\end{equation}
The adaptation to the case $s<t$ follows almost straightforwardly.
The proof is structured as follows: first we introduce some notation and some auxiliary identities, then we present the main ideas, and finally we work out the details.

Fix $\gamma>0$. Let $f_\gamma$ be such that
\begin{equation}\label{eq:resolvent}
(\gamma -\cL_n)f_\gamma(s,\cdot)\;=\; g(s,\cdot)
\end{equation}
for all $s\in[0,T]$.
Next, we write
\begin{equation}\label{eq:Dynkin1}
f_\gamma(T,\eta_{T})\;=\; f_\gamma(0,\eta_0) + \int_0^T \dd s\, (\partial_s+\cL_n)f_\gamma(s,\eta_{s})  + M_T^f\,,
\end{equation}
where the last term on the right hand side above is a martingale.
Define the backwards process $\bar\eta$ via $\bar\eta_{t} = \eta_{T-t}$, $0\leq t\leq T$. Note that its initial state  is given by the final state of the original process. In the same way we see that there exists a martingale $\bar M^f$ (with respect to the backwards filtration) such that
\begin{equation*}
f_\gamma(T,\bar\eta_{T}) \;=\; f_\gamma(0,\bar \eta_0) +\int_0^T \dd s\,(\partial_s+\bar\cL_n)f_\gamma(s,\bar \eta_{s})  + \bar M^f_T\,,
\end{equation*}
where $\bar\cL_n$ denotes the generator of $\bar\eta$. As a consequence of the graphical construction we have that $\cL_n=\bar\cL_n$.
Hence, the above equation can be rewritten as
\begin{equation}\label{eq:Dynkin2}
f_\gamma(0,\eta_0) \;=\; f_\gamma(T,\eta_{T}) + \int_0^T \dd s\, (\cL_n f_\gamma)(s,\eta_{s})  -\int_0^T \dd s\, (\partial_s f_\gamma)(s,\eta_{s}) +\bar M_T^f\,.
\end{equation}
Adding up~\eqref{eq:Dynkin1} and~\eqref{eq:Dynkin2} shows that
\begin{equation*}
-2\int_0^T \dd s\, (\cL_n f_\gamma)(s,\eta_{s})\;=\; M_T^f + \bar M_T^f\,.
\end{equation*}
The quadratic variation of $M^f$ is given by the \textit{carr\'e du champ}, i.e.,
\begin{equation}\label{eq:quadraticvar}
\begin{aligned}
\langle M^f, M^f\rangle_T&\;=\;\int_0^T \dd s\,\Big\{ \cL_n f_\gamma^2(s,\eta_{s}) -2 f_\gamma(s,\eta_{s})\cL_n f_\gamma(s,\eta_{s})\Big\}\,\\
&\;=\; \int_0^T \dd s\sum_{x\in\Z} \Big\{f_\gamma(s,\eta_{s}^{x,x+1})-f_\gamma(s,\eta_{s})\Big\}^2\,.
\end{aligned}
\end{equation}
Since $\cL_n = \bar\cL_n$ the same formula holds for $\bar M^f$.
Using~\eqref{eq:resolvent}, the above observations, and the Cauchy-Schwarz inequality, we can deduce that
\begin{align}
&\bbE_{\nu_{\rho_0^n(\cdot)}}\Big[\Big(\int_0^T \dd s\, g(s,\eta_{s})\Big)^2\Big] \notag\\
&\leq\; \gamma^2 \bbE_{\nu_{\rho_0^n(\cdot)}}\Big[2\Big(\int_0^T\dd s\, f_\gamma (s,\eta_{s})\Big)^2\Big] +  2\bbE_{\nu_{\rho_0^n(\cdot)}}\Big[\Big(\int_0^T\dd s\, \cL_n f_\gamma (s,\eta_{s})\Big)^2\Big]\notag\\
&\leq\; 2T\gamma^2\int_0^T \dd s\, \bbE_{\nu_{\rho_0^n(\cdot)}}\big[f_\gamma (s,\eta_{s})^2\big] + 4\int_0^T \dd s\,\sum_x \bbE_{\nu_{\rho_0^n(\cdot)}}\big[(f_\gamma (s,\eta_{s}^{x,x+1})-f_\gamma (s,\eta_{s}))^2\big]\,.\label{eq:KV}
\end{align}

The above suggests the necessity of
$L^2$-type estimates on $f_\gamma$, and to derive those we will work in Fourier space. The advantage of doing so is that it will turn out that the generator of the exclusion process has a very nice representation in Fourier coordinates. In a nutshell, it is possible to write $\cL_n$ as a sum of operators $\mathfrak{D}_{n,k}$ where $k\in\N$ ranges over the number of particles or the ``degree'' of the function $\hat f$ that $\mathfrak{D}_{n,k}$ acts on. In the particular case of functions with degree one, i.e., $k=1$, which is the case for $\hat g$, one has that $\mathfrak{D}_{n,1}= n^2\Delta$, where the latter denotes the discrete Laplacian. This then allows to calculate $f_\gamma$ rather explicitly and known estimates on the spatial correlations of the exclusion process allow to obtain the desired estimates.
\begin{remark}\rm
	\label{rem:otherprocesses}
	The approach adopted here should work whenever the following conditions are satisfied:
	\begin{itemize}
		\item[\rm{1)}] the generator of the process under consideration is of the form $\cL + \cA$, where $\cL$ denotes the symmetric part which is diagonal in Fourier space and $\cA$ denotes the asymmetric part. Moreover, the generator of the backwards process is of the form $\cL - \cA$.
		\item[\rm{2)}] one has a sufficient control on the spatial correlations of the process.
		\item[\rm{3)}] one can use the local central limit theorem.
	\end{itemize}
	The meaning of sufficient should become clear in the sequel.
	In that case the analysis should work in a way analogous to the one employed here. In particular the resolvent equation~\eqref{eq:resolvent} only considers the symmetric part.
\end{remark}
We will now work out the details.
For any finite subset $\Lambda$ of $\Z$, define
\begin{equation*}
\Psi_\Lambda(s,\eta)\;=\; \prod_{x\in\Lambda} \frac{\eta(x)-\rho_s^n(x)}{\sqrt{\rho_s^n(x)(1-\rho_s^n(x))}}\,,
\end{equation*}
where the collection $\{\Psi_\Lambda(s,\cdot):\, \Lambda\subseteq \Z,\, |\Lambda| <\infty\}$ builds an orthonormal basis of $L_0^2(\nu_{\rho_s^n})$, the space of $L^2$-functions with mean zero with respect to the  measure $\nu_{\rho_s^n}= \otimes_{x\in \Z} \mathrm{Ber}(\rho_s^n(x))$. Thus, any $f(s,\cdot)\in L_0^2(\nu_{\rho_s^n})$ can uniquely be written as
\begin{equation*}
f(s,\cdot)\;=\; \sum_{k=1}^{\infty}\sum_{\Lambda: \,|\Lambda|=k} \hat{f}(s,\Lambda)\Psi_\Lambda(s,\cdot)\,.
\end{equation*}
For $k\geq 1$, we denote the space generated by $\{\Psi_\Lambda(s,\cdot):\, |\Lambda|=k\} $ by $L_k^2(\nu_{\rho_s^n})$, and we refer to it sometimes as the $k$-th chaos. We further denote by $\pi_k$ the projection onto $L_k^2(\nu_{\rho_s^n})$. As pointed out by \cite[Equation (5.17)]{KLO}, we have
\begin{equation*}
\cL_n f \;=\; \sum_{k=1}^{\infty}\sum_{|\Lambda|=k}(\mathfrak{D}_{n,k} \hat{f})(\Lambda) \Psi_\Lambda
\end{equation*}
where
\begin{equation*}
(\fD_{n,k} \hat{f})(\Lambda)\;=\; n^2\sum_{x\in \Z} \big[\hat{f}(\Lambda_{x,x+1})-\hat{f}(\Lambda)\big]\,,
\end{equation*}
with
\begin{equation*}
\Lambda_{x,x+1}=
\begin{aligned}
\begin{cases}
(\Lambda\setminus\{x\})\cup\{x+1\}, &\text{if }x\in \Lambda, x+1\notin \Lambda\,,\\
(\Lambda\setminus\{x+1\})\cup\{x\}, &\text{if } x\notin \Lambda, x+1\in\Lambda\,,\\
\Lambda, &\text{otherwise.}
\end{cases}
\end{aligned}
\end{equation*}
In the case $|\Lambda|=1$, we can identify $\fD_{n,1}$ with the discrete Laplacian accelerated by $n^2$. We write $n^2\Delta$ instead of $\fD_{n,1}$ in that case.
Using that $\pi_k g=0$ for all $k\geq 2$ we see from the above considerations that the resolvent equation~\eqref{eq:resolvent} is equivalent to the following system of equations
\begin{equation*}	\begin{aligned}
\begin{cases}
(\gamma-n^2\Delta)\hat{f}_1(s,\cdot) = \hat{g}_1(s,\cdot)\,,\\
(\gamma-\fD_{n,2})\hat{f}_k(s,\cdot)= 0\, \quad\text{for } k\geq 2\,,
\end{cases}
\end{aligned}
\end{equation*}
where $\hat{f}_k$ denotes the coefficient of $f_\gamma$ in the $k$-th chaos. Since $\gamma$ is in the resolvent set of $\fD_{n,k}$, we in particular see that $\hat{f}_k= 0$ for all $k\geq 2$. Inverting $(\gamma-n^2\Delta)$ we see that
\begin{equation*}
\hat{f}_1(s,x)\;=\; \sum_{y\in\Z}\int_0^\infty \dd t\, e^{-t\gamma} p_{t}(x,y)\hat{g}_1(s,y)\,,
\end{equation*}
where we identified the set $\{x\}$ with $x$, and $p$ denotes the transition probability of a continuous time simple random walk whose waiting time parameter is $2n^2$. It now only remains to calculate $\hat g_1$. To that end we need to solve
\begin{equation*}
g(s,\eta) \;=\;\eta(0)- \rho_s^n(0) -\frac{1}{\ell}\sum_{y=0}^{\ell-1}\big[\eta(y)-\rho_s^n(y)\big] \;=\;\sum_y \hat g_1(s,y) \Psi_{\{y\}}(s,\eta)\,,
\end{equation*}
which leads us to
\begin{equation*}
\hat g_1(s,y) \;=\;\begin{cases}
(1-\frac{1}{\ell})\sqrt{\rho_s^n(0)(1-\rho_s^n(0))}\,, &y=0\,,\\
-\frac{1}{\ell}\sqrt{\rho_s^n(y)(1-\rho_s^n(y))}\,,& y\in\{1,2\ldots,\ell-1\}\,,\\
0,\, &\text{otherwise.}
\end{cases}
\end{equation*}
We therefore obtain
\begin{equation}\label{eq:fgamma}
\begin{aligned}
f_\gamma(s,\eta_{s})&\;=\; \sum_x \bigg[\int_0^\infty \dd t\, e^{-t\gamma}\Big(p_{t}(x,0)\sqrt{\rho_{s}^n(0)(1-\rho_{s}^n(0))}\\
&\hspace{2.5cm}-\frac{1}{\ell}\sum_{y=0}^{\ell -1} p_{t}(x,y)\sqrt{\rho_{s}^n(y)(1-\rho_{s}^n(y))}\Big)\bigg] \times\Psi_{\{x\}}(s,\eta_{s})\,.
\end{aligned}
\end{equation}
With the above representation of $f_\gamma$ at hand we are now finally able to estimate the two terms on the right hand side of~\eqref{eq:KV}. We begin with the first one.
To that end we will analyse the term inside the round brackets in~\eqref{eq:fgamma} which we write as $A_1(s,t,x) + A_2(s,t,x)$, where
\begin{align}
&A_1(s,t,x)\;=\; \Big(p_{t}(x,0)-\frac{1}{\ell}\sum_{y=0}^{\ell-1}p_{t}(x,y)\Big)\sqrt{\rho_{s}^n(0)(1-\rho_{s}^n(0))}\,,\quad \text{ and }\label{linha_1}\\
&A_2(s,t,x)\;=\;\frac{1}{\ell}\sum_{y=0}^{\ell-1}p_{t}(x,y)\Big(\sqrt{\rho_{s}^n(0)(1-\rho_{s}^n(0))}-\sqrt{\rho_{s}^n(y)(1-\rho_{s}^n(y))}\Big)\,.\label{linha_2}
\end{align}
Hence,
\begin{equation}\label{eq:fgammasquared}
\begin{aligned}
&	f_\gamma(s,\eta_{s})^2 \\&\lesssim\; \Big(\sum_x \!\int_0^\infty\!\! \dd t\,e^{-\gamma t} A_1(s,t,x)\, \Psi_{\{x\}}(s,\eta_{s})\Big)^2 +
\Big(\sum_x\! \int_0^\infty\!\!\dd t\, e^{-\gamma t}A_2(s,t,x)\Psi_{\{x\}}(s,\eta_{s})\Big)^2\\
&\overset{\text{def}}{=} \;\cA_1(s,n) + \cA_2(s,n)\,.
\end{aligned}
\end{equation}
\begin{lemma}\label{lem:A1}
	Let $T>0$. The estimate
	\begin{equation*}
	\bbE_{\nu_{\rho_0^n(\cdot)}}[\cA_1(s,n)]\;\lesssim\; \frac{\eps^{3/4}}{n}\Big(1+\frac{1}{\sqrt{\gamma}} + \frac{1}{\gamma^2}\Big)
	\end{equation*}
	holds uniformly in $n$ for all $0\leq s\leq T$,
	where the proportionality constant depends only on $T$.
\end{lemma}
\begin{proof}
	Note that $\cA_1(s,n)$ is a square. Writing  it as the product whose summation  indexes are, respectively, $x_1,t_1$ and $x_2,t_2$  induces us to restrict our attention
	to the following three cases: $\{\lvert x_1\rvert,\lvert x_2\rvert \leq n^2 \text{ and } t_1,t_2\geq \sqrt{\eps}\}$,
	$\{ t_1,t_2\leq \sqrt{\eps} \}$, and $\{\lvert x_1\rvert,\lvert x_2\rvert\geq n^2 \text{ and } t_1,t_2\geq \sqrt{\eps} \}\,.$\newline

	\noindent
	\textbf{1st case:} $\lvert x_1\rvert, \lvert x_2\rvert \leq n^2$ and $t_1, t_2\geq \sqrt{\eps}$. That is, we want to estimate
	\begin{equation}\label{eq:1stcase}
	\Big(\sum_{\lvert x\rvert\leq n^2} \int_{\sqrt{\eps}}^\infty \dd t\, e^{-\gamma t} A_1(s,t,x)\, \Psi_{\{x\}}(s,\eta_{s})\Big)^2\,.
	\end{equation}
	By the local central limit theorem given in \cite[Theorem 2.3.6, page 38 and also check the definition of $\overline{p}_n$ on page 22]{Lawler},
	\begin{equation}\label{eq:localclt}
	\begin{aligned}
	\big\vert p_{t}(x,y)-p_{t}(x,0)\big\vert\;\lesssim\;
	\big\vert K_{tn^2}(x,y) - K_{tn^2}(x,0)\big\vert + \frac{|y|}{n^4t^2}\,,
	\end{aligned}
	\end{equation}
	where by $K_t$ we denote the  transition kernel of a Brownian motion at time $2t$. Hence, using that $\rho$ is bounded, we can estimate \eqref{linha_1} as follows:
	\begin{align*}
	|A_1(s,t,x)|&\;\lesssim\; \Big|\frac{1}{\ell}\sum_{y=0}^{\ell-1}(K_{tn^2}(x,y)-K_{tn^2}(x,0))\Big| + \frac{1}{\ell}\frac{1}{n^4t^2}\sum_{y=0}^{\ell-1}y	\\
	&\;=:\; \I(t,x) + \cE(n,t)\,.
	\end{align*}
	It is therefore  sufficient to estimate
	\begin{equation}\label{eq:1stcaseb}
	\begin{aligned}
	&\bigg(\sum_{\lvert x\rvert\leq n^2} \int_{\sqrt{\eps}}^\infty \dd t\,e^{-\gamma t} \,\I(t,x)\,\Psi_{\{x\}}(s,\eta_{s})\bigg)^2	+ \bigg(\sum_{\lvert x\rvert\leq n^2} \int_{\sqrt{\eps}}^\infty \dd t \, e^{-\gamma t} \cE(n,t)\, \Psi_{\{x\}}(s,\eta_{s})\bigg)^2\,.
	\end{aligned}
	\end{equation}
	Using that $\ell=\eps n$ we see that the error term can be bounded from above as
	\begin{equation*}
	\cE(n,t)\;\lesssim\; \frac{\eps}{n^3 t^2}\,.
	\end{equation*}
	Hence, due to the boundedness of $\Psi_{\{x\}}$,
	\begin{equation*}
	\bigg|\sum_{\lvert x\rvert\leq n^2} \int_{\sqrt{\gep}}^\infty \dd t\, e^{-t\gamma}\cE(n,t) \Psi_{\{x\}}(s,\eta_{s})\bigg|\;\lesssim\; \frac{\sqrt{\eps}}{n}\,,
	\end{equation*}
	so that the contribution of the above to~\eqref{eq:1stcaseb} is at most of the order of
	\begin{equation*}
	\frac{\eps}{n^2}\,.
	\end{equation*}
	We turn now to the analysis of $\I(t,x)$.
	Using Taylor's formula, we write
	\begin{equation}\label{eq:Taylor}
	K_{tn^2}(x,y)\;=\; K_{tn^2}(x,0) + \partial_2 K_{tn^2}(x,0)y +\frac12 \partial_2^2 K_{tn^2}(x,\theta_{x,y}) y^2\,,
	\end{equation}
	where $\theta_{x,y}\in [0,y]$. Thus
	\begin{equation}\label{eq:I}
	\I(t,x) \;\leq\; \Big|\frac{1}{\ell}\sum_{y=0}^{\ell-1}\partial_2 K_{n^2 t}(x,0)y \Big| + \Big|\frac{1}{\ell}\sum_{y=0}^{\ell-1}\frac12 \partial_2^2 K_{n^2 t}(x,\theta_{x,y}) y^2\Big|\,.
	\end{equation}
	We first focus on the contributions coming from the first sum above. A direct computation shows that it is bounded from above by
	\begin{equation*}
	\frac{\epsilon}{2}\frac{1}{\sqrt{2\pi}} \frac{1}{n} \frac{1}{t^{3/2}} \frac{|x|}{n}\exp\Big\{-\frac{(x/n)^2}{2t}\Big\} \;=:\; B(t,x)\,.
	\end{equation*}
	Plugging the above into~\eqref{eq:1stcaseb}, using that $\rho$ is bounded from above and below, we see that we need to estimate
	\begin{equation*}
	\sum_{\substack{\lvert x_1\rvert\leq n^2,\\ \lvert x_2\rvert\leq n^2}}\int_{\sqrt{\eps}}^\infty\dd t_1\int_{\sqrt{\eps}}^\infty \dd t_2\, e^{-\gamma(t_1+t_2)}B(t_1, x_1)B(t_2,x_2) \,\bbE\Big[\Psi_{\{x_1\}}(s,\eta_{s})\Psi_{\{x_2\}}(s,\eta_{s})\Big]\,.
	\end{equation*}
	To estimate the above we distinguish between two cases, the diagonal case, i.e., those tuples $(x_1,x_2)\in \Z^2$ such that $x_1=x_2$, and the off-diagonal case which consists of the remaining ones. We start with the former.\newline

	\noindent \textbf{The diagonal case:} Write $x=x_1=x_2$.
	Removing the restrictions $\lvert x\rvert\leq n^2$ and $t_1, t_2\geq \sqrt{\eps}$, and using that $\Psi_{\{\cdot\}}$ is bounded we infer that it suffices to estimate
	\begin{equation}\label{eq:diagonal}
	\sum_{x\in\bb Z}	\int_0^\infty\dd t_1 \int_0^\infty \dd t_2\; e^{-\gamma(t_1+t_2)} \frac{\epsilon^2}{n^2}\frac{1}{t_1^{3/2}}\frac{1}{t_2^{3/2}}\Big(\frac{x}{n}\Big)^2 \exp\Big\{-\frac{(x/n)^2}{2t_1}\Big\}\exp\Big\{-\frac{(x/n)^2}{2t_2}\Big\}\,.
	\end{equation}
	Using a Riemann sum approximation and the Cauchy-Schwarz inequality to sum separately over the two exponentials below, we obtain that
	\begin{equation*}
	\frac{1}{n}\sum_{x\in \bb N} \Big(\frac{x}{n}\Big)^2 \exp\Big\{-\frac{(x/n)^2}{2t_1}\Big\}\exp\Big\{-\frac{(x/n)^2}{2t_2}\Big\}
	\;\lesssim\; t_1^{3/4}t_2^{3/4}\,.
	\end{equation*}
	Using the fact that the primitive of $e^{-\gamma t}/t^{3/4}$ equals $-\Gamma(\frac14, \gamma t)/\gamma^{1/4}$, where $\Gamma$ denotes the gamma function, we see that~\eqref{eq:diagonal} is bounded from above by some proportionality constant times
	\begin{equation*}
	\frac{\epsilon^2}{n}\frac{1}{\sqrt{\gamma}}\,.
	\end{equation*}

	\noindent \textbf{The off-diagonal case:}
	By \cite[Lemma 3.2]{jaralandim2006} we know that
	\begin{equation}\label{eq:cor}
	\sup_{x_1\neq x_2} \bbE\Big[\Psi_{\{x_1\}}(s,\eta_{s})\Psi_{\{x_2\}}(s,\eta_{s})\Big]\;\lesssim\; \frac{\sqrt{s}}{n}\,.
	\end{equation}
	Thus, removing once again the restrictions $\lvert x_1\rvert, \lvert x_2\rvert\leq n^2$ and $t_1, t_2\geq \sqrt{\eps}$, we see that it is sufficient to estimate
	\begin{equation}\label{eq:offdiag}
	\frac{\sqrt{s}}{n}\!\int_0^\infty\!\! \dd t_1\!\int_0^\infty \!\!\dd t_2\;\; e^{-\gamma(t_1+t_2)}\!\sum_{x_1\neq x_2}\!\! \epsilon^2\frac{1}{n^2}\frac{1}{t_1^{3/2}t_2^{3/2}}
	\Big|\frac{x_1}{n}\Big|\Big|\frac{x_2}{n}\Big|
	\exp\Big\{-\frac{(x_1/n)^2}{2t_1}\Big\}\exp\Big\{-\frac{(x_2/n)^2}{2t_2}\Big\}.
	\end{equation}
	As before, a Riemann sum approximation argument shows that
	\begin{equation*}
	\frac1n \sum_x \frac{1}{\sqrt{t_1}}\Big|\frac{x_1}{n}\Big| \exp\Big(-\frac{(x/n)^2}{2t_1}\Big) \;\lesssim\; \sqrt{t_1}\,.
	\end{equation*}
	Hence,~\eqref{eq:offdiag} is bounded from above by some proportionality constant times
	\begin{equation*}
	\frac{\sqrt{s}}{n}\epsilon^2 \Big(\int_0^\infty  \dd t\,e^{-\gamma t}\frac{1}{\sqrt{t}}\Big)^2\;=\; \frac{\sqrt{s}}{n}\epsilon^2 \frac{\Gamma(\frac12,0)^2}{\gamma}\,.
	\end{equation*}
	We now analyse the remainder term in~\eqref{eq:I} following a similar procedure as above.\\

	\noindent \textbf{The diagonal case:}
	It turns out to be handy to use the scaling properties of the heat kernel. To that end we note that for $x\in \bbZ/n$, and $\theta_{x,y}^n\in [0,y/n]$
	\begin{equation}\label{eq:scaling}
	\Big(\partial_2^2 K_{n^2t}\Big)(nx,n\theta_{x,y}^n)\;=\; \frac{1}{n^3}\partial_2^2 K_t(x,\theta_{x,y}^n)\,.
	\end{equation}
	Thus, performing the sum over $x=x_1=x_2$ we see that we need to estimate
	\begin{equation}\label{eq:remainderdiag}
	\int_{\sqrt{\eps}}^\infty \dd t_1\int_{\sqrt{\eps}}^\infty\dd t_2\, e^{-\gamma(t_1+t_2)}\sum_{x\in\bbZ/n} \frac{1}{n^6}\frac{1}{\ell^2}\sum_{y_1, y_2=0}^{\ell-1}|\partial_2^2 K_{t_1}(x,\theta_{x,y}^n)|\,|\partial_2^2 K_{t_2}(x,\theta_{x,y}^n)|y_1^2y_2^2\,.
	\end{equation}
	We focus on
	\begin{equation}\label{eq:sumsecondderivative}
	\frac1n \sum_{x\in\bbZ/n}|\partial_2^2 K_{t_1}(x,\theta_{x,y}^n)|\,|\partial_2^2 K_{t_2}(x,\theta_{x,y}^n)|\,.
	\end{equation}
	The key property that we will use is that by a  scaling argument one has that
	\begin{equation}\label{eq:heatest}
	|\partial_2^2 K_{t_1}(x,\theta_{x,y}^n)|\;\lesssim\; \Big(\sqrt{t_1}+ |x-\theta_{x,y}^n|\Big)^{-3}\,.
	\end{equation}
	To make use of that estimate we split the sum above into two parts.
	Summing first over $\lvert x\rvert\leq \sqrt{\epsilon}$,  we estimate~\eqref{eq:sumsecondderivative} by
	\begin{equation*}
	\frac1n \sum_{x\in\bbZ/n, \lvert x\rvert\leq \sqrt{\epsilon}}\frac{1}{t_1^{3/2}}\frac{1}{t_2^{3/2}}\;\lesssim\; \frac{\sqrt{\epsilon}}{t_1^{3/2}t_2^{3/2}}\,.
	\end{equation*}
	The sum over $\lvert x\rvert\geq \sqrt{\gep}$ can be estimated with the help of~\eqref{eq:heatest} from above by some proportionality constant times
	\begin{equation*}
	\frac1n \sum_{x\in\bbZ/n, \lvert x\rvert\geq \sqrt{\gep}}\frac{1}{|x-\theta_{x,y}^n|^6}
	\;\lesssim\;\frac{1}{n}\sum_{x\in\bbZ/n,\lvert x\rvert\geq \sqrt{\gep}}\frac{1}{\lvert x\rvert^6}\;\lesssim\; \gep^{-5/2}\,.
	\end{equation*}
	Here, we used that $\eps < 1$ to obtain the first inequality, so that $\theta_{x,y}^n\leq \tfrac{\eps}{n} \ll \sqrt{\eps}$.
	Hence, ~\eqref{eq:sumsecondderivative} is bounded from above by some proportionality constant times
	\begin{equation*}
	\frac{\sqrt{\eps}}{t_1^{3/2}t_2^{3/2}} + \eps^{-5/2}\,.
	\end{equation*}
	Since
	\begin{equation}\label{eq:squaresum}
	\frac1\ell \sum_{y=0}^{\ell-1}y^2\;\lesssim\; \ell^2
	\end{equation}
	and $\ell=\eps n$, we deduce that \eqref{eq:remainderdiag} is bounded by some proportionality constant times
	\begin{equation*}
	\int_{\sqrt{\eps}}^\infty\dd t_1 \int_{\sqrt{\eps}}^\infty\dd t_2\, e^{-\gamma(t_1+t_2)}\frac{1}{n^5}\ell^4 (\frac{\sqrt{\eps}}{t_1^{3/2}t_2^{3/2}} + \eps^{-5/2})
	\;\lesssim\; \frac{\eps^{4}}{n} + \frac{\eps^{3/2}}{\gamma^2n}\,.
	\end{equation*}

	\noindent \textbf{The off diagonal case:}
	Making use of~\eqref{eq:cor}, the relation~\eqref{eq:scaling}, and ~\eqref{eq:squaresum} we see that we need to estimate
	\begin{equation*}
	\sqrt{s}\,\frac{\eps^4}{n^3}\sum_{x_1\neq x_2\in\Z/n}
	\int_{\sqrt{\eps}}^\infty \dd t_1\int_{\sqrt{\eps}}^\infty \dd t_2\, e^{-\gamma(t_1+t_2)}\sup_{0\leq y\leq \ell -1}\partial_2^2 K_{t_1}(x_1,\theta_{x_1,y}^n)\partial_2^2 K_{t_2}(x_2,\theta_{x_2,y}^n)\,.
	\end{equation*}
	Making use of~\eqref{eq:heatest} and distinguishing between the different cases $\lvert x_1\rvert\leq \sqrt{\eps}, \vert x_1\vert\geq \sqrt{\eps}$ and the same for $x_2$ as in the previous case we see that the above sum is bounded by some proportionality constant times
	\begin{equation*}
	\frac{\sqrt{s}}{n}\eps^{9/2} + \frac{\sqrt{s}}{n} \frac{\eps^2}{\gamma^2}\,.
	\end{equation*}
	In conclusion we have shown that the contribution of the first case to $\cA_1(s,n)$ is at most a proportionality constant times
	\begin{equation*}
	\frac{\eps}{n^2} + \frac{\epsilon^2}{n}\frac{1}{\sqrt{\gamma}} + \frac{\sqrt{s}}{n}\epsilon^2 \frac{\Gamma(\frac12,0)^2}{\gamma} + \frac{\eps^{4}}{n} + \frac{\eps^{3/2}}{\gamma^2n} + \frac{\sqrt{s}}{n}\eps^{9/2} + \frac{\sqrt{s}}{n} \frac{\eps^2}{\gamma^2}\,.
	\end{equation*}

	\noindent
	\textbf{2nd case:} $t_1, t_2\leq \sqrt{\eps}$.
	In that case we write
	\begin{equation}\label{eq:A12ndcase}
	\begin{aligned}
	& \Big( \sum_x \int_0^{\sqrt{\eps}} \dd t\,e^{-\gamma t} A_1(s,t,x) \Psi_{\{x\}}(s, \eta_{s})\Big)^2\\
	&\lesssim   \Big(\sum_{x} \int_0^{\sqrt{\eps}}\dd t\, e^{-\gamma t}p_{t}(x,0)\Psi_{\{x\}}(s, \eta_{s})\Big)^2
	\\
	&\qquad+ \Big(\sum_{x} \int_0^{\sqrt{\eps}}\dd t\, e^{-\gamma t}\frac1\ell \sum_{y=0}^{\ell-1}p_{t}(x,y)\Psi_{\{x\}}(s, \eta_{s})\Big)^2\,.
	\end{aligned}
	\end{equation}
	We start by dealing with the first term on the right hand side above. As before we distinguish between diagonal and off diagonal terms.\\

	\noindent \textbf{The diagonal case:}  Using that $\Psi$ is bounded we estimate for $t_1, t_2\geq 0$
	\begin{equation*}
	\sum_x p_{t_1}(x,0) p_{t_2}(x,0) \;=\; p_{t_1+t_2}(0,0) \;\lesssim\; \frac{1}{n \sqrt{(t_1+t_2)}}\,.
	\end{equation*}
	Hence,
	\begin{equation*}
	\begin{aligned}
	&\sum_x \int_0^\eps\dd t_1\int_0^{\sqrt{\eps}} \dd t_2 \, e^{-\gamma(t_1+t_2)}p_{t_1}(x,0)\, p_{t_2}(x,0)\,\Psi_{\{x\}}(s,\eta_{s})^2\\
	&\lesssim\;
	\frac1n\int_0^{\sqrt{\eps}} \dd t_1\int_0^{\sqrt{\eps}} \dd t_2\,  \frac{1}{\sqrt{t_1+t_2}} \;\lesssim\; \frac1n \eps^{3/4}\,.
	\end{aligned}
	\end{equation*}

	\noindent \textbf{The off diagonal case:}
	Using~\eqref{eq:cor}, for fixed $t_1$ and $t_2$ we see that it is sufficient to estimate
	\begin{equation*}
	\frac{\sqrt{s}}{n}\sum_{x_1\neq x_2}p_{t_1}(x_1, 0)p_{t_2}(x_2,0)\;\leq\; \frac{\sqrt{s}}{n}\,.
	\end{equation*}
	Integrating the above with respect to $t_1$ and $t_2$ in $[0,\sqrt{\eps}]^2$ yields
	\begin{equation*}
	\frac{\eps \sqrt{s}}{n}\,.
	\end{equation*}
	It now remains to deal with the second term in~\eqref{eq:A12ndcase}.\\

	\noindent \textbf{The diagonal case:}
	We estimate, for fixed $t_1$ and $t_2$,
	\begin{equation}\label{eq:diagsecondcase}
	\frac{1}{\ell^2}\sum_x \sum_{y_1, y_2=0}^{\ell-1} p_{t_1}(x,y_1)p_{ t_2}(x,y_2)
	\;=\; \frac{1}{\ell^2}\sum_{y_1, y_2=0}^{\ell-1} p_{t_1+t_2}(y_1, y_2)\;.
	\end{equation}
	Using that for any $y_1,y_2\in\bb Z$ we have that
	\begin{equation}\label{eq:probatzero}
	p_{t_1+t_2}(y_1,y_2)\;\lesssim\; \frac{1}{n\sqrt{t_1+t_2}}\,,
	\end{equation}
	we bound~\eqref{eq:diagsecondcase} from above by the right hand side of~\eqref{eq:probatzero}.
	Integrating this with respect to $t_1,t_2\in[0,\sqrt{\eps}]$ yields an upper bound of $\tfrac{\eps^{3/4}}{n}$.\newline

	\noindent \textbf{The off diagonal case:}
	Finally, for fixed $t_1$ and $t_2$ and using~\eqref{eq:cor} the expectation of the off diagonal term can be estimated from above by
	\begin{equation*}
	\frac{\sqrt{s}}{n}\frac{1}{\ell^2}\sum_{x_1\neq x_2}\sum_{y_1, y_2=0}^{\ell-1} p_{t_1}(x_1,y_1)p_{t_2}(x_2,y_2)\;\leq\; \frac{\sqrt{s}}{n}\,.
	\end{equation*}
	Thus, after integration the total contribution is $\tfrac{\eps \sqrt{s}}{n}$. Summing up what we have done in this case, the left hand side in~\eqref{eq:A12ndcase} is bounded from above by some proportionality constant times
	\begin{equation*}
	\frac1n \eps^{3/4} + \frac{\eps \sqrt{s}}{n}\,.
	\end{equation*}
	\noindent
	\textbf{3rd case:} $\lvert x_1\rvert, \lvert x_2\rvert\geq n^2, t_1,t_2\geq \sqrt{\eps}$. Using large deviations estimate for the random walk if  $t\leq n/\eps$, and the same estimates as in the second case if $t\geq n/\eps$ and the exponential factor $e^{-\gamma t}$ shows that the contribution to the first term of~\eqref{eq:fgammasquared} is at most some proportionality constant times
	\begin{equation*}
	\frac{n}{\eps} e^{-cn/\eps} + \frac{e^{-\gamma n/\eps}}{\gamma}\;\leq\; \frac{2\eps}{c n} + \frac{\eps}{\gamma^2n}\,,
	\end{equation*}
	for some positive constant $c$. Here, we used the estimates
	\begin{equation*}
	xe^{-cx}\;\leq\; \frac{2}{cx}\quad\text{and}\quad e^{-\gamma x}\;\leq\; \frac{1}{\gamma x}
	\end{equation*}
	which hold for all $x\geq 0$. Summing up the contributions from the three cases above concludes the proof of Lemma~\ref{lem:A1}.
\end{proof}
The estimates on the diagonal terms above imply the following corollary.
\begin{corollary}\label{cor:A1}
	Let $T>0$, then uniformly in $n$ for all $0\leq s\leq T$,
	\begin{equation*}
	\sum_{x\in\Z}\bbE\Big[\Big(\int_0^\infty \dd t\, e^{-\gamma t}A_1(s,t,x) \Big)^2\Big]\;\lesssim\; \frac{\eps^{3/4}}{n}\Big(1+\frac{1}{n\sqrt{\gamma}} + \frac{1}{n\gamma^2}\Big)\,,
	\end{equation*}
	where the proportionality constant only depends on $T$.
\end{corollary}
\begin{lemma}\label{lem:A2}
	Let $T>0$ and assume that the initial profile $\rho_0$ is globally Lipschitz continuous with Lipschitz constant $L>0$. Then,	the following estimate holds uniformly in $n$ for all $0\leq s\leq T$:
	\begin{equation*}
	\bbE[\cA_2(s,n)]\;\lesssim\;  \frac{1}{\gamma^2}\frac{\eps}{n}+ 	\frac{1}{\gamma^2} \frac{\eps^2}{n} \,,
	\end{equation*}
	where the proportionality constant does only depend on $T$.
\end{lemma}

\begin{proof}
	Our first step consists in obtaining an estimate for
	\begin{equation*}
	\sqrt{\rho_{s}^n(0)(1-\rho_{s}^n(0))}-\sqrt{\rho_{s}^n(y)(1-\rho_{s}^n(y))}
	\end{equation*}
	uniformly in $|y|\leq \eps n$.
	To that end, first of all note that
	$\rho_{s}^n(y) = \sum_{x\in \Z} p_{s}(y-x) \rho_0(x/n)$,
	so that
	\begin{equation*}
	\begin{aligned}
	\rho_{s}^n(0)-\rho_{s}^n(y) &\;=\; \sum_{x\in\Z}(p_{s}(-x) - p_{s}(y-x)) \rho_0(x/n)\\
	&\;=\; \sum_{x\in\Z}p_{s}(x)(\rho_0(-x/n)- \rho_0((y-x)/n))\,.
	\end{aligned}
	\end{equation*}
	By the Lipschitz continuity and the fact that $p_{s}(\cdot)$ is a transition probability, we obtain that
	\begin{equation*}
	|\rho_{s}^n(0)-\rho_{s}^n(y)|\;\lesssim\; \frac{|y|}{n}\,.
	\end{equation*}
	Since $\rho$ is bounded we see that we have the same kind of estimate for
	\begin{equation*}
	\big\vert\rho_{s}^n(0)(1-\rho_{s}^n(0))-\rho_{s}^n(y)(1-\rho_{s}^n(y))\big\vert\,.
	\end{equation*}
	Moreover, since $x\mapsto\sqrt{x}$ is globally Lipschitz for $x$ bounded away from zero we see that, since $\rho_0$ is bounded away from zero and from one,
	\begin{equation}\label{eq:rhoest}
	\big\vert\sqrt{\rho_{s}^n(0)(1-\rho_{s}^n(0))} - \sqrt{\rho_{s}^n(y)(1-\rho_{s}^n(y))}\big\vert\;\lesssim\; \frac{|y|}{n}\,.
	\end{equation}
	As in the proof of the previous lemma we write the square of the second term in the middle line of~\eqref{eq:fgammasquared} as
	\begin{equation}\label{eq:A2square}
	\begin{aligned}
	&\sum_x \int_0^\infty\dd t_1\int_0^\infty \dd t_2\,e^{-\gamma(t_1+t_2)} A_2(s,t_1,x)A_2(s,t_2,x) \Psi_{\{x\}}(s,\eta_{s})^2 \\
	&+ \sum_{x_1\neq x_2}\int_0^\infty\dd t_1\int_0^\infty \dd t_2\, e^{-\gamma(t_1+t_2)} A_2(s,t_1,x_1)A_2(s,t_2,x_2) \Psi_{\{x_1\}}(s,\eta_{s})\Psi_{\{x_2\}}(s,\eta_{s})\,.
	\end{aligned}
	\end{equation}
	We first deal with the first term above. For fixed $t_1, t_2$, using that $\Psi$ is bounded and~\eqref{eq:rhoest} we estimate
	\begin{equation*}
	\begin{aligned}
	\sum_x A_2(s,t_1,x)A_2(s,t_2,x)\Psi_{\{x\}}(s,\eta_{s})^2 &\lesssim
	\frac{1}{\ell^2}\sum_x \sum_{y_1,y_2=0}^{\ell-1}p_{t_1}(x,y_1)p_{t_2}(x,y_2)\frac{y_1y_2}{n^2} \\&=
	\frac{1}{\ell^2}\sum_{y_1,y_2=0}^{\ell-1}p_{t_1+t_2}(y_1,y_2)\frac{y_1y_2}{n^2}\\
	&=\frac{1}{\ell^2}\frac{1}{n^2}\sum_{y_1=0}^{\ell-1}y_1\bbE_{y_1}[X_{t_1+t_2}\one_{\{0\leq X_{t_1+t_2}\leq \ell-1\}}]\,,
	\end{aligned}
	\end{equation*}
	where $X$ above is a simple random walk accelerated by $n^2$ starting in $y_1$. Estimating the expectation by $\ell$, calculating the sum and using that $\ell=\epsilon n$ we see that the above is at most
	$\tfrac{\eps}{n}$. Integrating this with respect to $t_1$ and $t_2$ and taking the term $e^{-\gamma(t_1+t_2)}$ into account we obtain the estimate
	\begin{equation*}
	\frac{1}{\gamma^2}\frac{\eps}{n}\,.
	\end{equation*}
	We turn to the second term in~\eqref{eq:A2square}. Making use of the correlation estimate~\eqref{eq:cor}, we see that the expectation of the integrand is for fixed $t_1$ and $t_2$ at most a proportionality constant times
	\begin{align*}
	\frac{\sqrt{s}}{n}\sum_{x_1\neq x_2}A_2(s,t_1,x_1)A_2(s,t_2,x_2)&\;\lesssim\; \frac{\sqrt{s}}{n}\frac{1}{\ell^2}\sum_{x_1\neq x_2}\sum_{y_1,y_2=0}^{\ell-1}p_{t_1}(x_1,y_1)p_{t_2}(x_2,y_2)\frac{y_1y_2}{n^2}\\
	&\;\lesssim\; \frac{\sqrt{s}}{n^3 }\frac{1}{\ell^2}\sum_{y_1,y_2=0}^{\ell-1}\sum_x p_{t_1}(x,y_1)\sum_x p_{t_2}(x,y_2)y_1y_2\\
	&\;=\; \frac{\sqrt{s}}{n^3 }\frac{1}{\ell^2}\sum_{y_1,y_2=0}^{\ell-1} y_1y_2\\
	&\;\lesssim\; \frac{\eps^2\sqrt{s}}{n}\,.
	\end{align*}
	Here we used in the last line that $\ell=\eps n$.
	Plugging this estimate into~\eqref{eq:A2square} yields the bound
	\begin{equation*}
	\frac{1}{\gamma^2} \frac{\eps^2\sqrt{s}}{n}\,.
	\end{equation*}
	We thus  conclude the proof.
\end{proof}
The estimates on the diagonal terms above imply the following corollary.
\begin{corollary}
	\label{cor:A2}
	Let $T>0$, then uniformly in $n$ for all $0\leq s\leq T$ it holds that
	\begin{equation*}
	\sum_{x\in\Z}\bbE\Big[\Big(\int_0^\infty \dd t\, e^{-\gamma t}A_2(s,t,x)\Big)^2\Big]\;\lesssim\; \frac{1}{\gamma^2}\frac{\eps}{n}\,,
	\end{equation*}
	where the proportionality constant  depends solely on $T$.
\end{corollary}
We now come to the proof of Theorem~\ref{thm:KipnisVaradhan}. As aforementioned, we are going to show the result between times $0$ and $T$ for ease of notation (that is, we are going to show \eqref{zero_T}), which can be straightforwardly adapted to times $0\leq s<t\leq T$.
\begin{proof}
	Choosing $\gamma=\tfrac{1}{\sqrt{T}}$, Lemmas~\ref{lem:A1} and~\ref{lem:A2} imply that
	\begin{equation*}
	\int_0^T \dd s \,\bbE\left[f_\gamma(s, \eta_{s})^2\right] \;\lesssim\; \frac{T\eps^{3/4}}{n}(1+T^{1/4} + T)\,.
	\end{equation*}
	Regarding the second term on the right hand side of~\eqref{eq:KV} and expanding $f_\gamma$ in Fourier space, we see that it equals
	\begin{align*}
	4\int_0^T \dd t\sum_x \bbE\bigg[\Big(&\hat{f}_\gamma(s,x)\frac{(\eta_{s}(x+1)-\eta_{s}(x))}{\sqrt{\rho_{s}^n(x)(1-\rho_{s}^n(x))}}
	\\
	&+ \hat{f}_\gamma(s, x+1)\frac{(\eta_{s}(x)-\eta_{s}(x+1))}{\sqrt{\rho_{s}^n(x+1)(1-\rho_{s}^n(x+1))}}\Big)^2\bigg]\,.
	\end{align*}
	Hence, by the boundedness of $\eta$ and since $\rho$ is strictly bounded away from zero and one,
	the expression above is at most
	\begin{equation*}
	8\int_0^T \dd t\sum_x\bbE\big[\hat{f}_\gamma(s,x)^2+\hat{f}_\gamma(s,x+1)^2\big]\,.
	\end{equation*}
	Note that
	\begin{equation*}
	\hat{f}_\gamma(s,x) \;=\; \int_0^\infty \dd t\, e^{-\gamma t}(A_1(s,t,x) + A_2(s,t,x)) \;\overset{\text{def}}{=}\;\hat{f}_{\gamma,1}(s,x) + \hat{f}_{\gamma,2}(s,x)\,.
	\end{equation*}
	Thus,
	\begin{equation*}
	\int_0^T\dd t\sum_x \bbE[\hat{f}_\gamma^2(s,x)]\;\leq\; 2\int_0^T \dd t \sum_x\bbE\big[\hat{f}_{\gamma,1}^2(s,x) + \hat{f}_{\gamma,2}^2(s,x)\big]\,.
	\end{equation*}
	The result can now be deduced with the help of Corollaries~\ref{cor:A1} and~\ref{cor:A2}. This completes the proof.
\end{proof}

\section{Correlation function estimates -- Proofs of Theorems~\ref{gnE} and~\ref{grad}}\label{sec_4}
We start this section by introducing more notation.
Given $k$ times $0\leq t_1\leq t_2 \leq \cdots \leq t_k\leq T$ and $k$ points $x_1,\dots, x_k$, let us define the correlation function as
\begin{equation}\label{defgcor}
\phi(t_1,\dots, t_k; x_1,\dots,x_k)\;:=\; \bb E_{\nu_{\rho_0^n(\cdot)}}\Big[\prod_{i=1}^k \overline{\eta}_{t_i}(x_i) \Big]\,.
\end{equation}
Note that we do not require that $t_1,\ldots, t_k$ or that $x_1,\dots,x_k$ are distinct from each other.

For some special cases, an upper bound on the correlation functions has been obtained. When $k=2$, it was shown in \cite[Lemma 3.2]{jaralandim2006} that if $0\leq s<t\leq T$, then
\begin{equation}\label{eq:JaraLandim}
\sup_{x_1,\, x_2\,\, \text{distinct}}\big\lvert \phi(t,t; x_1,x_2)\big\rvert \;\leq\; \frac{C\sqrt{t}}{n} \,,\quad \sup_{x_1,\,x_2}\big\lvert \phi(s,t; x_1,x_2)\big\rvert \;\leq\; \frac{C}{n}\Big\{\sqrt{s}+\frac{1}{\sqrt{t-s}}\Big\}
\end{equation}
for some constant $C$ that depends only on the initial profile $\rho_0(\cdot)$. Note that the second inequality holds even if $x_1=x_2$, whereas this is not true for the first bound.

An estimate of the correlation function of multiple spatial points at a fixed time has been obtained in \cite{FPSV_I}.
\begin{theorem}[\cite{FPSV_I}]\label{Vares}
	Fix an integer $k\geq 2$ and assume that the initial profile $\rho_0(\cdot)$ is continuously differentiable. Then there exists a constant $C=C(\rho_0,T)$ such that for all $0\leq t\leq T$,
	$$\sup_{x_1,\dots,\, x_k\,\,\text{\rm distinct}}\big\lvert \phi(t,\dots, t; x_1,\dots,x_k) \big\rvert \;\leq\; Cn^{-k/2}$$
	if $k$ is even, and
	$$\sup_{x_1,\dots,\,x_k\,\, \text{\rm distinct}}\big\lvert \phi(t,\dots, t; x_1,\dots,x_k) \big\rvert \;\leq\; Cn^{-(k+1)/2}\log n$$
	if $k$ is odd.
\end{theorem}
In particular, this theorem implies that
\begin{equation}\label{samet}
\sup_{x_1,\dots,\,x_k\,\, \text{distinct}}\big\lvert \phi(t,\dots, t; x_1,\dots,x_k) \big\rvert \;\leq\; Cn^{-k/2}
\end{equation}
for all integers $k\geq 2$.
Our first step in obtaining a full estimate on~\eqref{defgcor} is to generalize Theorem~\ref{Vares} to the case in which the points $x_i$'s are \emph{not} necessarily distinct from each other. Recall that $\Lambda^k$ denotes the set of non-repetitive points with $k$ coordinates:
$$\Lambda^k\;:=\;\big\{\mathbf x=(x_1,\dots,x_k)\in\bb Z^k: x_i\neq x_j, \,\, \forall\,\, 1\leq i,j\leq k\big\}\,.$$
For a list of non-repetitive points $(x_i: i\in I)$, we sometimes write it as $\{x_i: i\in I\}$  in order to perform set operations.

We start by extending Theorem~\ref{Vares} to allow repetitive points. To that end we first need to introduce more notation. Consider a list of points $(x_i: i\in I)$. We denote by $|(x_i: i\in I)|$ the total number of points, which is equal to the cardinality of $I$, and by $\|(x_i:i\in I)\|$ the total number of non-repetitive points. In particular, repetitive points are not counted in $\|(x_i:i\in I)\|$. For example, given two real numbers $a\neq b$, then $|(a,a,b)|=3$ but $\|(a,a,b)\|=1$.

\begin{proposition}\label{ndis}
	Assume that the initial profile $\rho_0(\cdot)$ is $C^1$. Fix non-negative integers $k\geq 2$ and $\ell\leq k$. Then there exists a constant $C=C(\rho_0,T)$ such that, for all $0\leq t\leq T$,
	$$\sup_{\|(x_i\,:\, 1\leq i\leq k)\|=\ell}\big\lvert \phi(t,\dots, t; x_1,\dots,x_k) \big\rvert \;\leq\; Cn^{-\ell/2}\,.$$
\end{proposition}
\begin{remark}\rm
	The above result is natural. Indeed, assume that $k=2$ and that there are only repetitive points. In that case $\ell=0$ and the left hand side indeed does not decay in $n$.
\end{remark}
\begin{proof}
	We prove this theorem by induction on $k$. The case $k=2$ is a direct consequence of~\eqref{eq:JaraLandim}. Indeed, \eqref{eq:JaraLandim} deals with the case $k=\ell=2$ and the case $\ell=0< k=2$ is clear. When $k=3$, the case $\ell =0$ is trivial and the case $\ell =3$ is exactly~\eqref{samet}. The remaining case is  $\ell=1$, or in other words, where $x_1=x_2\neq x_3$. Using the  identity
	\begin{equation}\label{identity}
	\overline{\eta}(x)^2\;=\;(1-2\rho^n(x))\overline{\eta}(x)+ \rho^n(x)-\big(\rho^n(x)\big)^2\,,
	\end{equation}
	we obtain that
	\begin{equation*}
	\begin{aligned}
	\phi(t,t, t; x_1,x_1,x_3) &\;=\;
	\bb E[\overline{\eta}_t(x_1)\overline{\eta}_t(x_3)^2]\\
	& \;=\; (1-2\rho_t^n(x_3))\bb E[\overline{\eta}_t(x_1)\overline{\eta}_t(x_3)] + (\rho_t^n(x_3)-\rho_t^n(x_3)^2)\bb E[\overline{\eta}_t(x_1) ]\,,
	\end{aligned}
	\end{equation*}
	hence
	\begin{equation*}
	\big\lvert \phi(t,t, t; x_1,x_1,x_3) \big\rvert \;\lesssim\; \big\lvert \phi(t, t; x_1,x_3) \big\rvert \;\lesssim\; n^{-1}\,,
	\end{equation*}
	where  the contribution coming from $\rho(x)-\rho(x)^2$ is zero since all random variables are centred.

	Assume that the statement of the theorem holds for the cases when the total number of points is strictly less than $k$ for some $k\geq 4$. The case $\ell=k$ is just~\eqref{samet}. Thus, assume that there is at least one non-repetitive point. Assume without loss of generality that $x_k=x_{k-1}$. Note that under this assumption, the total number of non-repetitive points in $(x_1,\dots,x_k)$ is less than or equal to
	$$\|(x_1,\dots,x_{k-2})\|\wedge \|(x_1,\dots,x_{k-1})\|\,.$$
	This observation together with identity \eqref{identity}  yield
	\begin{equation*}
	\big\lvert \phi(t,\dots, t; x_1,\dots,x_k) \big\rvert \;\lesssim\; \big\lvert \phi(t,\dots, t; x_1,\dots,x_{k-1}) \big\rvert + \big\lvert \phi(t,\dots, t; x_1,\dots,x_{k-2}) \big\rvert,
	\end{equation*}
	which by induction assumption is bounded by
	$$Cn^{-\|(x_i\,:\, 1\leq i\leq k-1)\|/2}+Cn^{-\|(x_i\,:\, 1\leq i\leq k-2)\|/2}\;\lesssim\;n^{-\|(x_i\,:\, 1\leq i\leq k)\|/2}\,.$$
\end{proof}
The plan for the rest of this section is as follows. In Subsection \ref{41} we estimate the correlation functions of two lists of non-repetitive points at two different times. This is the content of Theorem~\ref{twotimecf}, which will be proved subject to some technical estimates.  Based on Theorem~\ref{twotimecf}, in Corollary~\ref{cftwotime} we estimate the correlation functions with two sets of general points at two different times. In Subsection \ref{42} we prove Theorem~\ref{gnE} subject to some technical estimates and as a corollary of Theorem~\ref{twotimecf}. In Subsection \ref{43} we prove Theorem~\ref{grad}. Theorem~\ref{grad} will then allow to provide the proofs of the technical estimates that were needed in the proofs of Theorems~\ref{twotimecf} and Theorem~\ref{gnE}. The proofs of these estimates will be given in Subsections \ref{44} and \ref{45}.

\subsection{Two times correlation estimates and some corollaries}\label{41}
Recall that we are assuming in this paper that the initial profile $\rho_0(\cdot)$ is $C^2$ with bounded derivatives.
\begin{theorem}\label{twotimecf}
	Fix non-negative integers $k_1, k_2$ and consider two lists of non-repetitive points $(x_i: 1\leq i\leq k_1)$ and $(z_j: 1\leq j\leq k_2)$.
	There exists a constant $C$ independent of $(x_i: 1\leq i\leq k_1)$ and $(z_j: 1\leq j\leq k_2)$ such that
	\begin{equation}
	\label{eq:twotimes}\Big\lvert \bb E_{\nu_{\rho_0^n(\cdot)}}\Big[ \prod_{i=1}^{k_1} \overline{\eta}_{t}(x_i) \prod_{j=1}^{k_2}\overline{\eta}_s(z_i) \big)\Big] \Big\rvert \,\leq \,Cn^{-(k_1+k_2)/2} \bigg(\frac{n}{\sqrt{n^2(t-s)+1}}\bigg)^{k_1\wedge k_2},
	\end{equation}
	for any $s<t$ such that $0\leq s,t\leq T$ and every integer $n\geq 1$.
\end{theorem}
\begin{remark}\rm
	The upper bound above can be improved a little bit if we  use the original bounds in Theorem \ref{Vares} instead of the simplified one in \eqref{samet}. In that case the right hand side in~\eqref{eq:twotimes} would become
	$$C h_n(k_1+k_2) \bigg(\frac{n}{\sqrt{n^2(t-s)+1}}\bigg)^{k_1\wedge k_2}$$
	where
	\begin{equation*}
	h_n(k)\;=\;
	\begin{cases}
	n^{-k/2} &\quad \text{if}\,\, k\,\, \text{is even},\\
	n^{-(k+1)/2}\log n &\quad \text{if}\,\, k\,\, \text{is odd}.\\
	\end{cases}
	\end{equation*}
	However, doing so would make the proofs of this theorem and related lemmas more complicated to follow. In addition, the upper bound obtained in this theorem is already enough for our purpose of the proof of tightness of the current in Section \ref{sec_5}.   The same applies to Theorem~\ref{gnE}.
\end{remark}
\begin{remark}\rm
	A previous result on the estimate of the two times correlation was obtained by C.~Landim in \cite{Landim05}. The upper bound obtained there is
	$$C\Big\{ \frac{\log n}{n^2}+ \frac{1}{(t-s)n^2+1}\Big\}\,.$$
	Our result in Theorem \ref{twotimecf} improves this bound significantly.
\end{remark}

\begin{proof}[Proof of Theorem~\ref{twotimecf}]
	The case $k_1=0$ is just Theorem~\ref{Vares}. Therefore, from now on we assume that $k_1\neq 0$.	Throughout the proof of this theorem, the collection of points $(z_i: 1\leq j\leq k_2)$ and the time $s$ are fixed. To simplify the notation, we write
	\begin{equation}\label{defvp}
	\varphi_t(x_i: 1\leq i\leq k_1)\;=\;  \bb E_{\nu_{\rho_0^n(\cdot)}}\Big[ \prod_{i=1}^{k_1} \overline{\eta}_{t}(x_i) \prod_{j=1}^{k_2}\overline{\eta}_s(z_j) \big)\Big]\,,\qquad t\geq s\,.
	\end{equation}
	A tedious but straightforward computation (see also~\cite{FPSV_I}) shows that, for $t\geq s$,
	\begin{equation*}
	\begin{split}
	\partial_t\varphi_t(x_i:1\leq i\leq k_1)\;=\; &L^{\lex}_{k_1}\,\varphi_t(x_i:1\leq i\leq k_1)\\
	&+ 2n^2 \sum_{i,j=1}^{k_1} \one\{|x_{j}-x_i|=1\}\big(\rho_t^n(x_{j})-\rho_t^n(x_i)\big)\big[\widetilde{\varphi}_t(x_{j})-\widetilde{\varphi}_t(x_i)\big]\\
	&- n^2 \sum_{i,j=1}^{k_1} \one\{|x_j-x_i|=1\}\big(\rho_t^n(x_{j})-\rho_t^n(x_i)\big)^2 \widetilde{\varphi}_t(x_i,x_j)\,,\\
	\end{split}
	\end{equation*}
	where $\widetilde{\varphi}_t(A)=\varphi_t(\{x_i: 1\leq i\leq k_1\}\backslash A)$ for any subset $A\subset\{x_1,\dots,x_{k_1}\}$, and $L_{k_1}^{\lex}$ is the generator of the SSEP with $k_1$ labelled particles accelerated by a factor $n^2$. Denote by $(\mathbf{X}^i:1\leq i\leq k_1)$ the accelerated SSEP with $k_1$ labelled particles, and by $\mathbf{P}_{(x_i\,:\,1\leq i\leq k_1)}$(resp.\ $\mathbf{E}_{(x_i\,:\,1\leq i\leq k_1)}$) the probability (resp.\ expectation) with respect to $(\mathbf{X}^i:1\leq i\leq k_1)$ starting from initial points $(x_i:1\leq i\leq k_1)$. By $\mathbf{X}^i_t$ we represent the location at time $t$ of the particle which was at site $x_i$ at time $0$.
	Applying Duhamel's Principle, for any $t>s$, we can write $\varphi_t(x_i:1\leq i\leq k_1)$ as the sum of an initial term and an integral term:
	\begin{equation}\label{indeq}
	\begin{split}
	\varphi_t(x_i:1\leq i\leq k_1)\;=\;& \mathbf{E}_{(x_i\,:\,1\leq i\leq k_1)}\big[\varphi_s(\mathbf{X}^i_{t-s}: 1\leq i\leq k_1) \big] \\
	&+\mathbf{E}_{(x_i\,:\, 1\leq i\leq k_1)}\Big[ \int_s^t \Psi(\mathbf{X}_{t-r},r) dr\Big]
	\end{split}
	\end{equation}
	where for every integer $k_1\geq 1$,
	\begin{equation}\label{defPsi}
	\begin{split}
	\Psi((x_i: 1\leq i\leq k_1),r) \;:=\;
	&2n^2  \sum_{i,j=1}^{k_1} \one\{|x_j-x_i|=1\}\big(\rho_{r}^n(x_j)-\rho_{r}^n(x_i)\big)\big[\widetilde{\varphi}_{r}(x_j)-\widetilde{\varphi}_{r}(x_i)\big]\\
	&- n^2\sum_{i,j=1}^{k_1} \one\{|x_j- x_i|=1\}\big(\rho_{r}^n(x_j)-\rho_{r}^n(x_i)\big)^2 \widetilde{\varphi}_{r}(x_i,x_j)\,.
	\end{split}
	\end{equation}

	We will show in Corollary~\ref{initialp} ahead that the initial term at the right hand side of \eqref{indeq} is bounded from above by
	$$Cn^{-(k_1+k_2)/2}\Big(\frac{n}{\sqrt{n^2(t-s)+1}}\Big)^{k_1\wedge k_2}\,,$$ while the integral term will be handled in Proposition \ref{estimatePsi}, which shows that
	$$\Big\lvert\mathbf{E}_{(x_i\,:\, 1\leq i\leq k_1)}\Big[ \int_s^t \Psi(\mathbf{X}_{t-r},r) dr\Big]\Big\rvert \;\lesssim\; n^{-(k_1+k_2)/2}\Big(\frac{n}{\sqrt{n^2(t-s)+1}}\Big)^{k_1\wedge k_2},$$
	proving the theorem.
\end{proof}

We now give an estimate on the two time correlation function involving with two lists of general points.
To formulate the result we need to introduce one more notation. Given a list of $k$ points $(x_i: 1\leq i\leq k)$ we write $\mc R((x_i: 1\leq i\leq k))$ for the collection of index sets $I$ such that
\begin{itemize}
	\item[(1)] $I\subseteq \{1,\ldots, k\}$,
	\item[(2)] the list $(x_i: i\in I)$ consists only of non-repetitive points,
	\item[(3)] every non-repetitive point in $(x_i: i\in I)$ is also contained in $(x_i: i\in I)$.
\end{itemize}

Note that if $I\in \mc R((x_i: 1\leq i\leq k))$, then $(x_i: i\in I)$ might contain more non-repetitive points than $(x_i: 1\leq i\leq k)$. As an example consider $(x_i: 1\leq i\leq k)= (a,b,c,c,d,d)$ and $I=\{1,2,3\}$, then $(x_i:i\in I)= (a,b,c)$.
\begin{corollary}\label{cftwotime}
	Fix non-negative integers $k_1, k_2$ and consider two lists of points $(x_i: 1\leq i\leq k_1)$ and $(z_j: 1\leq j\leq k_2)$.  Suppose that
	$$\| (x_i: 1\leq i\leq k_1)\|\;=\; \ell_1 \quad\text{and}\quad  \| (z_j: 1\leq j\leq k_2)\|\;=\; \ell_2\,.$$
	Then there exists a constant $C$ independent of $(x_i: 1\leq i\leq k_1)$ and $(z_i: 1\leq i\leq k_2)$ such that
	$$\Big\lvert \bb E_{\nu_{\rho_0^n(\cdot)}}\big[ \prod_{i=1}^{k_1} \overline{\eta}_{t}(x_i) \prod_{j=1}^{k_2}\overline{\eta}_s(z_j) \big)\big] \Big\rvert \;\leq\; C\sup_{I,J}\,n^{-(|I|+|J|)/2} \Big(\frac{n}{\sqrt{n^2(t-s)+1}}\Big)^{|I|\wedge |J|}$$
	for any $s<t$ with $0\leq s,t\leq T$ and every integer $n\geq 1$, where  the supremum is taken over all subsets $I\in \mc R((x_i: 1\leq i\leq k_1))$ and $J\in \mc R((z_j:1\leq j\leq k_2))$.
	Observing that
	$$|I|+|J|\geq \ell_1+\ell_2 \quad \text{and}\quad |I|\,\leq \,\frac{k_1+\ell_1}{2},\,\, |J|\,\leq\,\frac{\ell_2+k_2}{2},$$
	we have in particular that
	$$\Big\lvert \bb E_{\nu_{\rho_0^n(\cdot)}}\big[ \prod_{i=1}^{k_1} \overline{\eta}_{t}(x_i) \prod_{j=1}^{k_2}\overline{\eta}_s(z_j) \big)\big] \Big\rvert \;\leq\; Cn^{-(|\ell_1|+|\ell_2|)/2} \Big(\frac{n}{\sqrt{n^2(t-s)+1}}\Big)^{\frac{\ell_1+k_1}{2}\wedge \frac{\ell_2+k_2}{2}}\,.$$
\end{corollary}
\begin{proof}
	Using  the identity \eqref{identity}
	we can bound the absolute value of the expectation in the statement of this corollary
	by some constant depending only on $\rho_t(\cdot)$ times
	$$ \sum_{I,J}\Big\lvert \bb E_{\nu_{\rho_0^n(\cdot)}}\big[ \prod_{i\in I} \overline{\eta}_{t}(x_i) \prod_{j\in J }\overline{\eta}_s(z_j) \big)\big] \Big\rvert\, ,$$
	where the (finite) sum above is taken over all  $I\in \mc R((x_i: 1\leq i\leq k_1))$  and all  $J\in \mc R((z_j: 1\leq j\leq k_2))$. Applying Theorem~\ref{twotimecf} to estimate each term in the above sum, yields the first bound. For the second bound note that by the third item in the definition of $\mc R((x_i: 1\leq i\leq k))$ one has that $|I|\geq \ell_1$.
	To see that also $|I|\leq \tfrac{k_1+\ell_1}{2}$, note that to construct $I$ one first chooses all the $\ell_1$ indices corresponding to non-repetitive points in $(x_1: 1\leq i\leq k_1)$. Since all other points in $(x_1: 1\leq i\leq k_1)$ repeat at least twice, at most $\tfrac{k_1-\ell_1}{2}$ indices can be added to $I$. This finishes the proof.

\end{proof}

The next corollary will be helpful in establishing the next result which will play a key role in the tightness proof of the current in Section \ref{sec_5}.           Before stating it, we introduce some notation. Given a list of points $(x_i: 1\leq i\leq k)$, let $m_i$ be the number of points which are repeated $i$ times in $(x_i: 1\leq i\leq k)$ and denote the maximal number of distinct points appearing in $(x_i: 1\leq i\leq k)$ by
\begin{equation}\label{def[]}
[(x_i: 1\leq i\leq k)]\;=\;\sum_{i\geq 1}m_i\,.
\end{equation}
To illustrate the difference of $[\cdot]$ from $|\cdot|$ and $\|\cdot\|$ which are defined before, consider a list of points $(a,b,c,c,d,d,d)$ with $a\neq b\neq c\neq d$. Then
$$m_1=2, m_2=1, m_3=1\quad \text{and} \,\, [(a,b,c,c,d,d,d)]=4;$$
$$|(a,b,c,c,d,d,d)|\,=\,7, \quad \|(a,b,c,c,d,d,d)\|=2.$$
The following equations hold:
\begin{equation}\label{equk}
k\;=\;\sum_{i\geq 1}im_i\,, \quad \|(x_i: 1\leq i\leq k)\|\;=\;m_1\,.
\end{equation}
\begin{proposition}\label{expdif}
	Fix a list of points $(x_i: 1\leq i\leq k)$ with 	$\| (x_i: 1\leq i\leq k)\|= \ell$ and $[(x_i: 1\leq i\leq k)]=m$.
	Then, there exists a constant $C$ independent of the points in $(x_i: 1\leq i\leq k)$ such that
	$$\Big\lvert \bb E_{\nu_{\rho_0^n(\cdot)}}\Big[\prod_{i=1}^{k} \big(\overline{\eta}_t(x_i)- \overline{\eta}_s(x_i)\big) \Big]\Big\rvert \;\leq\; Cn^{-\ell/2}\Big(\frac{n}{\sqrt{n^2(t-s)+1}}\Big)^{\ell/2}\,,\quad \text{if}\;\; t-s\geq \frac{1}{n}$$
	and
	$$\Big\lvert \bb E_{\nu_{\rho_0^n(\cdot)}}\Big[\prod_{i=1}^{k} \big(\overline{\eta}_t(x_i)- \overline{\eta}_s(x_i)\big) \Big]\Big\rvert \;\leq\; Cn^{-m/2}\Big(\frac{n}{\sqrt{n^2(t-s)+1}}\Big)^{m-\ell/2}\,,\quad \text{if}\;\; t-s\leq \frac{1}{n}\,.$$
\end{proposition}
\begin{proof}

	We first derive a formula for the upper bound of the expectation to be estimated in this proposition.
	Opening the product inside the expectation and using the triangle inequality, we obtain
	\begin{equation*}
	\Big\lvert \bb E_{\nu_{\rho_0^n(\cdot)}}\Big[\prod_{i=1}^{k} \big(\overline{\eta}_t(x_i)- \overline{\eta}_s(x_i)\big) \Big]\Big\rvert
	\;\leq\;  \sum_{Q\subseteq\{1,\dots,\, k\}}\Big\lvert \bb E_{\nu_{\rho_0^n(\cdot)}}\Big[\prod_{i\in Q} \overline{\eta}_t(x_i)\prod_{j\in Q^c} \overline{\eta}_s(x_j) \Big]\Big\rvert
	\end{equation*}
	where $Q^c$ is the complement of $Q$ with respect to the set $\{1,\dots, k\}$. In view of Corollary \ref{cftwotime}, we have that
	\begin{equation*}
	\Big\lvert \bb E_{\nu_{\rho_0^n(\cdot)}}\Big[\prod_{i=1}^{k} \big(\overline{\eta}_t(x_i)- \overline{\eta}_s(x_i)\big) \Big]\Big\rvert
	\;\lesssim\; \sup_{Q\subseteq\{1,2,\ldots,k\}}\sup_{I,J}\,n^{-(|I|+|J|)/2} \Big(\frac{n}{\sqrt{n^2(t-s)+1}}\Big)^{|I|\wedge |J|}
	\end{equation*}
	where the supremum is taken over all subsets $I\in\mc R((x_i: i\in Q))$ and $J\in \mc R(x_j: j\in Q^c)$. Noticing that every non-repetitive point in $(x_i: 1\leq i\leq k)$ is a non-repetitive point in either $(x_i: i\in Q)$ or $(x_j:j\in Q^c)$, we can further conclude that
	\begin{equation}\label{IJ}
	\Big\lvert \bb E_{\nu_{\rho_0^n(\cdot)}}\Big[\prod_{i=1}^{k} \big(\overline{\eta}_t(x_i)- \overline{\eta}_s(x_i)\big) \Big]\Big\rvert
	\;\lesssim\; \sup_{I,J}\,n^{-(|I|+|J|)/2} \Big(\frac{n}{\sqrt{n^2(t-s)+1}}\Big)^{|I|\wedge |J|}
	\end{equation}
	where the supremum is taken over all subsets $I,J\in\{1,2,\cdots,k\}$ such that $I\cap J=\varnothing$, both $(x_i: i\in I)$ and $(x_j:j\in J)$ consist of only non-repetitive points, and every non-repetitive point in $(x_i:1\leq i\leq k)$ is contained in either $(x_i: i\in I)$ or $(x_j:j\in J)$.

	Let us denote the supremum on the right hand side of \eqref{IJ} by $M_{(x_i:1\leq i\leq k)}$. We are going to show that
	\begin{equation}\label{Mbound}
	M_{(x_i:1\leq i\leq k)} \;\leq\;
	\begin{cases}
	Cn^{-\ell/2}\Big(\frac{n}{\sqrt{n^2(t-s)+1}}\Big)^{\ell/2}\,,\quad &\text{if}\;\; t-s\geq \frac{1}{n}\\
	Cn^{-m/2}\Big(\frac{n}{\sqrt{n^2(t-s)+1}}\Big)^{m-\ell/2}\,,\quad &\text{if}\;\; t-s\leq \frac{1}{n}\,.
	\end{cases}
	\end{equation} The proof is based on induction in $m$. Obviously \eqref{Mbound} follows from Theorem \ref{twotimecf} if $m=\ell$ since in this case all points are non-repetitive. Since $m\geq \ell$ this is the basis case.

	Assume that the statement of \eqref{Mbound} holds for every list of points $(y_i: 1\leq i\leq k)$ and all $k\in\bb N$ such that $[(y_i: 1\leq i\leq k)]<m$. Consider now a list of points $(x_i: 1\leq i\leq k)$ such that $[(x_i: 1\leq i\leq k)]=m$. Pick a repetitive point of $(x_i: 1\leq i\leq k)$, say $x_{i_1}$. Assume this point is repeated $a$ times in $(x_i: 1\leq i\leq k)$: there exist indices $i_1,i_2,\dots, i_a$ such that
	$$x_{i_1}=x_{i_2}=\cdots=x_{i_a}.$$
	Let $A_{x_{i_1}}\subset\{1,2,\dots,k\}$ be the set of those indices: $A_{x_{i_1}}=\{i_1,i_2,\dots,i_a\}$.
	Denote by $I$ and $J$ those sets at which the supremum is achieved in \eqref{IJ}. We need to carefully  analyse to which sets of $I$ and $J$ do the elements of $A_{x_{i_1}}$ belong to.
	Recall that $I$ and $J$ correspond to non-repetitive lists so that at most one element of $A_{x_{i_1}}$ can belong to $I$ and $J$. There are four cases:
	\begin{enumerate}
		\item There is one element contained in $I$ and one element contained in $J$. W.l.o.g., assume $i_1\in I$ and $i_2\in J$.
		\item There is one element contained in $I$ and no element contained in $J$. W.l.o.g., assume $i_1\in I$.
		\item There is no element contained in $I$ and one element contained in $J$. W.l.o.g., assume $i_2\in J$.
		\item There is no element contained in $I$, neither in $J$.
	\end{enumerate}

	\textbf{Case $(1)$:}  Note that
	\begin{equation*}
	\begin{split}
	&n^{-(|I|+|J|)/2} \Big(\frac{n}{\sqrt{n^2(t-s)+1}}\Big)^{|I|\wedge |J|}\\
	&=n^{-(|I\backslash\{i_1\}|+|J\backslash\{i_2\}|)/2} \Big(\frac{n}{\sqrt{n^2(t-s)+1}}\Big)^{|I\backslash\{i_1\}|\wedge |J\backslash\{i_2\}|}\times \frac{1}{\sqrt{n^2(t-s)+1}}\,.
	\end{split}
	\end{equation*}
	Thus, by the definition of $ M_{(x_i: 1\leq i\leq k,\, i\notin A_{x_{i_1}})}$,
	$$n^{-(|I\backslash\{i_1\}|+|J\backslash\{i_2\}|)/2} \Big(\frac{n}{\sqrt{n^2(t-s)+1}}\Big)^{|I\backslash\{i_1\}|\wedge |J\backslash\{i_2\}|} \,\leq\, M_{(x_i: 1\leq i\leq k,\, i\notin A_{x_{i_1}})}\,.$$
	By the fact that $x_{i_1}$ is a repetitive point of the list $(x_i:i\leq i\leq k)$, we have that $\|(x_i: 1\leq i\leq k,\, i\notin A_{x_{i_1}})\|=\ell$ and $[(x_i: 1\leq i\leq k,\, i\notin A_{x_{i_1}})]=m-1$.
	Hence,
	\begin{equation*}
	M_{(x_i: 1\leq i\leq k,\, i\notin A_{x_{i_1}})}\;\lesssim\;
	\begin{cases}
	n^{-\ell/2}\Big(\frac{n}{\sqrt{n^2(t-s)+1}}\Big)^{\ell/2}\,,\quad &\text{if}\;\; t-s\geq \frac{1}{n}\\
	Cn^{-m/2+1/2}\Big(\frac{n}{\sqrt{n^2(t-s)+1}}\Big)^{m-\ell/2-1}\,,\quad &\text{if}\;\; t-s\leq \frac{1}{n}.
	\end{cases}
	\end{equation*}
	Based on the three inequalities above, applying the induction assumption on\break $M_{(x_i: 1\leq i\leq k,\, i\notin A_{x_{i_1}})}$, we get
	\begin{equation*}
	M_{(x_i:1\leq i\leq k)} \;\leq\;
	\begin{cases}
	Cn^{-\ell/2}\Big(\frac{n}{\sqrt{n^2(t-s)+1}}\Big)^{\ell/2}\,,\quad &\text{if}\;\; t-s\geq \frac{1}{n}\\
	Cn^{-m/2}\Big(\frac{n}{\sqrt{n^2(t-s)+1}}\Big)^{m-\ell/2}\,,\quad &\text{if}\;\; t-s\leq \frac{1}{n}.
	\end{cases}
	\end{equation*}
	Here, we used that $\tfrac{1}{\sqrt{n^2(t-s)+1}}\lesssim \tfrac{1}{\sqrt{n}} \leq 1$ if $t-s\geq \tfrac{1}{n}$.
	We thus can conclude the analysis of the first case.

	\textbf{Case $(2)$:} Note that
	\begin{equation*}
	\begin{split}
	&n^{-(|I|+|J|)/2} \Big(\frac{n}{\sqrt{n^2(t-s)+1}}\Big)^{|I|\wedge |J|}\\
	\leq\,&n^{-(|I\backslash\{i_1\}|+|J|)/2} \Big(\frac{n}{\sqrt{n^2(t-s)+1}}\Big)^{|I\backslash\{i_1\}|\wedge |J|}\times n^{-1/2}\frac{n}{\sqrt{n^2(t-s)+1}}\,.
	\end{split}
	\end{equation*}
	Here, we used that $t-s\leq T$, so that $\tfrac{1}{\sqrt{n^2(t-s)+1}}\gtrsim 1$.
	Similar to the case $(1)$, the expression on the right hand side is bounded by
	$$M_{(x_i: 1\leq i\leq k,\, i\notin A_{x_{i_1}})}\times n^{-1/2}\frac{n}{\sqrt{n^2(t-s)+1}}\,.$$
	Using the induction assumption we obtain
	\begin{equation*}
	M_{(x_i:1\leq i\leq k)} \;\leq\;
	\begin{cases}
	Cn^{-\ell/2}\Big(\frac{n}{\sqrt{n^2(t-s)+1}}\Big)^{\ell/2}\,,\quad &\text{if}\;\; t-s\geq \frac{1}{n}\\
	Cn^{-m/2}\Big(\frac{n}{\sqrt{n^2(t-s)+1}}\Big)^{m-\ell/2}\,,\quad &\text{if}\;\; t-s\leq \frac{1}{n}.
	\end{cases}
	\end{equation*}
	Here, we used that $\tfrac{1}{\sqrt{n^2(t-s)+1}}\lesssim \tfrac{1}{\sqrt{n}}$ if $t-s\geq \tfrac{1}{n}$.

	\textbf{Case $(3)$:} This case can be dealt with the same arguments as in Case $(2)$.

	\textbf{Case $(4)$.} In this case we have that
	$$n^{-(|I|+|J|)/2} \Big(\frac{n}{\sqrt{n^2(t-s)+1}}\Big)^{|I|\wedge |J|}\,\leq\, M_{(x_i: 1\leq i\leq k,\, i\notin A_{x_{i_1}})}.$$
	From this inequality and the induction assumption, we finally get
	\begin{equation*}
	M_{(x_i:1\leq i\leq k)} \;\leq\;
	\begin{cases}
	Cn^{-\ell/2}\Big(\frac{n}{\sqrt{n^2(t-s)+1}}\Big)^{\ell/2}\,,\quad &\text{if}\;\; t-s\geq \frac{1}{n}\\
	Cn^{-m/2+1/2}\Big(\frac{n}{\sqrt{n^2(t-s)+1}}\Big)^{m-\ell/2-1}\,,\quad &\text{if}\;\; t-s\leq \frac{1}{n}.
	\end{cases}
	\end{equation*}
	To conclude this case note that if $t-s\leq \tfrac1n$, then
	\begin{equation*}
	\sqrt{n}\frac{\sqrt{n^2(t-s)+1}}{n}\;\leq\; \frac{\sqrt{n+ 1}}{\sqrt{n}}\;\lesssim\; 1\,.
	\end{equation*}

	We thus conclude the proof.
\end{proof}

\subsection{ Multiple times correlation estimate -- Proof of Theorem~\ref{gnE}}\label{42}
We now prove Theorem~\ref{gnE}.

\begin{proof}
	We prove this proposition by induction on $m$, the number of lists of points. The case $m=2$ is just Theorem \ref{twotimecf}.

	Assume $m\geq 3$ and that we have proved the statement of the proposition in the case in which  the number of lists of points is strictly less than $m$. The rest of the proof is quite similar to the proof of Proposition~\ref{twotimecf}, so we only outline it here.
	For  $t>t_{m-1}$ and a collection of points $(x_{i}: 1\leq i \leq k)$~, define the function
	\begin{equation}\label{defphi}
	\phi_t(x_{i}: 1\leq i\leq k)\;=\; \bb E_{\nu_{\rho_0^n(\cdot)}}\Big[\Big( \prod_{i=1}^{k} \overline{\eta}_t(x_{i})\Big) \prod_{j=1}^{m-1} \prod_{i_j=1}^{k_j} \overline{\eta}_{t_j}(x_{i_j})\Big].
	\end{equation}
	Taking the time derivative, and then using the Duhamel's principle, we obtain
	\begin{equation}\label{phieq}
	\begin{split}
	\phi_t(x_i:1\leq i\leq k)\;=\; &\mathbf{E}_{(x_i\,:\,1\leq i\leq k)}\big[\phi_{t_{m-1}}(\mathbf{X}^i_{t-t_{m-1}}: 1\leq i\leq k) \big] \\
	+\,& \mathbf{E}_{(x_i\,:\, 1\leq i\leq k)}\Big[ \int_{t_{m-1}}^t \dd r\, \Phi(\mathbf{X}_{t-r},r) \Big]\,,
	\end{split}
	\end{equation}
	where for every integer $k\geq 2$,
	\begin{equation*}
	\begin{split}
	\Phi((x_i: 1\leq i\leq k),r) \;:= \;&2n^2  \sum_{i,j=1}^{k} \one\{|x_j-x_i|=1\}\big(\rho_{r}^n(x_j)-\rho_{r}^n(x_i)\big)\big[\widetilde{\phi}_{r}(x_j)-\widetilde{\phi}_{r}(x_i)\big]\\
	&- n^2\sum_{i,j=1}^{k} \one\{|x_j- x_i|=1\}\big(\rho_{r}^n(x_j)-\rho_{r}^n(x_i)\big)^2 \widetilde{\phi}_{r}(x_i,x_j)\\
	\end{split}
	\end{equation*}
	and for $k=1$, $\Phi((x_i: 1\leq i\leq k),r):=0$. Here $\widetilde{\phi}_t(A)=\phi_t(\{x_i: 1\leq i\leq k\}\backslash A)$ for every subset $A\subset\{x_1,\dots,x_k\}$. The two terms at the right hand side of \eqref{phieq} are estimated in Corollary \ref{initialpm} and Proposition \ref{estimatePhi} respectively, both being absolutely bounded  by
	$$Cn^{-\sum_{j=1}^m k_j/2} \prod_{j=1}^{m-1} \Big(\frac{n}{\sqrt{n^2(t_{j+1}-t_j)+1}}\Big)^{k_j\wedge \sum_{l=j+1}^m k_l}. $$
	This proves the theorem.
\end{proof}

\subsection{Estimates on the transition probability of SSEP -- Proof of Theorem~\ref{grad}}\label{43}

\subsubsection{Bounds on transition probabilities}\label{Sec431}
In this subsection we summarize some of the results of Erhard/Hairer~\cite[Section 3]{EH22} and Landim~\cite{Landim05} about the transition probability of $k$ exclusion particles. In the former article the exclusion process was considered on a large torus, however since its results are based on the work of~\cite{Landim05}, which derived bounds on the same quantities considered on the whole space, the results from~\cite{EH22} carry over to the present context without any major modification.

We denote the transition probability of $k$ \emph{labeled exclusion} particles that evolve through the stirring process and are accelerated by a factor $n^2$ by $p^{\lex}(\cdot,\cdot)$. It was shown in~\cite[Theorem~3.1]{Landim05} that there exist constants $C_1, C_2 > 0$ such that for all $x,y\in \Lambda^k$ and all $n^2t> C_1$,
\begin{equation}\label{eq:upperexclusion}
p_t^{\lex}(\mathbf x,\mathbf y)\;\leq\; \frac{C_1}{(1+n^2t)^{k/2}}  \exp\Big\{-\frac{C_2 n^2t}{2(\log n^2t)^2}\Phi\Big(\frac{|\mathbf x-\mathbf y| \log n^2t}{C_2^2 n^2t}\Big)\Big\}\,,
\end{equation}
where
\begin{equation*}
\Phi(u) \;=\; \sup_{w\in \R}\{uw-w^2 \cosh w\}\,.
\end{equation*}
In~\cite[Lemma 3.2]{EH22} it was shown that we can further estimate for all $\mathbf x,\mathbf y$ and $t$ as above
\begin{equation*}
\frac{C_1}{(1+n^2t)^{k/2}}  \exp\Big\{-\frac{C_2 n^2t}{2(\log n^2t)^2}\Phi\Big(\frac{|\mathbf x-\mathbf y| \log n^2t}{C_2^2 n^2t}\Big)\Big\}\;\lesssim\; \prod_{i=1}^{k}\bar p_t(x_i,y_i)\,,
\end{equation*}
where $\bar p(\cdot,\cdot)$ is a kernel (not necessarily a probability kernel) defined on $\R_+\times\Z\times\Z$\
satisfying
\begin{equation}\label{eq:largetbarp}
\bar p_t(x_i, y_i)\; =\; \frac{1}{\sqrt{1+n^2t}} \exp\bigg\{-\frac{C_2 n^2t}{2k(\log n^2t)^2}\Phi\Big(\frac{|x_i-y_i| \log n^2t}{C_2^2 n^2t}\Big)\bigg\}\,.
\end{equation}
Recall that $\bar p_t(x,y)$ was defined only for $n^2t>C_1$.
If we define, for $n^2t\leq C_1$,
\begin{equation*}
\bar p_t(x,y)\;=\; e^{-|x-y|}\,,
\end{equation*}
then we still have that, for all $t\geq 0$,
\begin{equation}\label{boundhk}
p_t^{\lex}(\mathbf x,\mathbf y)\;\lesssim\; \prod_{i=1}^{k}\bar p_t(x_i,y_i)\,.
\end{equation}
Note that $\bar p$ actually depends on the number of particles $k$. However, since in our applications $k$ will always be a fixed number we do not indicate this in the notation.
Since $\Phi(w)\sim w^2$ for $w$ small and $\Phi(w)\sim w\log w$ for large $w$  we conclude that, for $t\geq C_1$, there exists a constant $a_1>0$ such that
\begin{equation}\label{eq:gaussian}
\bar p_t(x,y) \;\leq\; \frac{1}{\sqrt{n^2t+1}} \exp\Big\{-\frac{|x-y|^2}{a_1n^2t}\Big\}
\end{equation}
provided that $|x-y|\leq n^2t/\log n^2t$. The above essentially shows that $\bar p$ satisfies a Gaussian bound for large time and in the interesting spatial regime.
If $|x-y|\geq n^2t/\log n^2t$ (but still $n^2t\geq C_1$), then it is possible to show that there exists a constant $a_0> 0$ such that
\begin{equation}\label{eq:subgaussian}
\bar p_t(x,y) \;\leq\; \frac{1}{\sqrt{n^2t+1}} \exp\Big\{-\frac{|x-y|}{4a_0 \log n^2t}\log\frac{|x-y|\log n^2t}{4ea_0^2n^2t}\Big\}\,.
\end{equation}
We continue by presenting some of the most important properties of $\bar p$.
Let $\zeta\in[0,1]$, and denote $\Z/n= \{z_n:\, \exists\, z\in \Z \text{ s.t. } z_n= z/n\}$. Then, it was shown in~\cite[Lemma 3.5]{EH22} that
\begin{itemize}
	\item[1)] $(\sqrt{t} + \|x-y\| + n^{-1})^\zeta n^\zeta \bar p_{ t}(n x, n y)$ is bounded uniformly in $n$, $t\in [0,1]$ and $x,y\in \Z/n$.\smallskip
	\item[2)] $\sum_{x\in \Z/n} \bar p_{ t}(nx,ny) \lesssim 1$\,  uniformly in $n$, $t\in [0,1]$ and $x,y\in \Z/n$.\smallskip
	\item[3)] $\sum_{y\in \Z/n} \bar p_{ t}(nx,ny) \lesssim 1$\, uniformly in $n$, $t\in [0,1]$ and $x,y\in \Z/n$.
\end{itemize}
Finally, it was shown that $\bar p$ behaves well under convolution. To properly formulate it, we denote by $\Z^k/n$ the set of all $k$-tuples with all coordinates in $\Z/n$. We say that a kernel $p$ defined on $\R_+^k\times \Z/n\times\Z/n$ is of order $\zeta\geq 0$ if
\begin{equation*}
\sup_{(t,\,x,\,y)\in \R_+^k\times\Z/n\times\Z/n}\bigg(\Big(\sum_{i=1}^k|t_i|\Big)^{\tfrac12} + \|x-y\| + n^{-1}\bigg)^\zeta n^\zeta p_t(x,y) \;<\;\infty\,.
\end{equation*}
It was then shown that if $p$ is a kernel of order $\zeta$ defined on  $\R_+\times \Z/n\times\Z/n$ and defining
\begin{equation*}
p_{t}^{\otimes, 2}(x,z)\;=\; \sum_{y\in \Z/n} p_{t_1}(x,y) p_{t_2}(y,z)
\end{equation*}
where $t=(t_1, t_2)$, then $ p^{\otimes, 2}$ is of order $2\zeta-1$ and for all $t\in \R_+^2$ and all $x,z\in\Z/n$
\begin{equation*}
\sum_{x\in \Z/n} p_{t}^{\otimes, 2}(x, z)\lesssim 1\qquad \text{ and }\qquad \sum_{z\in \Z/n} p_{t}^{\otimes, 2}(x, z)\lesssim 1
\end{equation*}
where both bounds above hold uniformly over all parameters.
Corollary 3.9 in~\cite{EH22} then shows that the same holds for higher order convolutions, i.e., one can for $\ell\geq 3$ define a kernel $p^{\otimes, \ell}$ in pretty much the same way as $p^{\otimes, 2}$ by considering the spatial convolution of $\ell$ kernels as above.
We note that the above applies to the kernel $\bar p_{\cdot}(n\cdot, n\cdot)$ which is of order $\zeta= 1$.
To conclude this section let us note that as a consequence of the results presented here we have, for instance, the estimate
\begin{equation}\label{kernelupbound}
\bar p_{ t}(x,y) \;\lesssim\; \frac{1}{\sqrt{n^2 t+1}}
\end{equation}
which holds uniformly over all parameters. The same bound (with a possibly different proportionality constant) holds for $\bar p^{\otimes, 2}$, $\bar p^{\otimes, 3}$ and higher convolutions.

\begin{remark}\label{rem:potimes}\rm
	Note that the time component of $p^{\otimes 2}$ has two coordinates $t=(t_1, t_2)$. In the same vein $p^{\otimes, \ell}$ depends on an $\ell$-dimensional vector $(t_1, t_2, \ldots, t_\ell)$. However, since our estimates will always involve the sum of the time coordinates but never its specific values, whenever we deal with $p^{\otimes,\ell}$ for some $\ell\geq 2$ we will also write $p^{\otimes, \ell}_{t_1+t_2+\cdots+ t_\ell}$ even though this is not well defined.
\end{remark}

\subsubsection{Coupling and bounds on the gradient}
Before we come to the proof of Theorem~\ref{grad} we make some preparatory remarks. If instead of $k$ exclusion particles one would consider $k$ independent random walks, then one straightforwardly obtains the following estimate:
$$\big| p_t^{\rw}(\mathbf x,\mathbf y)-p^{\rw}_t(\mathbf x+e_i,\mathbf y)\big|\;\lesssim\; \frac{1}{(n^2t+1)^{(k+1)/2}}\,,$$
where $p_t^{\rw}$ is the transition probability of independent random walks speeded up by $n^2$.
If one tries to obtain a weaker version of gradient estimate
$$\big| p_t^{\ex}(\mathbf x,\mathbf y)-p^{\ex}_t(\mathbf x+e_i,\mathbf y)\big|$$
where $p_t^{\ex}$ is the transition probability of the exclusion process with indistinguishable particles speeded up by $n^2$, it might be a good idea
to use the previous gradient estimate on the random walk and try to use a comparison between $p^{\ex}_t$ and $p^{\rw}_t$, for instance the one provided in \cite[Theorem~1]{EDA}. However, it turns out that the existing results in the literature about such a comparison are not sharp enough for our purposes, and so we need to resort to a different strategy.

The proof of Theorem \ref{grad} is based on a coupling of two exclusion processes with labelled particles. We couple two processes $\{\mathbf X_t:t\geq 0\}$ and $\{\mathbf Y_t:t\geq 0\}$, where $\{\mathbf X_t:t\geq 0\}$ is the SSEP with $k$ labelled particles starting from $\mathbf x=(x_1,\dots,x_k)\in\Lambda^k$ and $\{\mathbf Y_t:t\geq 0\}$ is the SSEP with $k$ labelled particles starting from $\mathbf x+e_i$. Both processes are speeded by $n^2$.  Denote by $\bb P^{\coup}$ the probability on the path space of the joint process $\{(\mathbf X_t,\mathbf Y_t): t\geq 0\}$ and by $\bb E^{\coup}$ the corresponding expectation.
Our construction relies on the following graphical representation of SSEP with $k+1$ labelled particles.

For each bond $(a,a+1)$, $a\in\bb Z$, we associate an i.i.d.\ Poisson clock with parameter $2n^2$. Whenever the clock rings, particles at both ends (if there are) of this bond will both move to the other end with probability $1/2$, and stay put with probability $1/2$.  For example, when the clock at bond $(a,a+1)$ rings, and there is one particle at each end, then with probability $1/2$ these two particles exchange their positions and with probability $1/2$ these two particles stay at their current positions. If  there is one particle at site $a$ and site $a+1$ is empty when the clock at bond $(a,a+1)$ rings, then with probability $1/2$ the particle at site $a$ moves to site $a+1$ and with probability $1/2$  the particle at site $a$ does not move.

Let us denote  by $\{\mathbf Z_t:t\geq 0\}$ the exclusion process corresponding to the above graphical representation, with $k+1$ labelled particles started from $(x_1,\dots,x_i,x_i+1,x_{i+1},\dots,x_k)$. For each $1\leq j\leq k+1$, let $\mathbf Z_t^j$ be the position of the $j$-th particle at time $t$. We now construct the coupling. Without loss of generality, assume that $e_i=e_1$. In the process $\{\mathbf Z_t:t\geq 0\}$, we call the particle starting from site $x_1$ the first particle, and the particle starting from site $x_1+1$ the second particle.  Then we define the coupled process $(\mathbf X_t,\mathbf Y_t)$ by
\begin{align*}
\mathbf X_t \;:=\;(\mathbf Z_t^1, \mathbf Z_t^3,\dots,\mathbf Z_t^{k+1}) \quad\text{ and }\quad
\mathbf Y_t\;:=\;(\mathbf Z_t^2, \mathbf Z_t^3,\dots,\mathbf Z_t^{k+1})\,.
\end{align*}
In other words, $\mathbf X_t$ is the process obtained from $\mathbf Z_t$ by ignoring the second particle, and $\mathbf Y_t$ is the process obtained from $\mathbf Z_t$ by ignoring the first particle. It is easy to check that both $\mathbf X_t$ and $\mathbf Y_t$ are SSEP with $k$ labelled particles.

Let $\tau$ be the first time that the first and second particle are at positions $e_-$ and $e_+$ with $|e_--e_+|=1$ and the clock attached to the edge $e=(e_-, e_+)$ rings.  Once this clock rings, by the previous graphical representation, the first and second particle may or may not swap their positions with equal probability. Therefore, on the event $\{\tau\leq t\}$, $\mathbf Z_t^1$ and $\mathbf Z_t^2$ have the same distribution, and so do $\mathbf X_t$ and $\mathbf Y_t$. This in particular implies that for every $\mathbf z,\mathbf w\in\Lambda^k$,
\begin{equation}\label{swap}
\bb E^{\coup} \Big[ \one\{\tau\leq t\} \big[\one\{(\mathbf X_t,\mathbf Y_t)=(\mathbf z,\mathbf w)\} -\one\{(\mathbf X_t,\mathbf Y_t)=(\mathbf w,\mathbf z)\}\big] \Big]\;=\;0\,.
\end{equation}

The next lemma gives an upper bound on the gradient in terms of a probability involving the stopping time $\tau$.
\begin{lemma}\label{lemmacoup}
	There exists a constant $C$ independent of $x,y,t$ and $e_i$ such that
	$$\big| p_{2t}^{\lex}(\mathbf x,\mathbf y)-p^{\lex}_{2t}(\mathbf x+e_i,\mathbf y)\big|\;\leq\; \frac{C}{(n^2 t+1)^{k/2}}\,\bb P^{\coup}(\tau>t)\,.$$
\end{lemma}
\begin{proof}
	Note that
	\begin{equation*}
	\begin{split}
	\big| p_{2t}^{\lex}(\mathbf x,\mathbf y)-p^{\lex}_{2t}(\mathbf x+e_i,\mathbf y)\big|\;=\;&\Big|\sum_{\mathbf z\in\Lambda^k}p_t^{\lex}(\mathbf x,\mathbf z)p_t^{\lex}(\mathbf z,\mathbf y)-p_t^{\lex}(\mathbf x+e_i,\mathbf z)p_t^{\lex}(\mathbf z,\mathbf y)\Big|\\
	\leq\,&\sum_{\mathbf z\in\Lambda^k}\big |p_t^{\lex}(\mathbf x,\mathbf z)-p_t^{\lex}(\mathbf x+e_i,\mathbf z)\big| \,p_t^{\lex}(\mathbf z,\mathbf y)\,.\\
	\end{split}
	\end{equation*}
	By \cite[Theorem 1.1]{Landim05}, we know that  $p_t^{\lex}(\mathbf z,\mathbf y)$ is uniformly bounded by $(1+n^2t)^{-k/2}$. On the other hand, let $p_t^{\coup}$ be the transition probability of the coupled process that was defined before this lemma. Then
	\begin{align*}
	&\sum_{\mathbf z}\big |p_t^{\lex}(\mathbf x,\mathbf z)-p_t^{\lex}(\mathbf x+e_i,\mathbf z)\big|\\
	&=\;\sum_{\mathbf z}\big| \sum_{\mathbf w}\big[p_t^{\coup}\big((\mathbf x,\mathbf x+e_i),(\mathbf z,\mathbf w) -p_t^{\coup}\big((\mathbf x,\mathbf x+e_i),(\mathbf w,\mathbf z) \big)\big] \big|\\
	&\leq\; \sum_{\mathbf z\neq \mathbf w}\big| p_t^{\coup}\big((\mathbf x,\mathbf x+e_i),(\mathbf z,\mathbf w) -p_t^{\coup}\big((\mathbf x,\mathbf x+e_i),(\mathbf w,\mathbf z)\big)\big|\\
	&=\; \sum_{\mathbf z\neq \mathbf w}\big| \bb E^{\coup}\big[\one\{(\mathbf X_t,\mathbf Y_t)=(\mathbf z,\mathbf w)\} -\one\{(\mathbf X_t,\mathbf Y_t)=(\mathbf w,\mathbf z)\}\big] \big|.
	\end{align*}
	By \eqref{swap}, the previous expression is equal to
	\begin{align*}
	& \sum_{\mathbf z\neq \mathbf w}\Big| \bb E^{\coup}\Big[ \one\{\tau>t\} \big[\one\{(\mathbf X_t,\mathbf Y_t)=(\mathbf z,\mathbf w)\} -\one\{(\mathbf X_t,\mathbf Y_t)=(\mathbf w,\mathbf z)\}\big] \Big] \Big|\\
	& \leq\; \sum_{\mathbf z\neq \mathbf w}\Big| \bb E^{\coup} \Big[ \one\{\tau>t\} \big[\one\{(\mathbf X_t,\mathbf Y_t)=(\mathbf z,\mathbf w)\} +\one\{(\mathbf X_t,\mathbf Y_t)=(\mathbf w,\mathbf z)\}\big] \Big] \Big|\\
	& \leq\; 2\sum_{\mathbf z\neq \mathbf w}\bb P^{\coup}\big(\tau>t, (\mathbf X_t,\mathbf Y_t)\neq(\mathbf z,\mathbf w)\big)\\
	& \leq\; 2\,\bb P^{\coup}(\tau>t)\,.
	\end{align*}

\end{proof}

Theorem \ref{grad} is a simple consequence of Lemma \ref{lemmacoup}, once we can show that
\begin{equation}\label{coup}
\bb P^{\coup}(\tau>t)\;\lesssim\; \frac{1}{\sqrt{n^2 t+1}}\,.
\end{equation}

Note that from the graphical representation, the motion of every particle is determined only by the Poisson clock on each bond, but not effected by the other particles. Therefore, when dealing with $\tau$, we can focus on the first and second particle and ignore the presence of the other particles.
By definition of the stopping time $\tau$, it is not hard to see that before $\tau$, the first particle and the second particle perform accelerated independent simple random walks. Denote by $\tilde{\bb P}$ the probability on the path space of two independent random walks, and by $\tilde{\tau}$ the first meeting time of these two random walks. The key observation is that the events $\{\tau> t\}$ and $\{\tilde{\tau}> t\}$ have the same distribution. Hence we have
$$\bb P^{\coup}(\tau>t)\;=\;\tilde{\bb P}(\tilde{\tau}>t)\;\lesssim\;  \frac{1}{\sqrt{n^2 t+1}}\,. $$
This finishes the proof of Theorem~\ref{grad}.

\subsection{Estimates on integral terms}\label{44}
From \eqref{eq:upperexclusion} we can get a trivial bound of the gradient:
$$\big| p_t^{\lex}(\mathbf x,\mathbf y)-p^{\lex}_t(\mathbf x+e_j,\mathbf y)\big|\;\lesssim\;  \frac{C_1}{(1+n^2t)^{k/2}}  \exp\Big\{-\frac{C_2 n^2t}{2(\log n^2t)^2}\Phi\Big(\frac{\|\mathbf x-\mathbf y\| \log n^2t}{C_2^2 n^2t}\Big)\Big\}$$
for every $t>0$, every $\mathbf x,\mathbf y\in \bb Z^k$, every $1\leq j\leq k$ such that $\mathbf x+e_j\in\Lambda^k$.
Combining the bound in Theorem \ref{grad} and using the geometric interpolation bound
$$\min\{a,b\} \;\leq\; \sqrt{ab}, \quad  \forall \,\, a,b>0\,,$$
we obtain
\begin{equation}\label{gradbound}
\begin{split}
\big| p_t^{\lex}(\mathbf x,\mathbf y)-p^{\lex}_t(\mathbf x+e_j,\mathbf y)\big|\;\lesssim\; &\frac{1}{(\sqrt{n^2t+1})^{k+1/2}} \exp\Big\{-\frac{C_2 n^2t}{4(\log n^2t)^2}\Phi\Big(\frac{\|\mathbf x-\mathbf y\| \log n^2t}{C_2^2 n^2t}\Big)\Big\}\\
\lesssim\;& \frac{1}{(1+n^2t)^{1/4}} \prod_{i=1}^{k}\bar p_t(x_i,y_i)
\end{split}
\end{equation}
when $n^2t>C_1$ and an analogous inequality holds for $n^2t\leq C_1$.
Note that at this point that the definition of $\bar p$ is actually slightly different from the one in Section~\ref{Sec431}. Namely, because of the geometric interpolation bound we actually need to consider the square root of the exponential term in~\eqref{eq:largetbarp}. This of course only changes the constants inside the exponential but none of its properties. We therefore continue using the notation $\bar p$.

With this bound, we are now able to estimate the second expectation on the right hand side of \eqref{indeq} in Theorem \ref{twotimecf}. Recall the definition of $\varphi$ in \eqref{defvp} and the definition of $\Psi$ in \eqref{defPsi} in the proof of Theorem \ref{twotimecf}. Recall also that before giving these definitions, we have fixed a collection of points $(z_j: 1\leq j\leq k_2)$ and a time $s$.
\begin{proposition}\label{estimatePsi}
	There exists a constant $C$ independent of $(x_i: 1\leq i\leq k_1)$ and $(z_j: 1\leq j\leq k_2)$  such that
	\begin{equation*}
	\Big\lvert \mathbf{E}_{(x_i: 1\leq i\leq k_1)}\Big[ \int_s^t \dd r\, \Psi(\mathbf{X}_{t-r},r) \Big] \Big\rvert \;\leq\; Cn^{-(k_1+k_2)/2}\times \Big(\frac{n}{\sqrt{n^2(t-s)+1}}\Big)^{k_1\wedge k_2},
	\end{equation*}
	for any $s<t$ with $0\leq s,t\leq T$ and every integer $n\geq 1$.
\end{proposition}
We first prove a technical lemma which will be very useful in the proof of this proposition. To that end we introduce more notation.
Let $\bar p_t^{\otimes,\ell}:\bb Z\times\bb Z\to\bb R$ be the convolution which is defined inductively by
\begin{equation}\label{conv}
\bar p_{t_1+t_2}^{\otimes,\ell}(x,z)\;=\; \sum_{y\in\bb Z}  \bar p_{t_1}^{\otimes,\ell-1}(x,y) \bar p_{t_2}(y,z)\,,\quad \forall\, t_1,t_2>0\,.
\end{equation}
Note that $\bar p^{\otimes,\ell}$ defined in this way is actually not well defined. We refer to Remark~\ref{rem:potimes} for an explanation on that.
With the convention $\bar p_t^{\otimes,1}\;=\;\bar p_t$, recall that  we showed in Section \ref{Sec431} that
\begin{equation}\label{finitemeasurebarp}
\sup_{0\leq t\leq T}\sup_{x}\sum_{z} \bar p_{t}^{\otimes,\ell}(x,z) \;\lesssim\; 1\,.
\end{equation}
For every $\mathbf x=(x_1,\dots,x_k)\in\Lambda^k$ and $\mathbf y=(y_1,\dots,y_k)\in\Lambda^k$, let
\begin{equation*}
\tilde{p}_{t}^{\otimes,\ell}(\mathbf x,\mathbf y)\;=\;\prod_{i=1}^k\bar p_{t}^{\otimes,\ell}(x_i,y_i)\,.
\end{equation*}
As a consequence of \eqref{finitemeasurebarp},
\begin{equation}\label{finitemeasuretildep}
\sup_{0\leq t\leq T}\sup_{\mathbf x}\sum_{\mathbf z} \tilde p_{t}^{\otimes,\ell}(\mathbf x,\mathbf z) \;\lesssim\; 1\,.
\end{equation}
The following quantity plays a crucial role in the proof of Proposition \ref{estimatePsi}:
$$S_{t,r}^{k,\ell}\;:=\; \sup_{\mathbf x\in\Lambda^k} \sum_{\mathbf y\in\Lambda^k} \tilde{p}_{t-r}^{\otimes,\ell}(\mathbf x,\mathbf y) \big| \varphi_r(\mathbf y)\big|\,,$$
for every non-negative integer $k$ and $s\leq r<t$. In the case $k=0$,  $\varphi_r(\mathbf y)$ is simply given by
$$ \bb E_{\nu_{\rho_0^n(\cdot)}}\Big[  \prod_{j=1}^{k_2}\overline{\eta}_s(z_j) \Big].$$
Using the bound \eqref{samet} and property \eqref{finitemeasuretildep}, we note that
\begin{equation}\label{eq:seed}
S_{t,r}^{0,\ell} \;\leq\; C(\ell)n^{-k_2/2}\,.
\end{equation}

\begin{lemma}\label{difbound}
	For every non-negative integers $k$ and $\ell$, there exists a constant $C>0$ independent of $(z_j: 1\leq j\leq k_2)$, such that for every $0\leq r< s < t$ with $0\leq s,t\leq T$,
	\begin{equation*}
	S_{t,r}^{k,\ell}\;\leq\; C n^{-(k+k_2)/2}\bigg(\frac{n}{\sqrt{n^2(t-s)+1}}\bigg)^{k\wedge k_2}\frac{(n^2(t-s)+1)^{1/4}}{(n^2(r-s)+1)^{1/4}}\,.
	\end{equation*}
\end{lemma}

The idea of the proof is to perform induction on $k$, i.e., we bound $S_{t,r}^{k,\ell}$ in terms of $S_{t,r}^{k-2,\ell+1}$ and $S_{t,r}^{k-1,\ell+1}$ and use a seed bound on $S_{t,r}^{0,\ell}$ to conclude.
\begin{proof}[Proof of Lemma \ref{difbound}]

	The case $k=0$ holds trivially by~\eqref{eq:seed}.
	Assume $k\geq 1$. By \eqref{indeq}, for every $\mathbf x\in\Lambda^k$, we have
	$$\sum_{\mathbf y\in\Lambda^k} \tilde{p}_{t-r}^{\otimes,\ell}(\mathbf x,\mathbf y) |\varphi_r(\mathbf y)|\;\leq\; A_1+A_2+A_3\,,$$
	where
	\begin{equation*}
	A_1\;=\;\sum_{\mathbf y\in\Lambda^k} \tilde{p}_{t-r}^{\otimes,\ell}(\mathbf x,\mathbf y)\mathbf E_{\mathbf y}\big[\big\lvert \varphi_s(\mathbf X_{r-s}) \big\rvert\big]\,,
	\end{equation*}
	\begin{equation*}
	\begin{split}
	A_2\;=\;2n^2 \sum_{\mathbf  y\in\Lambda^k} \tilde{p}_{t-r}^{\otimes,\ell}(\mathbf x,\mathbf y) \Big\lvert\mathbf  E_{\mathbf y}\Big[ \int_s^r \dd\tau &\sum_{i,j} \one_{\{|\mathbf X_{r-\tau}^i -\mathbf X_{r-\tau}^j |=1 \}}\\
	\times &\big[\rho_\tau^n(\mathbf X_{r-\tau}^i)-\rho_\tau^n(\mathbf X_{r-\tau}^j)  \big]  \Big[ \tilde{\varphi_\tau}(\mathbf X_{r-\tau}^i)- \tilde{\varphi_\tau}(\mathbf X_{r-\tau}^j) \Big] \Big]\Big\rvert\,,
	\end{split}
	\end{equation*}
	and
	\begin{equation*}
	\begin{split}
	A_3\;=\;n^2 \sum_{\mathbf y\in\Lambda^k} \tilde{p}_{t-r}^{\otimes,\ell}(\mathbf x,\mathbf y) \mathbf E_{\mathbf y}\Big[ \int_s^r \dd\tau\, &\sum_{i,j} \one_{\{|\mathbf X_{r-\tau}^i -\mathbf X_{r-\tau}^j |=1 \}}\\
	\times&\big\lvert \rho_\tau^n(\mathbf X_{r-\tau }^i)-\rho_\tau^n(\mathbf X_{r-\tau }^j)  \big\rvert^2  \Big\lvert \tilde{\varphi_\tau}(\mathbf X_{r-\tau }^i, \mathbf X_{r-\tau }^j)\Big\rvert \Big].
	\end{split}
	\end{equation*}
	In the above formulas, recall that $\mathbf X_{t}$ denotes the vector of positions of the collection of exclusion walkers at time $n^2 t$, $\mathbf X^i_{t}$ denotes the position of the $i$-th particle at time $n^2 t$, and $\mathbf E_{\bf y}$ the probability with respect to process $\{\mathbf X_t: t\geq 0\}$ starting from ${\bf y}=(y_1,\dots,y_k)\in\Lambda^k$.
	We note that $A_2$ and $A_3$ are actually zero unless $k\geq 2$.

	For the first term we have that
	$$A_1\;\lesssim\; n^{-(k+k_2)/2}\Big(\frac{n}{\sqrt{n^2(t-s)+1}}\Big)^{k\wedge k_2}$$
	by Lemma \ref{initialq} and Remark \ref{condtildep}.
	We now estimate $A_3$. Since $\rho_0$ is continuously differentiable with bounded derivative,
	\begin{equation}\label{eq:rhodif1}
	\sup_{\substack{0\leq \tau\leq T\\ |\mathbf x-\mathbf y|=1}}\big| \rho_\tau^n(\mathbf x)-\rho_\tau^n(\mathbf y)\big|^2\;=\;O(n^{-2})\,,
	\end{equation}
	thus
	\begin{align*}
	A_3\;\lesssim\;&\int_s^r \dd\tau \sum_{i,j}\sum_{\mathbf y\in\Lambda^{k}} \tilde{p}^{\otimes,\ell}_{t-r}(\mathbf x,\mathbf y)\sum_{\mathbf w: |w_i-w_j|=1} p^{\lex}_{r-\tau}(\mathbf y,\mathbf w)\lvert \tilde{\varphi}_\tau(w_i,w_j)\rvert \\
	\lesssim\,&\int_s^r \dd\tau \sum_{i,j}\sum_{\mathbf y\in\Lambda^{k}} \tilde{p}_{t-r}^{\otimes,\ell}(\mathbf x,\mathbf y)\sum_{\mathbf w: |w_i-w_j|=1} \tilde{p}_{r-\tau}(\mathbf y,\mathbf w)\lvert \tilde{\varphi}_\tau(w_i,w_j)\rvert \\
	=\,&\int_s^r \dd\tau \sum_{i,j} \sum_{\mathbf w: |w_i-w_j|=1} \tilde{p}_{t-\tau}^{\otimes,\ell+1}(\mathbf x,\mathbf w)\lvert \tilde{\varphi}_\tau(w_i,w_j)\rvert\,.
	\end{align*}
	By the explicit expression of $\tilde{p}_t^{\otimes,\ell+1}$, we have
	\begin{align*}
	A_3\;\lesssim\;&\int_s^r\dd\tau \sum_{i,j} \sum_{\mathbf w: |w_i-w_j|=1} \prod_l \bar p^{\otimes,\ell+1}_{t-\tau}(x_l,w_l)\lvert \tilde{\varphi}_\tau(w_i,w_j)\rvert \\
	=\;&\int_s^r \dd\tau \sum_{i,j} \sum_{\mathbf w: |w_i-w_j|=1}  \Big( \prod_{l\neq i,j} \bar p^{\otimes,\ell+1}_{t-\tau}(x_l,w_l)\Big)\Big(  \bar p^{\otimes,\ell+1}_{t-\tau}(x_i,w_i) \bar p^{\otimes,\ell+1}_{t-\tau}(x_j,w_j)\Big)  \lvert \tilde{\varphi}_\tau(w_i,w_j)\rvert\,.
	\end{align*}
	Note that $\bar p^{\otimes,\ell+1}$ is translation invariant and symmetric, therefore
	\begin{align*}
	&\sum_{w_i,w_j} 1_{\{|w_i-w_j|=1\}}\bar p^{\otimes,\ell+1}_{t-\tau}(x_i,w_i) \bar p^{\otimes,\ell+1}_{t-\tau}(x_j,w_j)\\
	&=\;\sum_{w_i} \bar p^{\otimes,\ell+1}_{t-\tau}(x_i,w_i) \bar p^{\otimes,\ell+1}_{t-\tau}(x_j,w_i+1)+\sum_{w_i} \bar p^{\otimes,\ell+1}_{t-\tau}(x_i,w_i) \bar p^{\otimes,\ell+1}_{t-\tau}(x_j,w_i-1)\\
	& =\;\sum_{w_i} \bar p^{\otimes,\ell+1}_{t-\tau}(x_i,w_i) \bar p^{\otimes,\ell+1}_{t-\tau}(x_j-1,w_i)+\sum_{w_i} \bar p^{\otimes,\ell+1}_{t-\tau}(x_i,w_i) \bar p^{\otimes,\ell+1}_{t-\tau}(x_j+1,w_i)\\
	& \lesssim\; \sum_{w_i} \bar p^{\otimes,2(\ell+1)}_{2(t-\tau)}(x_i,x_j-1) +\sum_{w_i} \bar p^{\otimes,2(\ell+1)}_{2(t-\tau)}(x_i,x_j+1) \\
	& \lesssim\; \frac{1}{\sqrt{n^2(t-\tau)+1}}\,.
	\end{align*}
	Here, the last bound follows from~\eqref{kernelupbound} and its following comments.
	On the other hand, by the explicit expression of $\tilde p_t^{\otimes,\ell+1}$ and the definition of $S_{t,\tau}^{k-2, \ell+1}$,
	$$\prod_{l\neq i,j} \bar p^{\otimes,\ell+1}_{t-\tau}(x_l,w_l)  \lvert\tilde{\varphi}_\tau(w_i,w_j)\rvert \;\leq\; S_{t,\tau}^{k-2, \ell+1}.$$
	Putting all the estimates together, we obtain
	\begin{equation}\label{A3}
	A_3\;\lesssim\; \int_s^r \dd\tau \frac{1}{\sqrt{n^2(t-\tau)+1}} S_{t,\tau}^{k-2,\ell+1} \,.
	\end{equation}

	We now turn to the estimate of $A_2$.
	\begin{equation*}
	\begin{split}
	A_2\;\lesssim\; &  n^2\int_s^r \dd\tau \sum_{i,j} \sum_{\mathbf y\in\Lambda^k} \tilde{p}_{t-r}^{\otimes,\ell}(\mathbf x,\mathbf y) \Big\lvert \!\!\! \sum_{\mathbf w:|w_i-w_j|=1}\!\!\!\!\!p^{\lex}_{r-\tau}(\mathbf y,\mathbf w) \big[\rho_\tau^n(w_i)-\rho_\tau^n(w_j)  \big]  \Big[ \tilde{\varphi_\tau}(w_i)- \tilde{\varphi_\tau}(w_j) \Big]   \Big\rvert\,.
	\end{split}
	\end{equation*}
	A summation by parts shows that the expression inside the absolute value sign is equal to the sum of
	$$\sum_{\mathbf w:w_j=w_i+1} \tilde{\varphi_\tau}(w_i)\Big[ p^{\lex}_{r-\tau}(\mathbf y,\mathbf w) \big[\rho_\tau^n(w_i)-\rho_\tau^n(w_i+1)  \big] -p^{\lex}_{r-\tau}(\mathbf y,\mathbf w-e_i) \big[\rho_\tau^n(w_i-1)-\rho_\tau^n(w_i)  \big]  \Big]\,,  $$
	and
	$$\sum_{\mathbf w:w_j=w_i-1} \tilde{\varphi_\tau}(w_i)\Big[ p^{\lex}_{r-\tau}(\mathbf y,\mathbf w) \big[\rho_\tau^n(w_i)-\rho_\tau^n(w_i-1)  \big] -p^{\lex}_{r-\tau}(\mathbf y,\mathbf w+e_i) \big[\rho_\tau^n(w_i+1)-\rho_\tau^n(w_i)  \big]  \Big].$$
	Observe that, since the second derivative of $\rho_0$ is bounded $$ \big[\rho_\tau^n(w_i)-\rho_\tau^n(w_i+1)  \big] -  \big[\rho_\tau^n(w_i-1)-\rho_\tau^n(w_i)  \big] \;=\;O(n^{-2})\,.$$
	From this, we can deduce that
	\begin{align*}
	&\Big\lvert  \sum_{\mathbf w:|w_i-w_j|=1}p^{\lex}_{r-\tau}(\mathbf y,\mathbf w) \big[\rho_\tau^n(w_i)-\rho_\tau^n(w_j)  \big]  \Big[ \tilde{\varphi_\tau}(w_i)- \tilde{\varphi_\tau}(w_j) \Big]   \Big\rvert\\
	&\leq\; \Big\lvert  \sum_{\mathbf w:w_j=w_i+1} \tilde{\varphi_\tau}(w_i)\big[ p^{\lex}_{r-\tau}(\mathbf y,\mathbf w) -p^{\lex}_{r-\tau}(\mathbf y,\mathbf w-e_i) \big] \big[\rho_\tau^n(w_i-1)-\rho_\tau^n(w_i)  \big]   \Big\rvert\\
	&+\Big\lvert \sum_{\mathbf w:w_j=w_i-1} \tilde{\varphi_\tau}(w_i)\big[ p^{\lex}_{r-\tau}(\mathbf y,\mathbf w) -p^{\lex}_{r-\tau}(\mathbf y,\mathbf w+e_i) \big] \big[\rho_\tau^n(w_i+1)-\rho_\tau^n(w_i)  \big]  \Big\rvert\\
	&+Cn^{-2}\Big\lvert \sum_{\mathbf w:|w_i-w_j|=1} \tilde{\varphi_\tau}(w_i)p^{\lex}_{r-\tau}(\mathbf y,\mathbf w) \Big\rvert\,.
	\end{align*}
	By \eqref{gradbound}, \eqref{boundhk}, and \eqref{eq:rhodif1} the previous expression is bounded by
	\begin{align*}
	& C\Big(\frac{n^{-1}}{(n^2(r-\tau)+1)^{1/4}}+\frac{1}{n^2}\Big) \sum_{\mathbf w:|w_i-w_j|=1} \big\lvert \tilde{\varphi_\tau}(w_i) \big\rvert p^{\lex}_{r-\tau}(\mathbf y,\mathbf w) \\
	&\lesssim\;\frac{n^{-1}}{(n^2(r-\tau)+1)^{1/4}} \sum_{\mathbf w:|w_i-w_j|=1}\lvert \tilde{\varphi_\tau}(w_i) \rvert \tilde{p}_{r-\tau}(\mathbf y,\mathbf w)\,.
	\end{align*}

	Based on the above estimates we have
	\begin{align*}
	A_2\;\lesssim\; &n\int_s^r \dd\tau \sum_{i,j} \sum_{\mathbf y\in\Lambda^k} \tilde{p}_{t-r}^{\otimes,\ell}(\mathbf x,\mathbf y) \frac{1}{(n^2(r-\tau)+1)^{1/4}} \sum_{\mathbf w:|w_i-w_j|=1}\lvert \tilde{\varphi_\tau}(w_i) \rvert \tilde{p}_{r-\tau}(\mathbf y,\mathbf w)\\
	\lesssim\,&n\int_s^r \dd\tau \frac{1}{(n^2(r-\tau)+1)^{1/4}} \sum_{i,j} \sum_{\mathbf y} \sum_{\mathbf w:|w_i-w_j|=1} \Big(\prod_{l\neq i} \bar p^{\otimes,\ell}_{t-r}(x_l,y_l)\bar p_{r-\tau}(y_l,w_l)\lvert \tilde{\varphi_\tau}(w_i) \rvert\Big) \times\\
	& \times \bar p^{\otimes,\ell}_{t-r}(x_i,y_i)\bar p_{r-\tau}(y_i,w_i)\\
	\lesssim\,&n\int_s^r \dd\tau \frac{1}{(n^2(r-\tau)+1)^{1/4}} \sum_{i,j} \sum_{\mathbf w:|w_i-w_j|=1}\Big(\prod_{l\neq i}\bar p^{\otimes,\ell+1}_{t-\tau}(x_l,w_l)|\tilde{\varphi_\tau}(w_i)|\Big) \bar p^{\otimes,\ell+1}_{t-\tau}(x_i,w_i)\,.
	\end{align*}
	Note that
	$$\sum_{\mathbf w:|w_i-w_j|=1} \prod_{l\neq i}\bar p^{\otimes,\ell+1}_{t-\tau}(x_l,w_l)|\tilde{\varphi_\tau}(w_i)| \;\lesssim\; S_{t,\tau}^{k-1,\ell+1}\,.$$
	Since by \eqref{kernelupbound},
	$$ \sup_{x,w\in\bb Z} \bar p^{\otimes,\ell+1}_{t}(x,w)\;\lesssim\;\frac{1}{\sqrt{n^2 t+1}}\,,$$
	we finally obtain
	\begin{equation}\label{A2}
	A_2\;\lesssim\;n\int_s^r \dd\tau \frac{1}{(n^2(r-\tau)+1)^{1/4}}  S_{t,\tau}^{k-1,\ell+1} \frac{1}{\sqrt{n^2(t-\tau)+1}}\,.
	\end{equation}

	It remains to bound $A_2$ and $A_3$ using \eqref{A2} and \eqref{A3}, respectively, and our induction assumption. Since these computations are quite similar, we only show how to prove
	\begin{equation}
	\label{A2bound}
	A_2\;\lesssim\; n^{-(k+k_2)/2}\Big(\frac{n}{\sqrt{n^2(t-s)+1}}\Big)^{k\wedge k_2}\frac{(n^2(t-s)+1)^{1/4}}{(n^2(r-s)+1)^{1/4}}\,.
	\end{equation}
	Applying the induction assumption on $S_{t,\tau}^{k-1,\ell+1}$, namely,
	\begin{equation*}
	S_{t,\tau}^{k-1,\ell+1}\;\lesssim\; n^{-(k-1+k_2)/2}\Big(\frac{n}{\sqrt{n^2(t-s)+1}}\Big)^{(k-1)\wedge k_2}\frac{(n^2(t-s)+1)^{1/4}}{(n^2(\tau-s)+1)^{1/4}}\,,
	\end{equation*}
	cancelling common terms on both sides of the above inequality, and using that
	\begin{equation*}
	\frac{n}{\sqrt{n^2(t-s)+1}}\;\gtrsim\; 1\,,
	\end{equation*}
	we see that it is enough to show that
	$$ \int_s^r \dd\tau \frac{1}{(n^2(r-\tau)+1)^{1/4}}  \frac{1}{(n^2(\tau-s)+1)^{1/4}} \frac{1}{\sqrt{n^2(t-\tau)+1}}\;\lesssim\; \frac{n^{-3/2}}{(n^2(r-s)+1)^{1/4}}\,.$$
	The expression on the left hand side is bounded by
	\begin{align*}
	&\frac{1}{\sqrt{n^2(t-\frac{r+s}{2})+1}} \int_s^{(r+s)/2} \dd\tau\, \frac{1}{(n^2(r-\tau)+1)^{1/4}}  \frac{1}{(n^2(\tau-s)+1)^{1/4}}\\
	&+\frac{1}{(n^2((\frac{r+s}{2}-s)+1)^{1/4}} \int_{(s+r)/2}^r \dd\tau\, \frac{1}{(n^2(r-\tau)+1)^{1/4}} \frac{1}{\sqrt{n^2(t-\tau)+1}}
	\end{align*}
	which is less than or equal to
	\begin{align*}
	&\frac{1}{\sqrt{n^2(t-\frac{r+s}{2})+1}} \frac{1}{(n^2(r-s)+1)^{1/4}}\frac{1}{n^2}(n^2(r-s)+1)^{3/4}\\
	&+\frac{1}{(n^2((\frac{r+s}{2}-s)+1)^{1/4}} \frac{1}{n^2} (n^2(r-\frac{r+s}{2})+1)^{1/4}\,.
	\end{align*}
	It is not hard to check the  inequalities
	$$ \frac{1}{\sqrt{n^2(t-\frac{r+s}{2})+1}}(n^2(r-s)+1)^{3/4}\;\lesssim\;\sqrt{n}$$
	and
	$$ \Big(n^2\big(\frac{r-s}{2}\big)+1\Big)^{1/4}\;\lesssim\; \sqrt{n}\,.$$
	With these two estimates,  we conclude the proof of \eqref{A2bound}.
\end{proof}

We now show how to bound
$\mathbf{E}_{(x_i\,:\, 1\leq i\leq k_1)}\Big[ \int_s^t \dd r\, \Psi(\mathbf{X}_{t-r},r) \Big]$
with the help of Lemma \ref{difbound}.
\begin{proof}[Proof of Proposition \ref{estimatePsi}]
	Recall the explicit formula of $\Psi$ given in \eqref{defPsi}. We have
	$$\Big\lvert \mathbf{E}_{(x_i\,:\, 1\leq i\leq k_1)}\Big[ \int_s^t\dd r\, \Psi(\mathbf{X}_{t-r},r) \Big] \Big\rvert \;\lesssim\; A_4+A_5\,,$$
	where
	\begin{equation*}
	A_4= 2n^2  \Big\lvert \mathbf{E}_{\mathbf x}\Big[ \int_s^t \dd\tau \sum_{i,j} 1_{\{|\mathbf X_{t-\tau}^i -\mathbf X_{t-\tau}^j |=1 \}} \big[\rho_\tau^n(\mathbf X_{t-\tau}^i)-\rho_\tau^n(\mathbf X_{t-\tau}^j)  \big]  \Big[ \tilde{\varphi_\tau}(\mathbf X_{t-\tau}^i)- \tilde{\varphi_\tau}(\mathbf X_{t-\tau}^j) \Big] \Big]\Big\rvert\,,
	\end{equation*}
	and
	\begin{equation*}
	A_5=n^2 \mathbf{E}_{\mathbf x}\Big[ \int_s^t \dd\tau\sum_{i,j} 1_{\{|\mathbf X_{t-\tau}^i -\mathbf X_{t-\tau}^j |=1 \}} \big\lvert \rho_\tau^n(\mathbf X_{t-\tau}^i)-\rho_\tau^n(\mathbf X_{t-\tau}^j)  \big\rvert^2  \Big\lvert \tilde{\varphi_\tau}(\mathbf X_{t-\tau}^i, \mathbf X_{t-\tau}^j)\Big\rvert \Big]\,.
	\end{equation*}
	Since $\rho_0(\cdot)$ has bounded first derivative,
	$$\sup_{0\leq t\leq T}\sup_{|x-y|=1}\big\lvert \rho_t^n(x)-\rho^n_t(y) \big\rvert^2 \;\leq\; n^{-2}\,,$$
	which gives us
	\begin{equation*}
	\begin{split}
	A_5 \;\lesssim\;&\int_s^t \dd\tau \sum_{i,j}\sum_{\mathbf w: |w_i-w_j|=1} p^{\lex}_{t-\tau}(\mathbf x,\mathbf w)\lvert \tilde{\varphi}_\tau(w_i,w_j)\rvert \\
	\lesssim\;&\int_s^t \dd\tau \sum_{i,j}\sum_{\mathbf w: |w_i-w_j|=1} \tilde{p}_{t-\tau}(\mathbf x,\mathbf w)\lvert \tilde{\varphi}_\tau(w_i,w_j)\rvert \\
	=\;&\int_s^t \dd\tau \sum_{i,j} \sum_{\mathbf w: |w_i-w_j|=1} \prod_\ell \bar p_{t-\tau}(x_\ell,w_\ell)\lvert \tilde{\varphi}_\tau(w_i,w_j)\rvert \\
	\end{split}
	\end{equation*}
	by \eqref{boundhk} and the explicit formula of $\tilde{p}_{t-\tau}(x,w)$.
	We have dealt with a similar (actually an even more complicated) expression while estimating $A_3$. It is not hard to see that
	$$A_5 \;\lesssim\; \int_s^t \dd\tau\, \frac{1}{\sqrt{n^2(t-\tau)+1}} S_{t,\tau}^{k_1-2,1} \,.$$

	On the other hand, $A_4$ is bounded by some proportionality constant times
	\begin{equation*}
	\begin{split}
	&  n^2\int_s^t \dd\tau \sum_{i,j}  \Big\lvert  \sum_{\mathbf w:|w_i-w_j|=1}p^{\lex}_{t-\tau}(\mathbf x,\mathbf w) \big[\rho_\tau^n(w_i)-\rho_\tau^n(w_j)  \big]  \Big[ \tilde{\varphi_\tau}(w_i)- \tilde{\varphi_\tau}(w_j) \Big]   \Big\rvert\\
	&\lesssim \; n\int_s^t \dd\tau \sum_{i,j}  \frac{1}{(n^2(t-\tau)+1)^{1/4}} \sum_{\mathbf w:w_j=w_i\pm1}\lvert \tilde{\varphi_\tau}(w_i) \rvert \tilde{p}_{t-\tau}(x,w)\,.
	\end{split}
	\end{equation*}
	Here, we used similar arguments as those used for $A_2$ in Lemma~\ref{difbound}. Continuing adapting the arguments already employed in the analysis of $A_2$, we obtain
	$$A_4\;\lesssim\;n\int_s^t \dd\tau\, \frac{1}{(n^2(t-\tau)+1)^{1/4}}  S_{t,\tau}^{k_1-1, 1} \frac{1}{\sqrt{n^2(t-\tau)+1}}\,.$$

	In view of Lemma \ref{difbound}, it is elementary to check that $A_4$ and $A_5$ are both bounded by
	$$Cn^{-(k_1+k_2)/2}\times \Big(\frac{n}{\sqrt{n^2(t-s)+1}}\Big)^{k_1\wedge k_2}.$$
\end{proof}

For the proof of Theorem \ref{gnE} we need to extend Proposition \ref{estimatePsi} to the case of multiple times. To that end recall the definition of $\phi$ given in \eqref{defphi}.
\begin{proposition}\label{estimatePhi}
	Consider $m$ lists of non-repetitive points $(x_{i_1}: 1\leq i_1\leq k_1),\dots,(x_{i_m}: 1\leq i_m\leq k_m)$. Assume $t_1<\cdots <t_m$ and $0\leq t_j,t_{j+1}\leq T$ for every $1\leq j< m$.  Then, there exists a constant $C$ independent of $(x_{i_j}: 1\leq i_j\leq k_j)$, $1\leq j\leq m$, such that
	\begin{equation*}
	\begin{split}
	&\Big\lvert \mathbf{E}_{(x_{i_m}:\,1\leq i_m\leq k_m)}\big[\phi_{t_{m-1}}(\mathbf{X}^{i_m}_{t_m-t_{m-1}}: 1\leq i_m\leq k_m) \big] \Big\rvert \\
	&\leq\; Cn^{-\sum_{j=1}^m k_j/2} \prod_{j=1}^{m-1} \Big(\frac{n}{\sqrt{n^2(t_{j+1}-t_j)+1}}\Big)^{k_j\wedge \sum_{l=j+1}^m k_l}
	\end{split}
	\end{equation*}
	for every integer $n\geq 1$.
\end{proposition}
The proof is quite similar to the proof of Proposition \ref{estimatePsi}, so we only sketch it. The key step is to estimate the following quantity, which plays a similar role as $S_{t,r}^{k,\ell}$ in the proof of Proposition \ref{estimatePsi}:
$$\mc S_{t,r}^{k,\ell}\;:=\; \sup_{\mathbf x\in\Lambda^k} \sum_{\mathbf y\in\Lambda^k} \tilde{p}_{t-r}^{\otimes,\ell}(\mathbf x,\mathbf y) \big| \phi_r(\mathbf y)\big|\,.$$
One can mimic the proof of Lemma \ref{difbound} to obtain
\begin{align*}
\mc S_{t_m,r}^{k_m,\ell}\;\lesssim\; &\sup_{\mathbf x\in\Lambda^{k_m}}\sum_{\mathbf y\in\Lambda^{k_m}} \tilde{p}_{t_m-r}^{\otimes,\ell}(\mathbf x,\mathbf y)\bb E\big[\big\lvert \phi_{t_{m-1}}(\mathbf X_{r-t_{m-1},\mathbf y}) \big\rvert\big]\\
&+  \int_{t_{m-1}}^r \dd\tau\, \frac{1}{\sqrt{n^2(t_m-\tau)+1}} \mc S_{t_m,\tau}^{k_m-2,\ell+1} \\
&+  n\int_{t_{m-1}}^r\dd\tau\, \frac{1}{(n^2(r-\tau)+1)^{1/4}} \mc S_{t_m,\tau}^{k_m-1,\ell+1} \frac{1}{\sqrt{n^2(t_m-\tau)+1}} \,.
\end{align*}
By Lemma \ref{initialqm} below, the first term on the right hand side of the above inequality is bounded by
$$Cn^{-\sum_{j=1}^m k_j/2} \prod_{j=1}^{m-1} \Big(\frac{n}{\sqrt{n^2(t_{j+1}-t_j)+1}}\Big)^{k_j\wedge \sum_{l=j+1}^m k_l}. $$
The second and third terms on the right hand side of the above inequality are bounded by the same quantity. Indeed, this can be seen by adapting the arguments of the proof of Lemma~\ref{difbound} and using the seed bound
$$ \mc S_{t_m,r}^{0,\ell}\;\lesssim\; n^{-\sum_{j=1}^{m-1} k_j/2}\prod_{j=1}^{m-2} \Big(\frac{n}{\sqrt{n^2(t_{j+1}-t_j)+1}}\Big)^{k_j\wedge \sum_{l=j+1}^{m-1} k_l}\,.$$
This then indeed implies the bound
$$\mc S_{t_m,r}^{k_m,\ell}\;\lesssim\;  Cn^{-\sum_{j=1}^m k_j/2} \prod_{j=1}^{m-1} \Big(\frac{n}{\sqrt{n^2(t_{j+1}-t_j)+1}}\Big)^{k_j\wedge \sum_{l=j+1}^m k_l}\frac{(n^2(t_m-t_{m-1})+1)^{1/4}}{(n^2(r-t_{m-1})+1)^{1/4}}\,,$$
which allows to conclude as in Lemma~\ref{difbound} and Proposition~\ref{estimatePsi}.

\subsection{Estimates on initial terms}\label{45}

\begin{lemma}\label{initialq}
	Fix a positive integer $k_1$. Consider a non-negative function $q:\bb R_+\times\Lambda^{k_1}\times\Lambda^{k_1}\to [0,\infty)$ such that, for every $t>0$ and every subset $J\subset \{1,\dots, k_2\}$,
	\begin{equation}\label{condq}
	\sup_{\substack{\mathbf x\in\Lambda^{k_1}\\\mathbf z\in\Lambda^{k_2}}}\sum_{\substack{\{y_i:\, 1\leq i\leq k_1\}\,:\\ \{y_i:\, 1\leq i\leq k_1\} \cap \{z_j: 1\leq j\leq k_2\}\, =\,\{z_j: \,j\in J\} }}
	\hspace{-1.5cm} q_t\big( (x_i: 1\leq i\leq k_1),  (y_i: 1\leq i\leq k_1)\big)\;\lesssim\; \frac{1}{(\sqrt{n^2 t+1})^{|J|}}\,.
	\end{equation}
	Then, there exists a constant $C$ independent of $(x_i: 1\leq i\leq k_1)$ and $(z_j: 1\leq j\leq k_2)$  such that
	\begin{equation*}
	\begin{split}
	&\sum_{y_i: 1\leq i\leq k_1}q_{t-s}\big( (x_i: 1\leq i\leq k_1),  (y_i: 1\leq i\leq k_1)\big) \bb E\Big[\prod_{i=1}^{k_1}\overline{\eta}_s(y_i) \prod_{j=1}^{k_2} \overline{\eta}_s(z_j)\Big]\\
	&\;\leq\; Cn^{-(k_1+k_2)/2}\times \Big(\frac{n}{\sqrt{n^2(t-s)+1}}\Big)^{k_1\wedge k_2},
	\end{split}
	\end{equation*}
	for any $s<t$ with $0\leq s,t \leq T$ and every integer $n\geq 1$.
\end{lemma}
\begin{proof}
	The sum to be estimated in the lemma can be written as
	\begin{equation}\label{sumpEell}
	\sum_{\ell=0}^{k_1+ k_2}\hspace{-20pt} \sum_{\substack{\{y_i: 1\leq i\leq k_1\}:\\ \big| \{y_i: 1\leq i\leq k_1\} \Delta \{z_j: 1\leq j\leq k_2\}\big| =\,\ell}}\hspace{-20pt}
	q_{t-s}\big( (x_i: 1\leq i\leq k_1),  (y_i: 1\leq i\leq k_1)\big) \bb E\Big[\prod_{i=1}^{k_1}\overline{\eta}_s(y_i) \prod_{j=1}^{k_2} \overline{\eta}_s(z_j)\Big]\,.
	\end{equation}
	Since the number of non-repetitive points in $(y_i,z_j: 1\leq i\leq k_1, 1\leq j\leq k_2)$ is exactly $\big| \{y_i: 1\leq i\leq k_1\} \Delta \{z_j: 1\leq j\leq k_2\}\big|$, if $\big| \{y_i: 1\leq i\leq k_1\} \Delta \{z_j: 1\leq j\leq k_2\}\big|= \ell$, by Proposition~\ref{ndis},  there exists a constant $C$ independent of $(y_i: 1\leq i\leq k_1)$ and $(z_j: 1\leq j\leq k_2)$ such that
	\begin{equation}\label{estiep}
	\Big\lvert \bb E\Big[\prod_{i=1}^{k_1}\overline{\eta}_s(y_i) \prod_{j=1}^{k_2} \overline{\eta}_s(z_j)\Big] \Big\rvert \;\lesssim\; n^{-\ell/2}\,.
	\end{equation}
	On the other hand, for a list of points $(y_i: 1\leq i\leq k_1)$ such that $\big| \{y_i: 1\leq i\leq k_1\} \Delta \{z_j: 1\leq j\leq k_2\}\big|\;=\;\ell $, there are $(k_1+k_2-\ell)/2$ points that appear in both sets $\{y_i:1\leq i\leq k_1\}$ and $\{z_j: 1\leq j\leq k_2\}$, i.e.,
	$$\big| \{y_i: 1\leq i\leq k_1\} \cap \{z_j: 1\leq j\leq k_2\}\big|\;=\;(k_1+k_2-\ell)/2\,.$$
	Therefore, for each $\ell$,
	\begin{equation}\label{eq:qest}
	\begin{split}
	&\sum_{\substack{\{y_i: 1\leq i\leq k_1\}:\\ \big| \{y_i: 1\leq i\leq k_1\} \Delta \{z_j: 1\leq j\leq k_2\}\big|=\ell }}q_{t-s}\big( (x_i: 1\leq i\leq k_1),  (y_i: 1\leq i\leq k_1)\big) \\
	=\,&\sum_{\substack{J\subset\{1,2,\dots,k_2\},\\ |J|=(k_1+k_2-\ell)/2}}\sum_{\substack{\{y_i: 1\leq i\leq k_1\}:\\ \{y_i: 1\leq i\leq k_1\} \cap \{z_j: 1\leq j\leq k_2\}\\ =\{z_j: \,j\in J\} }}q_{t-s}\big( (x_i: 1\leq i\leq k_1),  (y_i: 1\leq i\leq k_1)\big)\\
	\lesssim\,&\sum_{\substack{J\subset\{1,2,\dots,k_2\},\\ |J|=(k_1+k_2-\ell)/2}}\Big(\frac{1}{\sqrt{n^2 (t-s)+1}} \Big)^{(k_1+k_2-\ell)/2},\\
	\end{split}
	\end{equation}
	where in the last inequality we applied condition \eqref{condq}. The expression above is further bounded by
	\begin{equation*}
	\sum_{\substack{J\subset\{1,2,\dots,k_2\},\\ |J|=(k_1+k_2-\ell)/2}}n^{-(k_1+k_2-\ell)/2}\Big(\frac{n}{\sqrt{n^2 (t-s)+1}} \Big)^{k_1\wedge k_2},
	\end{equation*}
	since $(k_1+k_2-\ell)/2$ is exactly the number of points that appear in both sets $\{y_i;1\leq i\leq k_1\}$ and $\{z_i: 1\leq i\leq k_2\}$, which is of course less than or equal to $k_1\wedge k_2$. Here, we used as well that the distance between $t$ and $s$ is bounded. From this estimate and \eqref{estiep}, for each $\ell$, we have
	\begin{equation*}
	\begin{split}
	&\sum_{\substack{\{y_i:\, 1\leq i\leq k_1\}:\\ \big| \{y_i:\, 1\leq i\leq k_1\} \Delta \{z_i:\, 1\leq i\leq k_2\}\big| =\,\ell}}\hspace{-1cm}q_{t-s}\big( (x_i: 1\leq i\leq k_1),  (y_i: 1\leq i\leq k_1)\big) \bb E\Big[\prod_{i=1}^{k_1}\overline{\eta}_s(y_i) \prod_{j=1}^{k_2} \overline{\eta}_s(z_j)\Big] \\
	&\lesssim\, n^{-(k_1+k_2)/2}\times \Big(\frac{n}{\sqrt{n^2(t-s)+1}}\Big)^{k_1\wedge k_2}.
	\end{split}
	\end{equation*}
	Since the first sum over $\ell$ in \eqref{sumpEell} is finite, the lemma is proved.
\end{proof}

\begin{corollary}\label{initialp}
	There exists a constant $C$ independent of $(x_i: 1\leq i\leq k_1)$ and $(z_i: 1\leq i\leq k_2)$  such that
	\begin{equation*}
	\Big\lvert \mathbf{E}_{(x_i:\,1\leq i\leq k_1)}\big[\varphi_s(\mathbf{X}^i_{t-s}: 1\leq i\leq k_1) \big] \Big\rvert \;\leq\; Cn^{-(k_1+k_2)/2}\times \Big(\frac{n}{\sqrt{n^2(t-s)+1}}\Big)^{k_1\wedge k_2},
	\end{equation*}
	for any $s<t$ with $0\leq s,t\leq T$ and every integer $n\geq 1$.
\end{corollary}
\begin{proof}
	The expectation above can be written as
	\begin{equation}\label{sumpE}
	\sum_{y_i: 1\leq i\leq k_1}p^{\lex}_{t-s}\big( (x_i: 1\leq i\leq k_1),  (y_i: 1\leq i\leq k_1)\big) \bb E\Big[\prod_{i=1}^{k_1}\overline{\eta}_s(y_i) \prod_{j=1}^{k_2} \overline{\eta}_s(z_j)\Big]\,.
	\end{equation}
	By the previous lemma, it remains to check that $p^{\lex}$ satisfies \eqref{condq}.
	As a consequence of the analysis in Section~\ref{Sec431} we have that for every $ t\geq 0$,
	$$\sup_{\mathbf x,\mathbf z\in\Lambda^k}p^{\lex}_{t}\big( (x_i: 1 \leq i\leq k),  (z_j:  1\leq j \leq k)\big) \;\lesssim\;  \Big(\frac{1}{\sqrt{n^2 t+1}} \Big)^{k}.$$
	Given a subset $J$ as in~\eqref{condq} we can find an index set $I$ with $|I|= |J|$ such that $\{y_i: 1\leq i\leq k_1\} \cap \{z_i: 1\leq i\leq k_2\}= \{z_j: j\in J\}= \{y_i: i\in I\}$.
	By the fact that $p_t^{\lex}$ is a transition probability, and the above estimate on the transition probability, we have that  for every $t>0$, every subset $J\subset \{1,2,\dots, k_2\}$,
	\begin{align*}
	&\sup_{\substack{\mathbf x\in\Lambda^{k_1}\\\mathbf z\in\Lambda^{k_2}}}\sum_{\substack{\{y_i:\, 1\leq i\leq k_1\}\,:\\ \{y_i: 1\leq i\leq k_1\} \cap \{z_j: 1\leq j\leq k_2\}\, =\, \{z_j: \,j\in J\} }}p^{\lex}_t\big( (x_i: 1\leq i\leq k_1),  (y_i: 1\leq i\leq k_1)\big)\\
	&=\; \sup_{\substack{\mathbf x\in\Lambda^{k_1}\\\mathbf z\in\Lambda^{k_2}}}p_t^{\lex}\big((x_i:i\in I), (z_j, j\in J) \big)\;\lesssim\; \frac{1}{(\sqrt{n^2 t+1})^{|J|}}\,.
	\end{align*}
\end{proof}
\begin{remark}\label{condtildep}\rm
	The arguments to verify that $p^{\lex}$ satisfies condition \eqref{condq} can be easily adapted to $\tilde p^{\otimes,\ell}$. Since $\tilde p^{\otimes,\ell}$ is a product in terms of $\bar p^{\otimes,\ell}$, and since $\bar p^{\otimes,\ell}$ satisfies property \eqref{finitemeasurebarp}, we have that  for every $t>0$, every subset $J\subset \{1,2,\dots, k_2\}$,
	\begin{align*}
	&\sup_{\substack{\mathbf x\in\Lambda^{k_1}\\\mathbf z\in\Lambda^{k_2}}}\sum_{\substack{\{y_i: 1\leq i\leq k_1\}:\\ \{y_i: 1\leq i\leq k_1\} \cap \{z_j: 1\leq j\leq k_2\}\\ =\{z_j: \,j\in J\} }}\tilde p^{\otimes,\ell}_t\big( (x_i: 1\leq i\leq k_1),  (y_i: 1\leq i\leq k_1)\big)\\
	&\lesssim\; \sup_{\substack{\mathbf x\in\Lambda^{k_1}\\\mathbf z\in\Lambda^{k_2}}}\tilde p^{\otimes,\ell}_t\big((x_i:i\in I), (z_j, j\in J) \big)\;
	\lesssim\; \frac{1}{(\sqrt{n^2 t+1})^{|J|}}\,,
	\end{align*}
	where $I$ is as in the proof of Corollary~\ref{initialp}.
	The only difference is that in the third line, we no longer have equality.
\end{remark}

The results of the above lemma and corollary can be extended to the case of multiple lists of points. Recall the definition of $\phi_t$ given in \eqref{defphi}.
\begin{lemma}\label{initialqm}
	Consider $m$ lists of non-repetitive points $(x_{i_1}: 1\leq i_1\leq k_1), \ldots, (x_{i_m}: 1\leq i_m\leq k_m)$. Assume $t_1<\cdots <t_m$ and $0\leq t_j,t_{j+1}\leq T$ for every $1\leq j< m$.
	Given a non-negative function $q:\bb R_+\times\Lambda^{k_m}\times\Lambda^{k_m}\to [0,\infty)$ such that, for every $t>0$ and every subset $J\subset \{1,2,\dots, k_{m-1}\}$ holds
	\begin{equation}\label{condqm}
	\begin{split}
	&\sup_{\substack{\mathbf x\in\Lambda^{k_m}\\\mathbf z\in\Lambda^{k_{m-1}}}}\sum_{\substack{\{y_i:\, 1\leq i\leq k_m\}\,:\\ \{y_i: 1\leq i\leq k_m\} \cap \{z_j:\, 1\leq j\leq k_{m-1}\}\\ =\{z_j: \,j\in J\} }}q_t\big( (x_{i_m}: 1\leq i_m\leq k_m),  (y_i: 1\leq i\leq k_m)\big)\\
	&\;\lesssim\; \frac{1}{(\sqrt{n^2 t+1})^{|J|}}\,.
	\end{split}
	\end{equation}
	then, there exists a constant $C$ independent of $(x_{i_j}: 1\leq i_j\leq k_j)$, $1\leq j\leq m$,  such that
	\begin{equation*}
	\begin{split}
	&\sum_{y_i:\, 1\leq i\leq k_m}q_{t_m-t_{m-1}}\big( (x_{i_m}: 1\leq i_m\leq k_m),  (y_i: 1\leq i\leq k_m)\big)  \bb E\big[\prod_{i=1}^{k_m}\overline{\eta}_{t_{m-1}}(y_i) \prod_{j=1}^{m-1} \prod_{i_j=1}^{k_j} \overline{\eta}_{t_j}(x_{i_j})\big]\\
	&\;\leq\; Cn^{-\sum_{j=1}^m k_j/2} \prod_{j=1}^{m-1} \Big(\frac{n}{\sqrt{n^2(t_{j+1}-t_j)+1}}\Big)^{k_j\wedge \sum_{l=j+1}^m k_l},
	\end{split}
	\end{equation*}
	for any $s<t$ with $0\leq s,t\leq T$ and every integer $n\geq 1$.
\end{lemma}

\begin{proof}
	The proof is based on induction on $m$, the number of lists of points. Lemma \ref{initialp} deals with the case when $m=2$.

	Assume $m\geq 3$ and that the statement of this lemma holds if the  number of lists of points is strictly less than $m$. The rest of the proof is quite similar to the proof of the previous lemma, so we only outline the main steps. The sum to be estimated in the lemma can be written as
	\begin{equation*}
	\begin{split}
	&\sum_{\ell=0}^{k_m+ k_{m-1}}\!\!\! \sum_{\substack{\{y_i:\, 1\leq i\leq k_m\}\,:\\ \big| \{y_i:\, 1\leq i\leq k_m\} \Delta \{x_{i_{m-1}}:\, 1\leq i_{m-1}\leq k_{m-1}\}\big|=\ell }}\!\!\!\!\!q_{t_m-t_{m-1}}\big( (x_{i_m}: 1\leq i_m\leq k_m),  (y_i: 1\leq i\leq k_m)\big)\\
	&\hspace{7cm}\times \bb E\big[\prod_{i=1}^{k_m}\overline{\eta}_{t_{m-1}}(y_i) \prod_{j=1}^{m-1} \prod_{i_j=1}^{k_j} \overline{\eta}_{t_j}(x_{i_j})\big].
	\end{split}
	\end{equation*}
	For every $\ell\leq k_m+k_{m-1}$ and every collection of points $(y_i: 1\leq i\leq k_m)$ such that
	$$\big| \{y_i: 1\leq i\leq k_m\} \Delta \{x_{i_{m-1}}: 1\leq i_{m-1}\leq k_{m-1}\}\big|\;=\;\ell\,,$$
	using the identity \eqref{identity}
	and the induction assumption, we get that
	\begin{align*}
	&\Big\lvert  \bb E\big[\prod_{i=1}^{k_m}\overline{\eta}_{t_{m-1}}(y_i) \prod_{j=1}^{m-1} \prod_{i_j=1}^{k_j} \overline{\eta}_{t_j}(x_{i_j})\big] \Big\rvert\\
	& \leq Cn^{-\ell/2-\sum_{j=1}^{m-2} \!k_j/2}\Big(\frac{n}{\sqrt{n^2(t_{m-1}\!-\!t_{m-2})\!+\!1}}\Big)^{k_{m-2}\wedge \, \ell}\!\!\! \times\!\!\!\prod_{j=1}^{m-3}\!\! \Big(\frac{n}{\sqrt{n^2(t_{j+1}\!-\!t_{j})\!+\!1}}\Big)^{k_j\wedge (\ell+\sum_{l=j+1}^{m-2}k_{l})}\\
	&\leq Cn^{-\ell/2-\sum_{j=1}^{m-2} k_j/2} \prod_{j=1}^{m-2} \Big(\frac{n}{\sqrt{n^2(t_{j+1}-t_{j})+1}}\Big)^{k_j\wedge\sum_{l=j+1}^{m}k_{l}}\,.
	\end{align*}
	On the other hand, similarly to the way we deduced \eqref{eq:qest}, we obtain
	\begin{align*}
	&\sum_{\substack{\{y_i:\, 1\leq i\leq k_m\}\\ \,\big| \{y_i:\, 1\leq i\leq k_m\} \Delta \{x_{i_{m-1}}:\, 1\leq i_{m-1}\leq k_{m-1}\}\big|=\ell }}q_{t_m-t_{m-1}}\big( (x_{i_m}: 1\leq i_m\leq k_m),  (y_i: 1\leq i\leq k_m)\big) \\
	& \lesssim\;  \bigg(\frac{1}{\sqrt{n^2 (t_m-t_{m-1})+1}} \bigg)^{(k_m+k_{m-1}-\ell)/2}\\
	& \lesssim\; n^{-(k_m+k_{m-1}-\ell)/2}\bigg(\frac{n}{\sqrt{n^2 (t_m-t_{m-1})+1}} \bigg)^{k_m\wedge k_{m-1}}.
	\end{align*}
	These two estimates imply that the absolute value of the expectation in the lemma is bounded by
	\begin{equation*}
	C n^{-\sum_{j=1}^m k_j/2}\Big(\frac{n}{\sqrt{n^2 (t_m-t_{m-1})+1}} \Big)^{k_m\wedge k_{m-1}}  \prod_{j=1}^{m-2} \Big(\frac{n}{\sqrt{n^2(t_{j+1}-t_{j})+1}}\Big)^{k_j\wedge\sum_{l=j+1}^{m}k_{l}}\,,
	\end{equation*}
	proving the lemma.
\end{proof}

\begin{corollary}\label{initialpm}
	Consider $m$ lists of non-repetitive points $(x_{i_1}: 1\leq i_1\leq k_1),\dots,(x_{i_m}: 1\leq i_m\leq k_m)$. Assume $t_1<\cdots <t_m$ and $0\leq t_j,t_{j+1}\leq T$ for every $1\leq j< m$. Then, there exists a constant $C$ independent of $(x_{i_j}: 1\leq i_j\leq k_j)$, $1\leq j\leq m$, such that
	\begin{equation*}
	\begin{split}
	&\Big\lvert \mathbf{E}_{(x_{i_m}:\,1\leq i_m\leq k_m)}\big[\phi_{t_{m-1}}(\mathbf{X}^{i_m}_{t_m-t_{m-1}}: 1\leq i_m\leq k_m) \big] \Big\rvert \\
	&\leq\;Cn^{-\sum_{j=1}^m k_j/2} \prod_{j=1}^{m-1} \Big(\frac{n}{\sqrt{n^2(t_{j+1}-t_j)+1}}\Big)^{k_j\wedge \sum_{l=j+1}^m k_l}
	\end{split}
	\end{equation*}
	for every integer $n\geq 1$.
\end{corollary}
The proof of this corollary is almost the same as the one of Corollary \ref{initialp}, so we omit it.

\section{Joint Fluctuations -- Proof of Theorem~\ref{thm2.6}}\label{sec_5}

The proof of Theorem~\ref{thm2.6} follows the usual structure. First we establish tightness, and then we characterize the possible limit points. We start with the former.
Since the Cartesian product of compact sets is compact, in order to show the tightness of the pair current/occupation time it is enough to assure tightness for each one of them. We start with the easier one, the proof of the tightness for the occupation time.
\subsection{Tightness of the occupation time}
\begin{proposition}\label{prop:tightnessgamma}
	Assume that $\rho_0 \in C^1$. For any fixed $u\in \bb R$,
	the sequence of processes
	$\big\{\big(n^{1/2}\Gamma^n_{\lfloor un\rfloor}(t)\big): t\in [0,T]\big\}_{n\in \bb N}$ 	is tight in $\mc D([0,T], \bb R)$ endowed with the
	uniform topology.
\end{proposition}
\begin{proof}
	Without loss of generality, let us assume that  $u=0$. 	Fix $0\leq s<t\leq T$.
	It is enough to show that there exists a constant $C$ independent of $n$ such that
	\begin{equation}\label{KC}
	\bb E_{\nu_{\rho_0^n(\cdot)}}\Big[\Big\lvert  \sqrt{n}\,\Gamma_0^n(t) -  \sqrt{n}\,\Gamma_0^n(s) \Big\rvert^2 \Big] \;\leq\; C(t-s)^{10/7}\,.
	\end{equation}
	We shall prove this inequality only for $t,s>0$ such that $0\leq t-s\leq 1$.
	The expectation on the left hand side of \eqref{KC} is bounded by
	\begin{equation*}
	2\,\bb E_{\nu_{\rho_0^n(\cdot)}}\Big[ n\Big\lvert \int_s^t\dd r\, \Big\{\overline{\eta}_r(0)-\frac{1}{\varepsilon n} \sum_{x=0}^{\varepsilon n-1}\overline{\eta}_r(x)\Big\} \Big\rvert^2\Big] + 2\,\bb E_{\nu_{\rho_0^n(\cdot)}}\Big[ n\Big\lvert \int_s^t\dd r\, \frac{1}{\varepsilon n} \sum_{x=0}^{\varepsilon n-1}\overline{\eta}_r(x) \Big\rvert^2\Big]\,.
	\end{equation*}
	In view of Theorem~\ref{thm:KipnisVaradhan}, the first term is bounded by $C(t-s)^{10/7}$ if we choose
	$\varepsilon=(t-s)^{4/7}$.
	It remains to prove that
	\begin{equation}\label{2nd}
	\bb E_{\nu_{\rho_0^n(\cdot)}}\Big[\Big\lvert\int _s^t\dd r\, \frac{1}{\varepsilon\sqrt{n}}\sum_{x=0}^{\varepsilon n-1}\overline{\eta}_r(x) \Big\rvert^2 \Big]\;\leq\; C(t-s)^{10/7}\,.
	\end{equation}
	By Jensen's inequality,
	\begin{equation}\label{afterJensen}
	\bb E_{\nu_{\rho_0^n(\cdot)}}\Big[\Big\lvert\int _s^t \dd r\, \frac{1}{\varepsilon\sqrt{n}}\sum_{x=0}^{\varepsilon n-1}\overline{\eta}_r(x) \Big\rvert^2 \Big]\;\leq\;(t-s)\int_s^t \dd r\,\bb E_{\mu^n}\Big[\Big\lvert \frac{1}{\varepsilon\sqrt{n}}\sum_{x=0}^{\varepsilon n-1}\overline{\eta}_r(x) \Big\rvert^2 \Big]\,.
	\end{equation}
	It is shown in  \cite[Lemma 3.2]{jaralandim2006} that
	\begin{equation*}
	\sup_{x\neq y}\Big\lvert\bb E_{\mu_n}[\overline{\eta}_t(x)\overline{\eta}_t(y)] \Big\rvert\;\leq\; \frac{C_0\sqrt{t} }{n}
	\end{equation*}
	for some positive constant $C_0$ that depends only on the initial profile $\rho_0$. Since $t\leq T$ and $T$ is fixed,
	from this estimate on the correlations we have
	\begin{equation*}
	\begin{split}
	\bb E_{\mu^n}\Big[\Big\lvert \frac{1}{\varepsilon\sqrt{n}}\sum_{x=0}^{\varepsilon n-1}\overline{\eta}_r(x) \Big\rvert^2 \Big] \;=\; &\sum_{x=0}^{\varepsilon n-1} \frac{1}{\varepsilon^2 n}\bb E_{\mu_n} [\overline{\eta}_r(x)^2] + \sum_{\substack{x\neq y\\ 0\leq x,y\leq \varepsilon n-1}}\!\!\!\! \frac{\bb E_{\mu_n}[\overline{\eta}_r(x)\overline{\eta}_r(y)]}{\varepsilon^2 n}
	\;\lesssim\;  \frac{1}{\varepsilon} +1\,.
	\end{split}
	\end{equation*}
	Integrating the last expression from $s$ to $t$, by \eqref{afterJensen},
	we obtain that
	\begin{equation*}
	\bb E_{\nu_{\rho_0^n(\cdot)}}\Big[\Big\lvert\int _s^t \dd r\, \frac{1}{\varepsilon\sqrt{n}}\sum_{x=0}^{\varepsilon n-1}\overline{\eta}_r(x) \Big\rvert^2 \Big]\;\leq\;C(t-s)^2(\varepsilon^{-1}+1)\,.
	\end{equation*}
	Inequality \eqref{2nd} then follows from this upper bound and our previous choice $\varepsilon=(t-s)^{4/7}$.
\end{proof}

\subsection{Tightness of the current}
\begin{proposition}\label{prop:tightnessJ} Assume that $\rho_0\in C^2$ with bounded first and second derivatives and is bounded away from zero and from one. For any fixed $u\in \bb R$, the sequence of processes
	$\big\{n^{-1/2}\overline{J}^n_{\lfloor un\rfloor, \lfloor u n\rfloor +1}(t): t\in [0,T]\big\}_{n\in \bb N}$ 	is tight in $\mc D([0,T], \bb R)$ endowed with the uniform topology.
\end{proposition}
It is enough to give the proof of this proposition for the processes $\big\{n^{-1/2}\overline{J}^n(t)_{-1,0}: t\in[0,1/2]\big\}_{n\in \bb N}$ only.  For simplicity of notation, let us omit the subscript of
$\overline{J}^n_{-1,0}$.
\begin{remark}\rm
	In the equilibrium setting, it is known that $n^{-1/2}\,\overline{J}^n(t)$ converges to a fractional Brownian motion with Hurst exponent $1/4$. Therefore one expects that
	\begin{equation}\label{eq:Rem54}
	\bb E_{\nu_{\rho_0^n(\cdot)}}\Big[\Big\lvert  n^{-1/2}\,\overline{J}^n(t) -  n^{-1/2}\,\overline{J}^n(s) \Big\rvert^{p} \Big] \,\leq\, C(t-s)^{p/4}
	\end{equation}
	for every positive integer $p$. By the Kolmogorov-Centsov tightness criterion, $p/4$, the power of $t-s$,  has to be larger than $1$. So the first attempt one may think of is to estimate the sixth moment, like Peligrad and Sethuraman did in \cite{SP} in the equilibrium setting. However, to bound the sixth moment by $C(t-s)^{3/2}$ turns out to be very tricky in the non-equilibrium setting. This is because the estimate obtained in Theorem \ref{expdif} is probably not optimal for the case $t-s\leq n^{-1}$. The estimate~\eqref{eq:Rem54} however indicates that higher moments should lead to larger exponents of $(t-s)$. We will make use of that heuristics and estimate the tenth moment instead. This will give us more space to prove tightness with bounds that are probably not sharp.
\end{remark}
It is therefore sufficient to prove the following proposition.
\begin{proposition}
	There exists a constant $C>0$ such that
	\begin{equation}\label{KCc}
	\bb E_{\nu_{\rho_0^n(\cdot)}}\Big[\Big\lvert  n^{-1/2}\,\overline{J}^n(t) -  n^{-1/2}\,\overline{J}^n(s) \Big\rvert^{10} \Big] \;\leq\; C(t-s)^{5/4}\
	\end{equation}
	for all $0\leq s<t\leq T$ and all $n\in\bb N$.
\end{proposition}

\begin{proof}
	We first deal with the situation where $n$ is large in the sense that $n^2(t-s)\geq 1$.
	Given $K>0$, define the function $G^K:\bb R\to [0,1]$  by
	\begin{align}\label{heavi_aprox}
	G^K(v) \;=\; \Big\{1- \frac{v}{K}\Big\}^+\one\{v\geq 0\}\,.
	\end{align}
	We recall the following identity, which was obtained in the proof of Proposition~3.1 of \cite{jaralandim2006}: for any $n$ and $K$  such that $Kn\geq 1$, one has the identity
	\begin{equation}\label{repcur}
	n^{-1/2}\overline{J}^n(t) \,=\, \big[\mc Y_t^n(G^K) - \mc Y_0^n(G^K) \big] + \frac{1}{\sqrt{n}}\sum_{x=1}^{Kn}\frac{1}{Kn}M^n_{x-1,x}(t) +\frac{1}{\sqrt{n}} \int_0^t \dd s \,\frac{n}{K}\big[\overline{\eta}_s(0) -\overline{\eta}_s(nK) \big]\,,
	\end{equation}
	where $M^n_{x,x+1}(t)$ is the martingale given by
	\begin{equation*}
	M^n_{x,x+1}(t)\;=\; \overline{J}^n_{x,x+1}(t)-n^2 \int_0^t \dd s\,\big\{\overline{\eta}_s(x)-\overline{\eta}_s(x+1)\big\}
	\end{equation*}
	whose quadratic variation is
	\begin{equation*}
	\langle M^n_{x,x+1} \rangle_t \;=\; n^2\int_0^t \dd s\, \big\{\eta_s(x)-\eta_s(x+1)\big\}^2 \;\leq\; n^2 t\,.
	\end{equation*}
	We are going to estimate the tenth moment of each term on the right hand side of \eqref{repcur}. We will show that choosing $K=(t-s)^{3/8}$ then
	\begin{equation}\label{6m1}
	\bb E_{\nu_{\rho_0^n(\cdot)}}\Big[\Big\lvert  \mc Y_t^n(G^K) - \mc Y_s^n(G^K) \Big\rvert^{10} \Big] \,\leq\, C(t-s)^{5/4}\,,
	\end{equation}
	\begin{equation}\label{6m2}
	\bb E_{\nu_{\rho_0^n(\cdot)}}\Big[\Big\lvert  \frac{1}{\sqrt{n}}\sum_{x=1}^{Kn}\frac{1}{Kn}\big[M^n_{x-1,x}(t)-M^n_{x-1,x}(s)\big]\Big\rvert^{10} \Big] \;\leq\; C(t-s)^{5/4}\,,
	\end{equation}
	and
	\begin{equation}\label{6m3}
	\bb E_{\nu_{\rho_0^n(\cdot)}}\Big[\Big\lvert  \frac{1}{\sqrt{n}} \int_s^t \dd r\, \frac{n}{K}\big[\overline{\eta}_r(0) -\overline{\eta}_r(nK) \big] \Big\rvert^{10} \Big] \;\leq\; C(t-s)^{5/4}\,,
	\end{equation}
	which implies \eqref{KCc}. Note that $n^2(t-s)\geq 1$ implies that $Kn\geq 1$ for $K=(t-s)^{3/8}$, which is a condition to \eqref{repcur} to make sense.

	The expectation in  \eqref{6m1} is equal to
	\begin{equation}\label{6m1open}
	\frac{1}{n^5} \sum_{x_1,\ldots,\, x_{10} =0}^{Kn}\bb E_{\nu_{\rho_0^n(\cdot)}}\bigg[\prod_{i=1}^{10} \frac{Kn-x_i}{Kn}\big(\overline{\eta}_t(x_i)- \overline{\eta}_s(x_i)\big) \bigg].
	\end{equation}
	Note that the cardinality of the set
	\begin{equation*}
	\big\{(x_i: 1\leq i\leq 10): x_i\in\{1,2,\ldots, Kn\}\, \text{ and }\, [(x_i: 1\leq i\leq 10)]=m \big\}
	\end{equation*}
	is of the order $O((Kn)^{m})$.
	To get a good bound of the expectation inside the sum above, we consider two cases.

The first case is when $t-s\geq n^{-1}$. Let $\|(x_i: 1\leq i\leq 10)\|=\ell$ and $[(x_i: 1\leq i\leq 10)]=m$.
	In view of Proposition \ref{expdif}, the expression in \eqref{6m1open} is bounded by
	\begin{equation*}
	\frac{C}{n^5} \sum_{0\leq \ell\leq m\leq 10} n^{-\ell/2}\Big(\frac{1}{\sqrt{t-s}}\Big)^{\ell/2}(Kn)^{m}\;.
	\end{equation*}
	Fix $0\leq \ell\leq 10$, we will estimate
	\begin{equation*}
	n^{m-\ell/2-5}(t-s)^{-\ell/4}K^m\,.
	\end{equation*}
	Since $t-s\geq n^{-1}$ and $K=(t-s)^{3/8}$, observing that $m\leq \frac{10+\ell}{2}$, the term above is bounded from above by
	$$C(t-s)^{ 3m/8+5+\ell/2-m-\ell/4}\,=\, C(t-s)^{5+\ell/4-5m/8}.$$
	From \eqref{def[]} and \eqref{equk}, an easy computation gives
	$$5+\frac{\ell}{4}-\frac{5m}{8}\;\geq\; \frac{5}{4}\,.$$
	This finishes the proof of \eqref{6m1} in the first case.

	The second case is when $n^{-2}\,\leq\,t-s\leq n^{-1}$.  In view of Proposition \ref{expdif}, the expression in \eqref{6m1open} is bounded by
	\begin{equation*}
	\frac{C}{n^5}\!\! \sum_{0\leq \ell\leq m\leq 10}\hspace{-0.4cm} n^{-m/2}\Big(\frac{n}{\sqrt{n^2(t-s)+1}}\Big)^{m-\ell/2}(Kn)^{m}\;\leq\; CK^mn^{3m/2-5-\ell/2}\Big(\frac{1}{n^2(t-s)+1}\Big)^{m/2-\ell/4}.
	\end{equation*}
	Since $K=(t-s)^{3/8}$ and $m/2-\ell/4\geq 0$, using the bound
	$$\frac{1}{n^2(t-s)+1}\;\leq\; \frac{1}{n^2(t-s)}\,,$$
	the previous expression is less than or equal to
	$$C(t-s)^{3m/8-m/2+\ell/4}n^{3m/2-5-\ell/2-m+\ell/2}\;=\; C(t-s)^{\ell/4-m/8}n^{m/2-5}\,.$$
	Since  $n^{-1}\leq \sqrt{t-s}$, the last expression is further bounded from above by
	$$C(t-s)^{\ell/4-3m/8+5/2}\,.$$
	On the other hand,
	\begin{equation*}
	\frac{\ell}{4}-\frac{3m}{8} + \frac{5}{2}\;\geq\; \frac{5}{4}\,,
	\end{equation*}
	which allows to conclude.
	This concludes the proof of \eqref{6m1}.

	Applying the Burkholder-Davis-Gundy inequality and using the fact that the martingales $\{M^n_{x,x+1}\}_x$ are orthogonal, we see that the expectation in \eqref{6m2} is bounded by
	\begin{equation*}
	Cn^{-5}\Big \langle \sum_{x=1}^{Kn}\frac{1}{Kn} M^n_{x-1,x}\Big \rangle_{t-s}^5\;=\; Cn^{-5}\Big( \frac{Kn(t-s)}{K^2}\Big)^5\;=\; \frac{C(t-s)^5}{K^5}\,=\,C(t-s)^{25/8}\,.
	\end{equation*}

	The expectation in \eqref{6m3} is equal to
	$$10!\frac{n^5}{K^{10}}  \bb E_{\nu_{\rho_0^n(\cdot)}}\Big[ \int_s^t \dd r_1\int_s^{r_1}\dd r_2\cdots \int_s^{r_9} \dd r_{10}\,\prod_{i=1}^{10} \big[\overline{\eta}_{r_i}(0) -\overline{\eta}_{r_i}(nK) \big] \,  \Big]$$
	which is bounded by $
	\frac{C(t-s)^5}{K^{10}} =C(t-s)^{5/4}
	$ 	in view of \eqref{diftime}.

	We now consider the situation when $(t-s)n^2<1$. Rewriting the martingale $M^n_{-1,0}$ as
	\begin{equation*}
	M^n_{-1,0}(t)\;=\; \overline{J}^n(t)-n^2 \int_0^t \dd s\,\{\overline{\eta}_s(x)- \overline{\eta}_s(x+1)\}\,,
	\end{equation*}
	we have that
	\begin{align*}
	&\bb E_{\nu_{\rho_0^n(\cdot)}}\!\Big[\Big\lvert  n^{-1/2}\,\overline{J}^n_{-1,0}(t) \!-\!  n^{-1/2}\,\overline{J}^n(s) \Big\rvert^{10} \Big] \leq \bb E_{\nu_{\rho_0^n(\cdot)}}\!\Big[ n^{-5}\Big(M_{-1,0}(t) \!-\!M_{-1,0}(s)\Big)^{10} \Big] + n^{15}(t\!-\!s)^{10}.
	\end{align*}
	By Burkholder-Davis-Gundy inequality, the first expectation on the right hand side is bounded by
	$$Cn^{-5}\Big \langle M_{-1,0}\Big \rangle_{t-s}^5\;=\; Cn^{-5}\Big( n^2(t-s)\Big)^5\;=\; Cn^5(t-s)^5\;\lesssim\; (t-s)^{5/2}$$
	since $n<(t-s)^{-1/2}$. Moreover, under the assumption that $n<(t-s)^{-1/2}$, we have that $n^{15}(t-s)^{10}$ is also bounded by $C(t-s)^{5/2}$. This finishes the proof.
\end{proof}

\subsection{Convergence of finite-dimensional distributions}
At this point, we must recall the density fluctuations of the SSEP.
Recall that $\mc{X}(x)= x(1-x)$. To shorten notation  write also $\mc{X}_r := \mc{X}(\rho(u, r)) = \rho(u, r) (1-\rho(u, r))$.
Writing $H^{(k)}$ for the $k^{\text{th}}$ derivative of the function $H$, denote by $\mc S(\bb R)$ the  Schwartz space, that is, the set of all functions $H\in C^\infty$ such that
\begin{equation}\label{seminorms}
\Vert H \Vert_{k,\ell}\;:=\;\sup_{x\in \bb R}|x^\ell
\,H^{(k)}(x)|\;<\;\infty\,,
\end{equation}
for all integers $k,\ell\geq 0$, and denote by  $\mc S'(\bb R)$ its topological dual, see \cite{reedsimon} for more about that subject.
Let  the density field $\mc Y^n_t$  be defined by
\[
\mc Y^n_t (H)\;=\; \frac{1}{\sqrt{n}}\sum_{x\in \bb Z} \overline{\eta}_{t}(x)H(\pfrac{x}{n})\,,\quad \forall\, H\in \mc {S} (\bb R)\,,
\]
where $\overline{\eta}_{t}(x)=\eta_{t}(x)-\rho_t^n(x)$ is the centred occupation.
\begin{theorem}[Density fluctuations \cite{Ravi1992}]\label{thm21}
	Consider the Markov process $\{\eta_{t}: t\geq{0}\}$ starting from the  the slowly varying  Bernoulli product measure $\nu_{\rho^n_0(\cdot)}$ previously defined.
	Then, the sequence of processes $\{\mathcal{Y}_{t}^n\}_{ n\in{\bb N}}$ converges in distribution, as $n\rightarrow{+\infty}$, with respect to the
	Skorohod topology
	of $\mathcal{D}([0,T],\mathcal{S}'(\bb R))$, to  the generalized Ornstein-Uhlenbeck (Gaussian and zero mean) process $\mathcal{Y}_t$
	taking values on $\mathcal{C}([0,T],\mathcal{S}'(\bb R))$ whose covariances are given by
	\begin{equation}\label{cov_density_0}
	\bb E\big[\mc Y_t(H)\mc Y_s(G)\big]\;=\; \int_{\bb R}(T_{t-s}H)G\, \mc{X}_s
	-\int_0^s \dd r\int_{\bb R}(T_{t-r}H)(T_{s-r}G)\big\{\partial_r\mc{X}_r - \Delta \mc{X}_r \big\}
	\end{equation}
	for any $H,G\in \mc S(\bb R)$, where
	\begin{equation}\label{heat_kernel}
	T_t f(x):=\frac{1}{\sqrt{4\pi t}}\int_{\bb R}  \dd y\,f(y)\exp\big\{-(x-y)^2/4t\big\}
	\end{equation}
	is  the heat semi-group.
\end{theorem}
An integration by parts permits to rewrite \eqref{cov_density_0} as
\begin{equation}\label{cov_density}
\bb E\big[\mc Y_t(H)\mc Y_s(G)\big]\;=\; \int_{\bb R}(T_{t}H)(T_{s}G)\mc{X}_0
+2\int_0^s \dd r\int_{\bb R}(\nabla T_{t-r}H)(\nabla T_{s-r}G)\mc{X}_r
\end{equation}
for any $H,G\in \mc S(\bb R)$.
Without loss of generality,  we will deal with  the joint limit of current and occupation time only at two macroscopic points $u_1,u_2\in \bb R$.
Recall \eqref{heavi_aprox} and define  $G_u^K:\bb R\to [0,1]$  by $G_u^K(v) = G^K(v-u)$, which approximates the Heaviside function
$G_u(v) = \one\{v\geq u\}$ centered at $u\in \bb R$
as $K$ goes to infinity.
As a  consequence of  \cite{jaralandim2006} we have:
\begin{proposition}[\cite{jaralandim2006}, Proposition 3.1]\label{prop:currentdensity}
	For any $t\geq 0$ and $u\in \bb R$,
	\begin{equation*}
	\lim_{K\to\infty} \bbE_{\nu_{\rho_0^n(\cdot)}}\bigg[ \Big(\pfrac{1}{\sqrt{n}}
	\overline{J}_{\lfloor un\rfloor}(t)-\mc Y^n_t (G^K_u) + \mc Y^n_0 (G^K_u)\Big)^2\bigg]\;=\;0
	\end{equation*}
	uniformly in $n$.
\end{proposition}
On the other hand, about the occupation time, as an immediate consequence of Theorem~\ref{thm:KipnisVaradhan}, we have that
\begin{corollary}\label{gamma_replace}
	For any $t\in [0,T]$, $u\in \bb R$ and $n\in \bb N$,
	\begin{equation*}	\bbE_{\nu_{\rho_0^n(\cdot)}}\Big[\Big(n^{1/2}\Gamma^n_{\lfloor un\rfloor}(t) - \int_0^t \mc Y_s^n(\iota^{1/K}_u)\,\dd s \Big)^2\Big] \;\lesssim\; \frac{T}{K^{3/4}}(1+T^{1/4}+T)\,,
	\end{equation*}
	where $\iota^{1/K}_u:=K\one[u,u+1/K]$.
\end{corollary}
We start with two simple but useful lemmas on random variables.
\begin{lemma}\label{Lp_bound}
	Let $X_n$ and $X$ be random variables. 	For any $p\geq 1$, we have that $L^p$ bounds are preserved by convergence in distribution, that is, if $X_n\stackrel{d}{\longrightarrow}X$ and $\bb E[\vert X_n\vert^p]\leq c$ for all $n\in \bb N$, then $E[\vert X\vert^p]\leq c$.
\end{lemma}
\begin{proof}
	Let $M>0$. Since $\bb E[\vert X_n\vert^p]\leq c$, then $\bb E[\vert X_n\vert^p\wedge M]\leq c$. Noting that $f(x)=\vert x\vert^p\wedge M$ is a bounded continuous function, the convergence in distribution implies that $E[\vert X\vert^p\wedge M]\leq c$ and the Monotone Convergence Theorem concludes the proof.
\end{proof}
\begin{lemma}\label{lemmadiagonal}
	Fix $p\geq 1$ and let $A_n$,  $B_n^K$ and $B^K$ be sequences of random variables such that
	\begin{enumerate}
		\item $\Vert A_n-B^K_n\Vert_p \leq f(K)$ for all $n$ and all $K$, with $f(K)\to 0$ as $K\to\infty$.
		\item $B^K_n\stackrel{d}{\longrightarrow}B^K$ as $n\to\infty$.
		\item $\{A_n\}_{n\in \bb N}$ is tight.
	\end{enumerate}
	Then, there exists some random variable $B$ such that $B^K\stackrel{L^p}{\longrightarrow}B$ as $K\to \infty$. Furthermore $A_n\stackrel{d}{\longrightarrow}B$ as $n\to\infty$.
\end{lemma}
\begin{proof}
	By (1), we have that $\{B^K_n\}_{K\geq 1}$ is a $L^p$-Cauchy sequence, uniformly in $n\in \bb N$. By  Lemma~\ref{Lp_bound} and (2), we infer that $\{B^K\}_{K\geq 1}$ is a $L^p$-Cauchy sequence as well, thus $B^K\stackrel{L^p}{\longrightarrow}B$ as $K\to \infty$ for some $B$.
	Let $A_{n_j}$ be a subsequence of $A_n$ that converges in distribution, which exists by the tightness assumption (3), and denote by $A$ its limit. By Lemma~\ref{Lp_bound}, we deduce that $\Vert A-B^K\Vert_p \leq f(K)$ for all $K$. Therefore $A=B$, concluding the proof.
\end{proof}
Let us come back to our model. We claim that
\begin{equation}\label{vectordis}
\big(\mc Y_t^n (G^K_u) - \mc Y_0^n (G^K_u), \int_0^T \mc Y_s^n(\iota^{1/K}_u)\,\dd s\big)
\end{equation}
converges in distribution to a Gaussian vector, denoted by
\begin{equation*}
\big(\mc Y_t (G^K_u) - \mc Y_0 (G^K_u), \int_0^t \mc Y_s(\iota^{1/K}_u)\,\dd s\big)
\end{equation*}
as $n\to\infty$.
To prove the claim, we first note that the density field $\mc Y_t$ acts on functions on the Schwartz space, and $G^K_u$ and $\iota_{1/K}^u$ are discontinuous functions, thus not in the Schwartz space. Nevertheless, these functions   are continuous by parts with compact support, and a  $L^2$-density argument using \eqref{cov_density} then permits to extend the Gaussian field $\mc Y_t$  to act on these functions.
About the convergence, it is easy to find sequence of functions $H^{1, K, u}_{j}$ and $H^{2, K, u}_{j}$ in the Schwartz space
such that
$$\Vert H^{1, K, u}_{j}- G^K_u\Vert_1+\Vert H^{1, K, u}_{j}- G^K_u\Vert_2\;\leq\; c/j\,,\quad\forall\, j \in \bb N$$
 and
$$\Vert H^{2, K, u}_{j}- \iota^{1/K}_u\Vert_1+\Vert H^{2, K, u}_{j}- \iota^{1/K}_u\Vert_2\;\leq\; c/j\,, \quad \forall\,j\in \bb N,$$
 thus
\begin{equation}\label{jcurrent}
\Vert \mc Y^n_t(H^{1, K, u}_{j}- G^K_u)\Vert_2\;\leq\; \tilde{c}/j\,,\quad \forall\,j\in \bb N
\end{equation}
 and
\begin{equation}\label{joccupation}
\Big\Vert \int_0^T\mc Y^n_s(H^{2, K, u}_{j}- \iota^{1/K}_u)\,\dd s\Big\Vert_2\leq \tilde{c}/j\,,\quad \forall\,j\in \bb N\,.
\end{equation}
It follows that the sequence in \eqref{vectordis} is a Cauchy sequence, thus it has a limit, which  is a Gaussian vector due to  \eqref{jcurrent}, \eqref{joccupation} and Theorem \ref{thm21}. This proves the claim.

Since we already proved tightness for current and occupation time in the uniform topology, this assures tightness at any fixed time. Thus recalling Proposition~\ref{prop:currentdensity} and Corollary~\ref{gamma_replace}, we may apply  Lemma~\ref{lemmadiagonal} to conclude that
\begin{equation*}
(n^{-1/2}\overline{J^n}_{\lfloor u_1n\rfloor}(t), n^{1/2}\Gamma^n_{\lfloor u_2n\rfloor}(t))
\end{equation*}
converges in distribution as $n\to\infty$ to the $L^2$-limit
\begin{equation}\label{limiting}
\lim_{K\to \infty} \big(\mc Y_t (G^K_{u_1}) - \mc Y_0 (G^K_{u_1}), \int_0^t \mc Y_s(\iota_{1/K}^{u_2})\,\dd s\big)
\end{equation}
which is a Gaussian vector. We denote this limit by $\big(\mc Y_t (G_{u_1}) - \mc Y_0 (G_{u_1}), \int_0^t \mc Y_s(\delta_{u_2})\,\dd s\big)$  or by $(J_{u_1}(t), \Gamma_{u_2}(t))$.
The same argument can be adapted to finite times $t_1,\ldots, t_k$, which concludes the proof of convergence of the finite-dimensional distributions.

\subsection{Calculus of covariances}\label{covariances}
As promised above, let us determine the covariances of the limiting
process  \eqref{limiting}. This will be enough to conclude the proof of Theorem~\ref{thm2.6}.\medskip

\noindent \textbf{Two currents.}
Fix $0\leq s, t\leq T$ and $u_1\leq u_2\in\bb R$. To provide the covariance between two currents $J_{u_1}(s)$ and $J_{u_2}(t)$, we shall compute
\begin{equation}\label{eq41}
\begin{split}
&\bb E\big[ \big\{\mc Y_s(G_{u_1})-\mc Y_0(G_{u_1})\big\} \big\{\mc Y_t(G_{u_2})-\mc Y_0(G_{u_2})\big\} \big]\\
&=\lim_{K\to\infty}\bb E\big[ \big\{\mc Y_s(G^K_{u_1})-\mc Y_0(G^K_{u_1})\big\} \big\{\mc  Y_t(G^K_{u_2})-\mc Y_0(G^K_{u_2})\big\} \big]\\
&= \lim_{K\to\infty}\int_{\bb R}\dd u \,\Big\{ (T_s G^K_{u_1})(T_t G^K_{u_2})-(T_s G^K_{u_1})G^K_{u_2} -G^K_{u_1}( T_t G^K_{u_2}) +, G^K_{u_1}G^K_{u_2} \Big\} \mathcal X_0\\
&\quad +  \lim_{K\to\infty}2\int_0^s \dd r \int_{\bb R}\dd u \,(\nabla T_{s-r} G^K_{u_1})(\nabla T_{t-r} G^K_{u_2})\mathcal X_r
\end{split}
\end{equation}
where we recall that $\mathcal X_r = \mathcal X(\rho(r,u))$ and $\{T_t:t\geq 0\}$ stands for the semigroup associated to the Laplacian. For the second equality, we applied equation~\eqref{cov_density}.
The first term on the right hand side of the last equality in \eqref{eq41} can be rewritten as
\begin{equation}\label{t1}
\lim_{K\to\infty}\int_{\bb R}\dd u\;\big\{T_s G^K_{u_1}-G^K_{u_1}\big\} \big\{T_t G^K_{u_2}-G^K_{u_2}\big\} \mathcal X_0\,.
\end{equation}
We are going splitting this integral into the domains of integration  $(-\infty,u_1)$, $[u_1, u_2)$, $[u_2,u_1+K]$, $(u_1+K,u_2+K]$ and $(u_2+K,\infty)$  where we can assume that $K>u_2-u_1$.

By definition of $G_{u_i}^K$ with $i=1$ or $2$, for any $t\in[0,T]$, the function  $T_t G^K_{u_i}(u)$ is absolutely bounded by $\bb P_{u}[B_t\geq u_i]=\bb P_0[B_t\geq u_i-u]$ and converges to it as $K\to\infty$. Moreover, the former probability is integrable on $(-\infty, u_i)$ and since $u_1\leq u_2$, we have that $G^K_{u_i}(u)=0$ for all $u<u_1$. Thus, by the Dominated Convergence Theorem,
\begin{equation*}
\begin{split}
\lim_{K\to\infty}&\int_{-\infty}^{u_1}\dd u\,\big\{T_s G^K_{u_1}-G^K_{u_1}\big\} \big\{T_t G^K_{u_2}-G^K_{u_2}\big\} \mathcal X_0
= \int_{-\infty}^{u_1}\dd u\,\bb P_0[B_s\geq u_1-u]\bb P_0[B_t\geq u_2-u] \mathcal X_0\,.
\end{split}
\end{equation*}

We claim that the integral in \eqref{t1} restricted to $(u_2+K ,\infty)$ vanishes as $K\to\infty$. Indeed, in this interval $T_t G^K_{u_i}- G^K_{u_i}$ vanishes pointwisely as $K\to\infty$ and is bounded by
\begin{equation*}
\frac{1}{\sqrt{K}} \bb P_0(B_t\leq u_2-u+K) + \bbP_0(B_t\leq u_2-u+K-\sqrt{K})\,.
\end{equation*}
The claim then follows by standard Gaussian tail bounds.

On the interval $[u_1,u_2)$, we have that $T_s G^K_{u_1}-G^K_{u_1}$ is equal to
$$-\bb E_{u}[\one\{B_s\leq u_1\}\big(1-(B_s-u_1)/K\big)]-\bb E_{u}[\one\{B_s\geq u_1+K\}\big(1-(B_s-u_1)/K\big)] $$
which is absolutely bounded by $2$ and converges pointwisely to
$-\bb P_{u}[B_s\leq u_1]=-\bb P_0[B_s\leq u_1-u]$
as $K\to\infty$. It is easy to see that $T_t G^K_{u_2}-G^K_{u_2}$ converges pointwisely to $\bb P_0[B_t\geq u_2-u]$ on the interval $[u_1,u_2]$ as $K\to\infty$. Therefore we have
\begin{equation*}
\lim_{K\to\infty}\int_{u_1}^{u_2}\!\!\!\dd u\big\{T_s G^K_{u_1}-G^K_{u_1}\big\} \big\{T_t G^K_{u_2}-G^K_{u_2}\big\} \mathcal X_0
=-\int_{u_1}^{u_2}\!\!\! \dd u\, \bb P_0\big[B_s\leq u_1-u\big]\bb P_0\big[B_t\geq u_2-u\big] \mathcal X_0\,.
\end{equation*}
On the interval $[u_2,u_1+K]$, $T_t G^K_{u_2}-G^K_{u_2}$ is equal to
$$-\bb E_{u}[\one\{B_t\leq u_2\}\big(1-(B_t-u_2)/K\big)]-\bb E_{u}[\one\{B_t\geq u_2+K\}\big(1-(B_t-u_2)/K\big)]\,.$$
By the Cauchy-Schwarz inequality, $\bb E_{u}[\one\{B_t\leq u_2\}(B_t-u_2)/K]$ and $\bb E_{u}[\one\{B_t\geq u_2+K\}\big(1-(B_t-u_2)/K\big)]$ vanishes in $L^2([u_2,u_1+K])$ as $K\to\infty$. The same conclusion also holds for $T_s G^K_{u_1}-G^K_{u_1}$. Therefore, we conclude that the integral in \eqref{t1} over $[u_2,u_1+K]$ is equal to a negligible term in $K$ plus
$$\int_{u_2}^{u_1+K} \dd u\;\bb P_u[B_s\leq u_1]\bb P_u[B_t\leq u_2] \mathcal X_0\;\to\; \int_{u_2}^\infty \dd u\; \bb P_0\big[B_s\leq u_1-u\big]\bb P_0\big[B_t\leq u_2-u\big] \mathcal X_0$$
as $K$ goes to infinity.
Finally, by translation invariance, the integral in \eqref{t1} restricted to the interval $(u_1+K,u_2+K]$ is equal to
$$\int_{u_1}^{u_2} \dd u\;\big\{T_s G^K_{u_1-K}-G^K_{u_1-K}\big\} \big\{T_t G_K^{u_2-K}-G^K_{u_2-K}\big\} \mathcal X(\rho_0(u+K))\,.$$
Through the explicit expression we have that both $T_s G^K_{u_1-K}-G^K_{u_1-K}$ and $T_t G^K_{u_2-K}-G^K_{u_2-K}$  are bounded absolutely by $2$ and vanish pointwisely as $K\to\infty$ on $[u_1,u_2]$. Therefore by Dominated Convergence Theorem, the integral in \eqref{t1} over $(u_1+K,u_2+K]$ vanishes as $K\to\infty$.
In summary, we have shown that the expression in \eqref{t1} is equal to
\begin{align*}
&\int_{-\infty}^{u_1}  \dd u\;\bb P_0\big[B_s\geq u_1-u\big]\bb P_0\big[B_t\geq u_2-u\big] \mathcal X_0\\
&-\int_{u_1}^{u_2} \dd u\; \bb P_0\big[B_s\leq u_1-u\big]\bb P_0\big[B_t\geq u_2-u\big] \mathcal X_0\\
&+\int_{u_2}^\infty \dd u\; \bb P_0\big[B_s\leq u_1-u\big]\bb P_0\big[B_t\leq u_2-u\big] \mathcal X_0\,.
\end{align*}
To estimate the rightmost term in \eqref{eq41}, note that	$$\nabla T_t G^K_{u_i}(u)\;=\;p_t(u,u_i)+ \frac{1}{K}\int_{u_i}^{K+u_i} \dd v\,p_t(u,v)\,.$$
By the Cauchy-Schwarz inequality, the second term vanishes in $L^2(\bb R)$. From this we see that the second term at the right hand side of the last equality of \eqref{eq41} converges to
\begin{equation*}
\begin{split}
2\int_0^s \dd r \int_{\bb R}\dd u\; p_{s-r}(u,u_1)p_{t-r}(u,u_2) \mathcal X_r\,.
\end{split}
\end{equation*}

\noindent\textbf{Two occupation times.}
For easy of notation, let $\eps = 1/K$. In order to compute the covariance between two occupation times $\Gamma_{u_1}(s)$ and $\Gamma_{u_2}(t)$, we shall estimate, as $\varepsilon\to 0$, the
limit of
\begin{equation}\label{Taucov}
\begin{split}
&\bb E\big[ \int_0^s \dd {r_1}\; \mc  Y_{r_1}(\iota^{\varepsilon}_{u_1})  \int_0^t \dd r_2\; \mc  Y_{r_2}(\iota^\varepsilon_{u_2}) \big]
\;=\; \int_0^s \dd r_1\int_0^t \dd r_2\;\bb E\big[\mc  Y_{r_1}(\iota^\varepsilon_{u_1}) \mc  Y_{r_2}(\iota^\varepsilon_{u_2})\big]\,.
\end{split}
\end{equation}
By \cite[Theorem 2.2]{jaralandim2006},  if $0\leq r_2< r_1$, then
\begin{equation}\label{otsum}
\begin{split}
\bb E\big[\mc  Y_{r_1}(\iota^\varepsilon_{u_1}) \mc Y_{r_2}(\iota^\varepsilon_{u_2})\big]\,=\, &\int_{\bb R}\dd u\; T_{r_1-r_2}(\iota^\varepsilon_{u_1}) \, \iota^\varepsilon_{u_2}\, \mathcal X_{r_2}\\
&+2\int_0^{ r_2}\dd\tau \int_{\bb R}\dd u\; (T_{r_1-\tau} \iota^\varepsilon_{u_1})( T_{r_2-\tau} \iota^\varepsilon_{u_2})\{\partial_\tau \mathcal X_\tau -\Delta \mathcal X_\tau \}\,.
\end{split}
\end{equation}

By the  expression
$T_t(\iota^\varepsilon_{u_i})(u)\,=\, \frac{1}{\varepsilon}\int_{u_i}^{u_i+\varepsilon} p_t(u,v)dv$, with $i=1,2$, we have that
\begin{equation*}
\begin{split}
\lim_{\varepsilon\to 0} \int_{\bb R}\dd u\; T_{r_1-r_2}(\iota^\varepsilon_{u_1}) \, \iota^\varepsilon_{u_2}\, \mathcal X_{r_2}\;
=\; &\lim_{\varepsilon\to 0} \frac{1}{\varepsilon^2}\int_{u_2}^{u_2+\varepsilon}\dd u\; \int_{u_1}^{u_1+\varepsilon} p_{r_1-r_2}(u,v)dv\, \mathcal X_{r_2}\\
=\;& p_{r_1-r_2}(u_1,u_2) \mathcal X(\rho_{r_2}(u_2)).
\end{split}
\end{equation*}
Obviously the first term on the right hand side of \eqref{otsum} is simply bounded absolutely by $\frac{1}{\sqrt{4\pi(r_1-r_2)}}$ for all $0<\varepsilon\leq 1$. An easy computation shows that
\begin{equation*}
\int_0^s \dd r_1 \int_0^{t\wedge r_1}\dd r_2\, \frac{1}{\sqrt{4\pi(r_1-r_2)}} \;<\;+\infty\,.
\end{equation*}
Therefore, by the Dominated Convergence Theorem, we have
\begin{equation*}
\begin{split}
\lim_{\varepsilon\to 0} &\int_0^s \dd r_1 \int_0^{t\wedge r_1}\dd r_2\, \int_{\bb R}\dd u\; T_{r_1-r_2}(\iota^\varepsilon_{u_1}) \, \iota^\varepsilon_{u_2}\, \mathcal X_{r_2}=  \int_0^s \dd r_1 \int_0^{t\wedge r_1}\!\!\!\dd r_2\, p_{r_1-r_2}(u_1,u_2) \mathcal X(\rho_{r_2}(u_2)) .
\end{split}
\end{equation*}
By symmetry of the Gaussian kernel $p_t(\cdot,\cdot)$, we can write
$$T_t(\iota^\varepsilon_{u_i})(u)\;=\; \frac{1}{\varepsilon}\int^{u}_{u-\varepsilon} \dd v\,p_t(v,u_i), \quad i=1,2$$
which converges to $p_t(u,u_i)$ for every $u\in\bb R$.
For any $u_i\in\bb R$ and any $t\in[0,T]$, let us consider the Hardy-Littlewood maximal function $f_{t,u_i}:\bb R\to \bb R$ defined by
\begin{equation*}
f_{t,u_i}(u)\,:=\, \sup_{0<\varepsilon\leq 1} \frac{1}{\varepsilon}\int^{u}_{u-\varepsilon} \dd v\,p_t(v,u_i)\, .
\end{equation*}
Note that
$$M_i\;:=\;\int_0^s \dd r_1\int_0^{t\wedge r_1} \dd r_2 \int_0^{r_2} d\tau\int_{\bb R}\dd u\; p_{r_i-\tau}(u,u_i)^2\{\partial_\tau \mathcal X_\tau -\Delta \mathcal X_\tau \}\;<\;\infty\,, \quad \,\, i=1,2.$$
By the Hardy-Littlewood Maximal Inequality (see \cite[Theorem 2.5, page 31]{Duoandikoetxea} for instance), there exists a constant $C>0$ such that for all $0<\varepsilon\leq1$,
\begin{equation*}
\int_0^s \dd r_1\int_0^{t\wedge r_1} \dd r_2 \int_0^{r_2} \dd\tau\int_{\bb R}\dd u\; f_{r_i-\tau, u_i}(u)^2\{\partial_\tau \mathcal X_\tau -\Delta \mathcal X_\tau \}\;\leq\;CM_i\,.
\end{equation*}
Moreover, obviously the non-negative function $T_t \iota_{u_i}^\varepsilon$ is bounded by $f_{t,u_i}$ for any $\varepsilon\in(0,\varepsilon_0]$. Therefore, again by the Dominated Convergence Theorem
we can pass the limit inside the integral for the second term at the right hand side of \eqref{otsum}.

The case that $0\leq r_1<r_2$ can be handled similarly. So we can conclude that  the limit of the expression in \eqref{Taucov} is equal to
\begin{equation*}
\begin{split}
&\int_0^s \dd r_1\int_0^t  \dd r_2 \,\,p_{\lvert r_1-r_2\rvert}(u_1,u_2) \mathcal X(\rho_{r_2}(u_2))\\
&+ 2\int_0^s \dd r_1\int_0^t  \dd r_2\int_0^{r_1\wedge r_2}\dd\tau \int_{\bb R}\dd u\; p_{r_1-\tau}(u,u_1) p_{r_2-\tau}(u,u_2) \{\partial_\tau \mathcal X_\tau -\Delta \mathcal X_\tau \} \,.
\end{split}
\end{equation*}

\noindent\textbf{Current and occupation time.}  In this case to obtain the covariance  between $J_{u_1}(s)$ and $\Gamma_{u_2}(t)$,  we need to estimate the limit, as $\varepsilon\to 0$, of
\begin{equation}\label{jointcov}
\begin{split}
&\bb E\Big[ \big\{\mc Y_s(G_{u_1})-\mc Y_0(G_{u_1})\big\}  \int_0^t \dd r\, \mc Y_{r}(\iota^\varepsilon_{u_2}) \Big]
= \int_0^t \dd r\,\bb E\Big[\big\{\mc Y_s(G_{u_1})-\mc Y_0(G_{u_1})\big\} \mc Y_{r}(\iota^\varepsilon_{u_2})\Big]
\end{split}
\end{equation}
for any $0\leq \,s,t\leq T$. Recall that we have proved that $\mc Y_s(G^K_{u_1})-\mc Y_0(G^K_{u_1})$ converges in $L^2$ to $\mc Y_s(G_{u_1})-\mc Y_0(G_{u_1})$, thus the term on the right hand side of \eqref{jointcov} is equal to
\begin{equation}\label{eq:curocculimK}
\int_0^t \dd r\, \lim_{K\to\infty} \bb E\Big[\big\{\mc Y_s(G^K_{u_1})-\mc Y_0(G^K_{u_1})\big\} \mc Y_{r}(\iota^\varepsilon_{u_2})\Big].
\end{equation}

We first deal with the case that $0\leq s<r$. In view of Theorem 2.2  in \cite{jaralandim2006},  we have
\begin{equation}\label{jointsum}
\begin{split}
\bb E\Big[ \mc Y_s(G_{u_1}^K)\mc Y_r(\iota^\varepsilon_{u_2})\Big]\,=\,&
\int_{\bb R}\! \dd u\;G^K_{u_1} (T_{r-s} \iota^\varepsilon_{u_2}) \mathcal X_s \\
+\, &2 \int_0^{s}\!\!\!\! \dd\tau \int_{\bb R}\!\dd u \;(T_{s-\tau}G^K_{u_1}) ( T_{r-\tau} \iota^\varepsilon_{u_2} ) \{\partial_\tau \mathcal X_\tau -\Delta \mathcal X_\tau \}\,.
\end{split}
\end{equation}
By the Monotone Convergence Theorem,
$$\lim_{K\to\infty} \int_{\bb R}\! \dd u\;G^K_{u_1} (T_{r-s} \iota^\varepsilon_{u_2}) \mathcal X_s\,=\,\int_{u_1}^\infty\! \dd u\; (T_{r-s} \iota^\varepsilon_{u_2}) \mathcal X_s\,.$$
We showed in the previous computations that the non-negative function $T_s G^K_{u_1}$ is bounded by and converges to $\bb P_{u}[B_s\geq u_1]=\bb P_0[B_s\geq u_1-u]$. In addition, it can be shown by a direct computation that
\begin{equation*}
\int_0^{s} \dd\tau \int_{\bb R}\!\dd u \;  (T_{r-\tau} \iota^\varepsilon_{u_2}) \,  \{\partial_\tau \mathcal X_\tau -\Delta \mathcal X_\tau \}\;<\;+\infty\,.
\end{equation*}
Therefore by Dominated Convergence Theorem we have that
\begin{equation*}
\begin{split}
&\lim_{K\to\infty}\int_0^{s} \dd\tau \int_{\bb R}\!\dd u \;(T_{s-\tau}G^K_{u_1}) ( T_{r-\tau} \iota^\varepsilon_{u_2} ) \{\partial_\tau \mathcal X_\tau -\Delta \mathcal X_\tau \}\\
&=\;\int_0^{s} \dd\tau \int_{\bb R}\!\dd u \;\bb P_0[B_s\geq u_1-u] ( T_{r-\tau} \iota^\varepsilon_{u_2} ) \{\partial_\tau \mathcal X_\tau -\Delta \mathcal X_\tau \}\,.
\end{split}
\end{equation*}

We now consider the case that $0\leq r<s$. In this case by  \cite[Theorem 2.2 ]{jaralandim2006},  we have
\begin{equation}\label{jointsum2}
\begin{split}
\bb E\Big[\mc  Y_s(G_{u_1}^K)\mc  Y_r(\iota^\varepsilon_{u_2})\Big]\,=\,&
\int_{\bb R}\! \dd u\;(T_{s-r} G^K_{u_1}) \, \iota^\varepsilon_{u_2} \, \mathcal X_s \\
&+2 \int_0^{ r} \dd\tau \int_{\bb R}\dd u \;(T_{s-\tau}G^K_{u_1}) ( T_{r-\tau} \iota^\varepsilon_{u_2} ) \{\partial_\tau \mathcal X_\tau -\Delta \mathcal X_\tau \}\,.
\end{split}
\end{equation}
The second term on the right hand side can be treated as in the case $r>s$. The first term on the right hand side is equal to
$$\frac{1}{\varepsilon}\int_{u_2}^{u_2+\varepsilon}\dd u \, (T_{s-r} G^K_{u_1})  \, \mathcal X_s$$
which  by the Dominated Convergence Theorem converges to
$$\frac{1}{\varepsilon}\int_{u_2}^{u_2+\varepsilon}\dd u \, \bb P_0[B_{s-r}\geq u_1-u]  \, \mathcal X_s\,,$$
since  the non-negative function $T_s G^K_{u_1}$ is bounded by and converges to $\bb P_{u}[B_s\geq u_1]=\bb P_0[B_s\geq u_1-u]$, which is integrable on $[u_2,u_2+\varepsilon]$.
In summary we have shown that
\begin{equation}\label{klimit}
\begin{split}
& \lim_{K\to\infty} \bb E\Big[\big\{\mc Y_s(G^K_{u_1})-\mc Y_0(G^K_{u_1})\big\} \mc Y_{r}(\iota^\varepsilon_{u_2})\Big]\\
&= \int_{u_1}^\infty\! \dd u\; (T_{r-s} \iota^\varepsilon_{u_2}) \one\{r>s\} \mathcal X_s+ \frac{1}{\varepsilon}\int_{u_2}^{u_2+\varepsilon}\dd u \, \bb P_0[B_{s-r}\geq u_1-u] \one\{r<s\}  \, \mathcal X_s\\
& +\int_0^{s\wedge r}\!\!\!\! \dd\tau \int_{\bb R}\!\dd u \;\bb P_0[B_s\geq u_1-u] ( T_{r-\tau} \iota^\varepsilon_{u_2} ) \{\partial_\tau \mathcal X_\tau -\Delta \mathcal X_\tau \}- \int_{u_1}^\infty \dd u \,(T_r \iota_{u_2}^\varepsilon)\, \mc X_0\,.
\end{split}
\end{equation}
This finishes the analysis of the integrand in~\eqref{eq:curocculimK}. We now need to analyze its integral and the limit as $\eps\to 0$.
Note that $
\lim_{\varepsilon\to 0}\frac{1}{\varepsilon}\int_{u_2}^{u_2+\varepsilon}\dd u\,  \bb P_0[B_{s-r}\geq u_1-u] \one\{r<s\} \mathcal X_s$
converges to $\bb P_0[B_{s-r}\geq u_1-u_2] \one\{r<s\} \mathcal X(\rho_s(u_2)) $ and is absolutely bounded by $\frac{1}{\sqrt{4\pi(s-r)}}\one\{r<s\}$. Since
$$\int_0^t \dd r\, \frac{1}{\sqrt{4\pi(s-r)}}\one\{r<s\}\;<\;+\infty\,,$$
the Dominated Convergence Theorem implies that
\begin{align*}
&\lim_{\varepsilon\to 0}  \int_0^t \dd r\, \frac{1}{\varepsilon}\int_{u_2}^{u_2+\varepsilon}\dd u \, \bb P_0[B_{s-r}\geq u_1-u] \one\{r<s\}  \, \mathcal X_s\\
&=\; \int_0^t \dd r\, \bb P_0[B_{s-r}\geq u_1-u_2] \one\{r<s\} \mathcal X(\rho_s(u_2))\,.
\end{align*}
To show that we can pass the limit $\eps\to 0$ inside the integral of the remaining three terms we can use similar arguments as for the correlation estimate of two occupation times. We will omit the details. Therefore, we conclude that the limit as $\varepsilon\to 0$ of the integral on $r$ over $[0,t]$ of the sum of the four terms on the right hand side of \eqref{klimit} is equal to
\begin{align*}
&\int_0^t \dd r\int_{u_1}^\infty\! \dd u\; p_{r-s}(u,u_2)\one\{r>s\} \mathcal X_s\\
&+ \int_0^t \dd r \, \bb P_0[B_{s-r}\geq u_1-u_2] \one\{r<s\} \mathcal X(\rho_s(u_2))\\
& +\int_0^t \!\dd r\int_0^{s\wedge r} \dd\tau \int_{\bb R}\dd u \;\bb P_0[B_s\geq u_1-u] p_{r-\tau}(u,u_2) \{\partial_\tau \mathcal X_\tau -\Delta \mathcal X_\tau \}\\
& - \int_0^t \dd r\int_{u_1}^\infty \dd u \,p_r(u,u_2)\, \mc X_0\,,
\end{align*}
concluding the argument.

\section*{Acknowledgements}
D.E.~was supported by the National Council for Scientific and Technological Development - CNPq via a Bolsa de Produtividade 303520/2019-1 and 303348/2022-4. D.E. and T.Xu moreover acknowledge support by the Serrapilheira Institute (Grant Number Serra-R-2011-37582). D.E and T.F. moreover acknowledge support by the National Council for Scientific and Technological Development - CNPq via a Universal Grant (Grant Number 406001/2021-9). T.F was supported by the National Council for Scientific and Technological Development - CNPq via a Bolsa de Produtividade number 311894/2021-6.

\bibliographystyle{plain}
\bibliography{bibliography}

\end{document}